\documentclass[12pt]{amsart}

\usepackage{amsmath, amssymb, amsthm}
\usepackage[colorlinks=true]{hyperref}
\usepackage[alphabetic]{amsrefs}
\usepackage{verbatim}
\usepackage{xcolor}
\usepackage{url}
\usepackage{enumitem}
\usepackage{microtype}
\usepackage[pdftex]{graphicx}
\usepackage[margin=1in]{geometry}

\newtheorem{lemma}{Lemma}[section]
\newtheorem{theorem}[lemma]{Theorem}
\newtheorem{corollary}[lemma]{Corollary}
\newtheorem{proposition}[lemma]{Proposition}
\theoremstyle{definition}
\newtheorem{definition}[lemma]{Definition}

\newtheorem{claim}[lemma]{Claim}
\newtheorem{example*}[lemma]{Example}
\newtheorem{remark}[lemma]{Remark}

\theoremstyle{remark}

\newcommand{\defbold}{\em}

\makeatletter
\newtheorem*{rep@theorem}{\rep@title}
\newcommand{\newreptheorem}[2]{%
\newenvironment{rep#1}[1]{%
\def\rep@title{{\bf #2 \ref{##1}}}%
\begin{rep@theorem}}%
{\end{rep@theorem}}}
\makeatother
\newreptheorem{theorem}{Theorem}

\DeclareRobustCommand{\qedify}[1]{%
  \ifmmode \quad\hbox{#1}
  \else
    \leavevmode\unskip\penalty9999 \hbox{}\nobreak\hfill
    \quad\hbox{#1}%
  \fi
}
\newenvironment{example}{\begin{example*}\pushQED{\qedify{$\diamondsuit$}}}{\popQED\end{example*}}

\numberwithin{equation}{section}

\newcommand{\K}{\ensuremath{\Bbbk}} 

\newcommand{\Rbar}{\ensuremath{\overline{\mathbb R}}}

\newcommand{\C}{\mathcal{C}}

\newcommand{\F}{\mathcal{F}}
\newcommand{\G}{\mathcal{G}}

\newcommand{\ZZ}{\mathbb{Z}}
\newcommand{\QQ}{\mathbb{Q}}

\newcommand{\NN}{\mathbb{N}}
\newcommand{\RR}{\mathbb{R}}
\newcommand{\BB}{\mathbb{B}}

\newcommand{\bfx}{\mathbf{x}}
\newcommand{\bfy}{\mathbf{y}}

\newcommand{\bfu}{\mathbf{u}}
\newcommand{\bfv}{\mathbf{v}}
\newcommand{\bfw}{\mathbf{w}}

\newcommand{\bfe}{\mathbf{e}}
\newcommand{\bfzero}{\mathbf{0}}

\DeclareMathOperator{\inn}{in}

\DeclareMathOperator{\rank}{rank}

\DeclareMathOperator{\val}{val}

\DeclareMathOperator{\trop}{trop}
\DeclareMathOperator{\supp}{supp}

\DeclareMathOperator{\spann}{span}

\DeclareMathOperator{\cone}{cone}
\DeclareMathOperator{\mult}{mult}
\DeclareMathOperator{\GL}{GL}
\DeclareMathOperator{\Mat}{Mat}
\DeclareMathOperator{\uMat}{\underline{Mat}}

\DeclareMathOperator{\mon}{Mon}
\DeclareMathOperator{\relint}{relint}
\DeclareMathOperator{\starr}{star}
\DeclareMathOperator{\rk}{rk}
\DeclareMathOperator{\lc}{lc}

\DeclareMathOperator{\elimination}{elim}

\DeclareMathOperator{\Cox}{Cox}

\newcommand{\elim}[4]{{#1} \leftarrow \elimination_{#2} ({#3},{#4})}

\makeatletter
\newcommand{\superimpose}[2]{{\ooalign{$#1\@firstoftwo#2$\cr\hfil$#1\@secondoftwo#2$\hfil\cr}}}
\makeatother
\newcommand{\ttimes}{\hspace{0.3mm}{\mathpalette\superimpose{{\circ}{\cdot}}}\hspace{0.3mm}}
\newcommand{\tplus}{\mathrel{\oplus}}

\begin{document}

\title{Varieties of Tropical Ideals are Balanced}

\author{Diane Maclagan}
\address{Mathematics Institute, University of Warwick, Coventry CV4 7AL, United 
Kingdom.}
\email{D.Maclagan@warwick.ac.uk}

\author{Felipe Rinc\'on} \address{School of Mathematical Sciences,
  Queen Mary University of London, London E1 4NS, United Kingdom.}
\email{f.rincon@qmul.ac.uk}

\begin{abstract}
Tropical ideals, introduced in \cite{TropicalIdeals}, define
subschemes of tropical toric varieties.  We prove that the
top-dimensional parts of their varieties are balanced polyhedral
complexes of the same dimension as the ideal.  This means that every
subscheme of a tropical toric variety defined by a tropical ideal has
an associated class in the Chow ring of the toric variety.   A key tool in the proof is that specialization of variables
in a tropical ideal yields another tropical ideal; this plays the role
of hyperplane sections in the theory.  We also show that elimination
theory (projection of varieties) works for tropical ideals
as in the classical case.  The
matroid condition that defines tropical ideals is crucial for these results.
\end{abstract}

\maketitle


\section{Introduction}

This paper is part of a program to develop an intrinsic scheme theory
for tropical geometry, begun in \cite{Giansiracusa2}; see \cites{
  DraismaRincon, Lorscheid, GG2, GGGrassmann, JooMincheva,
  JooMincheva2, MaclaganRincon1, MacPhersonThesis, Silversmith}.
We focus on subschemes of tropical
toric varieties.
Usual subschemes of affine
or projective space are defined by ideals in a polynomial ring.
Ideals in the semiring of tropical polynomials, however, are too
general with which to build a theory analogous to classical algebraic geometry.
For example, the variety of an ideal in this semiring is not necessarily a
finite polyhedral complex; see \cite{TropicalIdeals}*{Example 5.14}.

The
remedy proposed in \cite{TropicalIdeals} is to work with a smaller
class of ideals, called {\defbold tropical ideals} (see
Definition~\ref{d:tropicalideal} below).  The class of tropical ideals
includes all tropicalizations of classical ideals \cite{Giansiracusa2}, but it is strictly larger.
In \cite{TropicalIdeals},
the authors show that the variety of a tropical ideal is a finite
$\mathbb R$-rational polyhedral complex, and that tropical ideals
satisfy the ascending chain condition and the weak Nullstellensatz.  In
addition, homogeneous tropical ideals have a Hilbert polynomial, which
in particular allows a definition of dimension and degree.  This suggests that tropical ideals
form a reasonable class with which to work for tropical algebraic
geometry.

The main result of this paper is that the variety of a tropical ideal
is {\defbold balanced} with respect to an intrinsically defined
multiplicity on its maximal cells.  This generalizes the Structure
Theorem for tropicalizations of classical varieties.  The balancing
condition is a combinatorial constraint on a polyhedral complex that
can be interpreted as a ``zero-tension'' condition; see
Definition~\ref{d:balancing} for a precise definition.  It plays a
fundamental role  in tropical geometry.  Along the
way we prove other basic results for tropical ideals, including that
they are closed under specializations of the variables, and that the
dimension of their varieties agrees with the dimension of the defining
ideal.

We now state this more precisely, beginning with the definition of a
tropical ideal.  We write $\Rbar = (\mathbb R \cup \{\infty \} ,
\tplus ,\ttimes)$ for the tropical semiring, where $\tplus$ is $\min$
and $\ttimes$ is regular addition. For simplicity, we restrict 
our presentation in the introduction to ideals in the Laurent polynomial semiring.

\begin{definition}\label{d:tropicalideal}
An ideal $I \subseteq \Rbar[x_1^{\pm 1},\dots,x_n^{\pm 1}]$ is a 
{\defbold tropical ideal} if it satisfies the following
``monomial elimination axiom'': 

\parbox[center][][c]{0.95\textwidth}
{
$(\dag)$ For any $f, g \in I$ and any monomial 
$\bfx^\bfu$ for which $[f]_{\bfx^\bfu} = [g]_{\bfx^\bfu} \neq \infty$,
there exists $h \in I$ such that
$[h]_{\bfx^\bfu} = \infty$ and $[h]_{\bfx^\bfv} 
\geq \min( [f]_{\bfx^\bfv}, [g]_{\bfx^\bfv} )$ 
for all monomials $\bfx^\bfv$, with the equality holding
whenever $[f]_{\bfx^\bfv} \neq [g]_{\bfx^\bfv}$.
}

\noindent Here, we use the notation $[f]_{\bfx^\bfu}$ to denote the coefficient
of the monomial $\bfx^\bfu$ in the polynomial $f$.
As we  refer to this condition on $f$, $g$, and $h$ several times throughout the paper, 
we abbreviate it as
\[\elim{h}{\bfx^\bfu}{f}{g}.\]
We refer to $h$ as {\em an} elimination of $\mathbf{x}^{\mathbf{u}}$
from $f$ and $g$, as it is not uniquely defined by
$\mathbf{x}^{\mathbf{u}}$, $f$, and $g$.
\end{definition}

When $X$ is a $d$-dimensional irreducible subvariety of the torus
$(K^*)^n$ over a valued field $K$, the tropicalization $\trop(X)$ is the support of a pure
$d$-dimensional balanced polyhedral complex.  The following is the
main theorem of the paper, which generalizes this fact to varieties of
tropical ideals; see \S\ref{ss:varietiestropicalideals} for the definition.

\begin{theorem} \label{t:balancing}
Let $I \subseteq \Rbar[x_1^{ \pm 1},\dots,x_n^{\pm 1}]$ be a tropical
ideal of dimension $d$. Then the variety $V(I)$ is the support of a polyhedral complex $\Sigma$
whose maximal cells are $d$-dimensional.
Moreover, the weighted $\mathbb R$-rational polyhedral complex
$\Sigma^d$ consisting of the $d$-dimensional cells of $\Sigma$, with weights given by the
multiplicities of Definition~\ref{d:multiplicitydefn}, is balanced.
\end{theorem}	

The restriction to the $d$-skeleton of $\Sigma$ is necessary here, as
there is not yet a reasonable definition of an irreducible subscheme
of the tropical torus that implies that $V(I)$ is pure.
Theorem~\ref{t:balancing} allows us to define a Hilbert-Chow morphism for tropical ideals,
taking a subscheme of a tropical toric variety to its class in the
Chow ring of the toric variety; see Remark~\ref{r:HilbertChow}.

A key tool used is the fact, whose proof is non-trivial,
that the class of tropical ideals is closed under specialization of some of the variables. 

\begin{theorem} \label{t:sliceintro}
Let $I \subseteq \Rbar[x_1^{\pm 1},\dots,x_n^{\pm 1}]$ be a tropical ideal.  
For any $a \in \RR$, the ideal
$$I|_{x_n = a} := \{ f(x_1,\dots,x_{n-1},a) : f \in I \} \subseteq
\Rbar[x_1^{\pm 1},\dots,x_{n-1}^{\pm 1}]$$ is a tropical ideal.
When $V(I)$ is the support of a pure polyhedral complex, 
we have
$$V(I|_{x_n=a}) = \pi(V(I) \cap_{st}
\{x_n =a \}),$$ where $\pi \colon \RR^n \rightarrow \RR^{n-1}$ is
the projection onto the first $n-1$ coordinates, and $\cap_{st}$ is the stable intersection.
\end{theorem}

This theorem plays the role of a hyperplane section in tropical scheme
theory, as it allows induction on dimension.  In the realizable case,
$V(I|_{x_n=a})$ is the tropicalization of the intersection of the
variety with a generic translate of a subtorus; see
Remark~\ref{r:realizableslicing}.

Two important consequences of
Theorem~\ref{t:sliceintro}, which were already part of the standard
tropical tool-kit in the case that $I$ is the tropicalization of a
classical ideal, are the following.
\begin{enumerate}
\item (Theorem~\ref{t:dimiscorrect}). The variety of a tropical ideal of dimension $d$ is the support
  of a polyhedral complex with maximal cells {\defbold of dimension $d$}.

\item (Theorem~\ref{t:projection}).  If $I \subseteq \Rbar[x_1^{\pm 1},\dots,x_n^{\pm 1}]$ is a
  tropical ideal, and $\pi \colon \mathbb R^n \rightarrow \mathbb
  R^{n-1}$ is the projection onto the first $n-1$ coordinates, then 
$$V(I \cap \Rbar[x_1^{\pm 1},\dots,x_{n-1}^{\pm 1}]) = \pi(V(I)).$$
\end{enumerate}

We note that this differs slightly from the non-tropical case, 
where a closure is needed on $\pi(V(I))$.  The proof of the projection
result uses a Nullstellensatz of Grigoriev and Podolskii
\cite{GrigorievPodolskii} in a crucial fashion.  

The last
important ingredient in the proof of Theorem~\ref{t:balancing} is the
fact (Theorem~\ref{t:degissummult}) that the degree of a
zero-dimensional ideal is the sum of the multiplicities of the points
in its variety.  The proof here is more complicated than in the
classical case, owing to the lack of primary decomposition (so far) in
tropical scheme theory.

The structure of the paper is as follows.  In Section~\ref{s:Groebner}
we develop more Gr\"obner theory for tropical ideals.  The proofs of these results are fairly similar to the realizable case.  
Section~\ref{s:specialization} contains the first deep result, 
with the proof of the key
specialization theorem (Theorem~\ref{t:restrictiontheorem}).  The
results about dimension and projection (Theorem~\ref{t:dimiscorrect}
and Theorem~\ref{t:projection}) are proved in
Section~\ref{d:dimensionprojection}, while Section~\ref{s:degreezero}
contains the key facts about degrees of zero-dimensional ideals.
Finally, Theorem~\ref{t:balancing} is proved in
Section~\ref{s:balancing} (Theorem \ref{t:balanced}).

\noindent {\bf Acknowledgments}. DM was partially supported by EPSRC
grant EP/R02300X/1.  FR was partially supported by the Research
Council of Norway grant 239968/F20.  We thank Alex Fink for useful
conversations about Theorem~\ref{t:degissummult}, and the Institut
Mittag-Leffler for hosting those conversations during the program on
{\em Tropical Geometry, Amoebas, and Polytopes}.


\section{Gr\"obner theory} \label{s:Groebner}

In this section we prove basic results about initial ideals of tropical ideals
and the connection with initial ideals with respect to monomial term orders.

\subsection{Variants of tropical ideals}

Throughout this paper we will consider tropical ideals in both
$\Rbar[x_1,\dots,x_n]$ and $\Rbar[x_1^{\pm 1},\dots,x_n^{\pm 1}]$, and
also homogeneous tropical ideals in $\Rbar[x_0,\dots,x_n]$.  A slightly more general setting also occurs in \S~\ref{s:degreezero}.
In each case the definition of tropical ideal is that it obeys
the monomial elimination axiom $(\dag)$ given in
Definition~\ref{d:tropicalideal}.  In the case of homogeneous tropical
ideals in $\Rbar[x_0,\dots,x_n]$ it suffices to check the condition
when $f$ and $g$ are homogeneous.  Equivalently, we require that for
any finite selection $E$ of monomials (which can be Laurent in the
case $I \subseteq \Rbar[x_1^{\pm 1},\dots,x_n^{\pm 1}]$), the
restriction $I|_E$ is the set of vectors of a {\defbold valuated
  matroid} $\Mat(I|_E)$ on the ground set $E$.  
We denote by $\uMat(I|_E)$ the underlying matroid of $\Mat(I|_E)$, which
is a matroid on the set $E$.
  See \cite{TropicalIdeals}*{\S 2} for
more on this perspective.

We will also consider tropical ideals where the semiring of
coefficients is the {\defbold Boolean semiring} $\BB
:= \{0, \infty\} \subseteq \Rbar$.  Many results also hold for more
general additively idempotent semifields, as in \cite{TropicalIdeals};
we restrict to $\mathbb B$ and $\mathbb R$ here as the main focus is
on the polyhedral structure of varieties.

We now describe the connection between these three versions of tropical ideals.  
The {\defbold homogenization} of a tropical polynomial $f = \bigoplus c_\bfu
 \bfx^{\mathbf{u}} \in \Rbar[x_1,\dots,x_n]$ is 
$$\tilde{f} = \textstyle \bigoplus
c_{\mathbf{u}} x_0^{d-|\mathbf{u}|} \bfx^{\mathbf{u}} \in \Rbar[x_0,x_1,\dots,x_n] $$
 where 
$\textstyle |\bfu| := \sum_{i=1}^n u_i$ and $d =
  \max(|\mathbf{u}| : c_{\mathbf{u}} \neq \infty )$.  The
  homogenization of an ideal $I \subseteq
  \Rbar[x_1,\dots,x_n]$ is the ideal 
$$I^h := \langle \tilde{f} : f \in I \rangle \subseteq
  \Rbar[x_0,x_1,\dots,x_n].$$  
Conversely, if $f \in \Rbar[x_0,\dots,x_n]$ is a homogeneous polynomial,
its {\defbold dehomogenization} is the polynomial
$f|_{x_0=0} := f(0,x_1, \dots,x_n) \in \Rbar[x_1,\dots,x_n]$. The 
dehomogenization of a homogeneous ideal $J \subseteq \Rbar[x_0,x_1,\dots,x_n]$
is the ideal
$$J|_{x_0=0} := \{ f|_{x_0=0} : f \in J \} \subseteq \Rbar[x_1,\dots,x_n].$$

If $J \subseteq \Rbar[x_1^{\pm 1},\dots,x_n^{\pm 1}]$ is an ideal in the 
Laurent polynomial semiring, the intersection $J \cap \Rbar[x_1,\dots,x_n]$
is an ideal in the affine polynomial semiring $\Rbar[x_1,\dots,x_n]$. 
Conversely, any ideal $I \subseteq \Rbar[x_1,\dots,x_n]$
generates an ideal 
$$I\Rbar[x_1^{\pm 1},\dots,x_n^{\pm 1}] \subseteq \Rbar[x_1^{\pm 1},\dots,x_n^{\pm 1}].$$
Its elements are the Laurent polynomials of the form
$f \bfx^\bfu$ with $f \in I$ and $\bfx^\bfu$ a Laurent monomial.

If $I$ is an ideal in $\Rbar[x_1,\dots,x_n]$ and $m$ is a monomial,
then 
$$(I: m) := \{ f \in \Rbar[x_1,\dots,x_n] : f m \in I \}.$$  
The {\defbold saturation} of $I$ with respect to $m$ is 
$$(I:m^{\infty}) := \{ f \in
\Rbar[x_1,\dots,x_n] : f m^k \in I \text{ for some } k \geq 0 \}.$$

The following lemma details the relationships between these ideals.

\begin{lemma}\label{l:basicproperties}
\begin{enumerate}
\item \label{enum:torustoaffine} If $J \subseteq \Rbar[x_1^{\pm
    1},\dots,x_n^{\pm 1}]$ is a tropical ideal then $J \cap
  \Rbar[x_1,\dots,x_n]$ is a tropical ideal as well.  Conversely,
  if $I \subseteq \Rbar[x_1,\dots,x_n]$ is a tropical ideal then so is
  $I\Rbar[x_1^{\pm 1},\dots,x_n^{\pm 1}] \subseteq \Rbar[x_1^{\pm 1},\dots,x_n^{\pm 1}]$.  
  
\item \label{enum:affinetoproj}  If $I \subseteq \Rbar[x_1,\dots,x_n]$ is a tropical ideal then $I^h \subseteq \Rbar[x_0,\dots,x_n]$ is a tropical ideal.  Conversely, if $J \subseteq \Rbar[x_0,\dots,x_n]$ is a homogeneous tropical ideal, then $J|_{x_0=0} \subseteq \Rbar[x_1,\dots,x_n]$ is a tropical ideal.

\item If $I \subseteq \Rbar[x_1,\dots,x_n]$ is a tropical ideal and $m$ is any monomial, then
  $(I:m)$ and $(I:m^{\infty})$ are also tropical ideals.  When
  $m=\prod_{i=1}^n x_i$, then $(I:m^{\infty}) = I\Rbar[x_1^{\pm
  1},\dots,x_n^{\pm 1}] \cap \Rbar[x_1,\dots,x_n]$.  In particular, if
  $J \subseteq \Rbar[x_1^{\pm 1},\dots,x_n^{\pm 1}]$ and $I =
  J \cap \Rbar[x_1,\dots,x_n]$, then $(I:m^{\infty}) = I$.
\end{enumerate}
\end{lemma}

\begin{proof}
\begin{enumerate}
\item Suppose $J \subseteq \Rbar[x_1^{\pm
    1},\dots,x_n^{\pm 1}]$ is a tropical ideal, and let 
    $I = J \cap \Rbar[x_1,\dots,x_n]$. If $f, g \in I$ with $[f]_{\mathbf{x}^{\mathbf{u}}} =
  [g]_{\mathbf{x}^{\mathbf{u}}}$ then $f,g \in J$, so there is $h \in
    J$ with $\elim{h}{\bfx^\bfu}{f}{g}$. 
            This in particular implies that $h \in I$, and so
            $I$ satisfies the monomial elimination axiom.

    Assume now that $I \subseteq \Rbar[x_1,\dots,x_n]$ is a tropical ideal,
    and let $J = I\Rbar[x_1^{\pm 1},\dots,x_n^{\pm 1}]$.
    Fix $f, g \in J$ with $[f]_{\mathbf{x}^{\mathbf{u}}} =
  [g]_{\mathbf{x}^{\mathbf{u}}}$. Take a monomial $\bfx^\bfv$ such that
  $f \bfx^\bfv, g\bfx^\bfv \in I$. Since $[f\bfx^\bfv]_{\mathbf{x}^{\mathbf{u}+\bfv}} =
  [g\bfx^\bfv]_{\mathbf{x}^{\mathbf{u}+\bfv}}$, there is $h \in I$ satisfying
  $\elim{h}{\bfx^{\bfu+\bfv}}{f\bfx^\bfv}{g\bfx^\bfv}$. It follows that
  $h \bfx^{-\bfv} \in J$ satisfies $\elim{h\bfx^{-\bfv}}{\bfx^\bfu}{f}{g}$,
  showing that $J$ satisfies the monomial elimination axiom.
  
\item
  Suppose $I \subseteq \Rbar[x_1,\dots,x_n]$ is a tropical ideal.
  Since $I^h$ is a homogeneous ideal, it is enough to prove the monomial
  elimination axiom for homogeneous polynomials $f,g \in I^h$ of the same
  degree. Suppose $[f]_{\mathbf{x}^{\mathbf{u}}} =
  [g]_{\mathbf{x}^{\mathbf{u}}}$,
  and let $\mathbf{u}'$ be the last $n$
  coordinates of $\mathbf{u}$.  Then $f':=f|_{x_0=0}, g':=g|_{x_0=0} \in
  I$, and $[f']_{\mathbf{x}^{\mathbf{u}'}} =
  [g']_{\mathbf{x}^{\mathbf{u}'}}$.  By the monomial elimination axiom
  for $I$ there is $h' \in I$ with  $\elim{h'}{\bfx^{\bfu'}}{f'}{g'}$.
  Then  $h :=
  x_0^{\deg(f)-\deg(h')}\tilde{h'}$ satisfies $\elim{h}{\bfx^{\bfu}}{f}{g}$,
  showing that $I^h$ is a tropical ideal.

Conversely, suppose $J \subseteq \Rbar[x_0,\dots,x_n]$ is a homogeneous
tropical ideal, and let $f,g \in J|_{x_0=0}$ with
$[f]_{\mathbf{x}^{\mathbf{u}}} = [g]_{\mathbf{x}^{\mathbf{u}}}$. 
Without loss of generality we may assume that $d := \deg(\tilde{f}) - \deg(\tilde{g}) \geq 0$.  
Take $c \in \NN$ such that $x_0^c\tilde{f}, x_0^{c+d}\tilde{g} \in J$.
Denote $e = c+\deg(\tilde f) - |\bfu|$.
By the monomial elimination axiom for $J$ there is
$h\in J$ with $\elim{h}{x_0^{e}\mathbf{x}^{\mathbf{u}}}{x_0^c\tilde{f}}{x_0^{c+d}\tilde{g}}$, 
and thus $\elim{h|_{x_0=0}}{\mathbf{x}^{\mathbf{u}}}{f}{g}$, 
which shows that $J|_{x_0=0}$ is a tropical ideal.

\item Suppose that $f, g \in (I:m)$ with
  $[f]_{\mathbf{x}^{\mathbf{u}}} = [g]_{\mathbf{x}^{\mathbf{u}}}$.
  Then $mf, mg \in I$, with $[mf]_{m\mathbf{x}^{\mathbf{u}}} =
  [mg]_{m\mathbf{x}^{\mathbf{u}}}$, so there exists $h \in I$ with
  $\elim{h}{m\bfx^{\bfu}}{mf}{mg}$.
  Thus every term of $h$ is
  divisible by $m$, and so $h':=h/m \in (I:m)$ satisfies 
  $\elim{h'}{\bfx^{\bfu}}{f}{g}$,  
  as desired.  

We have $I \subseteq (I:m) \subseteq (I:m^2) \subseteq \dotsb$.  Since
these are all tropical ideals, by the ascending chain condition
\cite{TropicalIdeals}*{Theorem 3.11} this chain stabilizes, so there is $N \geq 0$
for which $(I:m^k) = (I:m^N)$ for all $k \geq N$.  
We then have $(I:m^{\infty}) =
(I:m^N)$, which shows that $(I:m^{\infty})$ is a tropical ideal.

When $m = \prod_{i=1}^n x_i$, if $m^kf \in I$ then $f= 1/m^k (m^kf)
\in I \Rbar[x_1^{\pm 1},\dots,x_n^{\pm 1}]$, so 
$(I:m^{\infty}) \subseteq I\Rbar[x_1^{\pm 1},\dots,x_n^{\pm 1}] \cap \Rbar[x_1,\dots,x_n]$.
Conversely, if $g \in
I\Rbar[x_1^{\pm 1},\dots,x_n^{\pm 1}] \cap \Rbar[x_1,\dots,x_n]$,
we have that $g = \mathbf{x}^{\mathbf{u}} f$ for $f \in I$ and
$\bfx^\bfu$ a (possibly Laurent) monomial.
Setting $k= -\min \{ u_i : u_i<0\}$, we get $m^kg \in I$.  

The last claim follows from the fact that $J = I \Rbar[x_1^{\pm 1},\dots,x_n^{\pm 1}]$.\qedhere
\end{enumerate}

\end{proof}

We next recall the Gr\"obner theory developed in \cite{TropicalIdeals}*{\S 3}.
For $\bfw \in \mathbb R^n$ and $f=\bigoplus_{\bfu \in \NN^n}
c_\bfu \bfx^\bfu \in \Rbar[x_1,\ldots,x_n]$, the {\defbold initial term}
of $f$ with respect to $\bfw$ is
\[ \inn_\bfw(f):=\bigoplus_{\bfu \,:\, c_\bfu+\bfu \cdot \bfw=f(\bfw)
} \bfx^{\bfu} \quad \in \BB[x_1,\ldots,x_n]. \]
For a tropical ideal $I$ we define the  {\defbold initial ideal} with respect to $\bfw$ as 
\[\inn_\bfw( I):=\langle \inn_\bfw( f) \mid f \in I \rangle \quad \subseteq \BB[x_1,\ldots,x_n].\] 
Note that in fact $\inn_\bfw(I)$ is equal to the set $\{ \inn_\bfw( f) \mid f \in I \}$, 
as this set is already closed under tropical addition, scalar multiplication, and multiplication by any monomial.
Analogous definitions apply for polynomials and ideals in the
Laurent polynomial ring $\Rbar[x_1^{\pm 1},\ldots,x_n^{\pm 1}]$, 
and also when the coefficients are in $\mathbb B$. 

We have the following relationships between initial ideals.

\begin{lemma}\label{l:basicinitial}
\leavevmode
\begin{enumerate}
\item \label{enum:initialaffinetorus} If $I \subseteq \Rbar[x_1,\dots,x_n]$ is a tropical ideal and $\mathbf{w} \in \mathbb R^n$ then
$$\inn_{\mathbf{w}}(I
  \Rbar[x_1^{\pm 1},\dots,x_n^{\pm 1}]) = \inn_{\mathbf{w}}(I) \mathbb
  B[x_1^{\pm 1},\dots,x_n^{\pm 1}].$$
\item \label{enum:initialaffineproj} If $I \subseteq \Rbar[x_0,\dots,x_n]$ is a homogeneous tropical ideal and $\mathbf{w} \in \mathbb R^n$ then 
$$\inn_{\mathbf{w}}(I|_{x_0=0}) = \inn_{(0,\mathbf{w})}(I)|_{x_0=0}.$$
\item \label{enum:initialtorustoaffine} If $I \subseteq \Rbar[x_1^{\pm 1},\dots,x_n^{\pm 1}]$ is a tropical
ideal and $\mathbf{w} \in \mathbb R^n$ then
$$\textstyle \inn_{\mathbf{w}}(I) \cap
\mathbb B[x_1,\dots,x_n] = (\inn_{\mathbf{w}}(I \cap
\Rbar[x_1,\dots,x_n]) : (\prod_{i=1}^n x_i)^{\infty}).$$
\end{enumerate}

\end{lemma}

\begin{proof}
\begin{enumerate}
\item Let $f \in I\Rbar[x_1^{\pm 1},\dots,x_n^{\pm 1}]$.  Then $f =
  \mathbf{x}^{\mathbf{u}} f'$ for some Laurent monomial
  $\mathbf{x}^{\mathbf{u}}$ and $f' \in I$, so $\inn_{\mathbf{w}}(f) =
  \mathbf{x}^{\mathbf{u}} \inn_{\mathbf{w}}(f') \in
  \inn_{\mathbf{w}}(I) \mathbb B[x_1^{\pm 1},\dots,x_n^{\pm 1}]$.
  Conversely, let $g \in \inn_{\mathbf{w}}(I) \mathbb B[x_1^{\pm
      1},\dots,x_n^{\pm 1}]$.  Then $g = \mathbf{x}^{\mathbf{u}} g'$
  for some Laurent monomial $\mathbf{x}^{\mathbf{u}}$ and $g' \in
  \inn_{\mathbf{w}}(I)$.  Write $g' = \inn_{\mathbf{w}}(f)$ for $f \in
  I$. Then $\mathbf{x}^{\mathbf{u}} f \in I \Rbar[x_1^{\pm 1},
    \dots,x_n^{\pm 1}]$, so
  $g =  \mathbf{x}^{\mathbf{u}} \inn_{\mathbf{w}}(f) = 
  \inn_{\mathbf{w}}(\mathbf{x}^{\mathbf{u}}f) \in
  \inn_{\mathbf{w}}(I \Rbar[x_1^{\pm 1},\dots,x_n^{\pm 1}])$.

\item If $f \in I$ then $\inn_{(0,\bfw)}(f)|_{x_0=0} = \inn_{\mathbf{w}}(f|_{x_0=0}) \in
  \inn_{\mathbf{w}}(I|_{x_0=0})$, so 
  $\inn_{(0,\mathbf{w})}(I)|_{x_0=0} \subseteq \inn_{\mathbf{w}}(I|_{x_0=0})$.  
  Conversely, if $g \in
  \inn_{\mathbf{w}}(I|_{x_0=0})$ then $g = \inn_\bfw(f|_{x_0 = 0})$ for some $f \in I$, so
  $g = \inn_{(0,\mathbf{w})}(f)|_{x_0=0} \in \inn_{(0,\mathbf{w})}(I)|_{x_0=0}$.

\item If $f \in \inn_{\mathbf{w}}(I) \cap \mathbb B[x_1,\dots,x_n]$
then $f = \inn_{\mathbf{w}}(g)$ for $g \in I$.  Choose a monomial
$\mathbf{x}^{\mathbf{u}}$ with $\mathbf{x}^{\mathbf{u}} g \in
\Rbar[x_1,\dots,x_n]$.  Then $\mathbf{x}^{\mathbf{u}} f =
\inn_{\mathbf{w}}(\mathbf{x}^{\mathbf{u}} g) \in \inn_{\mathbf{w}}(I
\cap \Rbar[x_1,\dots,x_n])$, so $f \in (\inn_{\mathbf{w}}(I \cap
\Rbar[x_1,\dots,x_n]) : (\prod_{i=1}^n x_i)^{\infty})$.
Conversely, if $f \in (\inn_{\mathbf{w}}(I \cap \Rbar[x_1,\dots,x_n])
: (\prod_{i=1}^n x_i)^{\infty})$ then there is $g \in I \cap
\Rbar[x_1,\dots,x_n]$ with $\mathbf{x}^{\mathbf{u}}f =
\inn_{\mathbf{w}}(g)$.  Then $\mathbf{x}^{-\mathbf{u}} g \in I$, and
$f = \inn_{\mathbf{w}}(\mathbf{x}^{-\mathbf{u}} g) \in
\inn_{\mathbf{w}}(I)$, as required.\qedhere
\end{enumerate}
\end{proof}

If $I \subseteq
\mathbb R[x_0,\ldots,x_n]$ is a homogeneous tropical ideal, its {\defbold Hilbert function}
is the map $H_I:\NN \to \NN$ given by $d \mapsto
\rank(\Mat(I_d))$, where $I_d$ is the degree-$d$ part of $I$.
More specifically, $H_I(d)$ is the size of any maximal subset $B$ of monomials of degree
$d$ with the property that $B$ does not contain the support of any polynomial in $I_d$. 
For 
 $\bfw \in \RR^n$ the initial ideal  $\inn_\bfw(I) \subseteq \BB[x_0,\ldots,x_n]$ is a homogeneous tropical ideal,
 and $H_{\inn_\bfw(I)} = H_I$  \cite{TropicalIdeals}*{Corollary 3.6}.

\subsection{Monomial term orders}

In commutative algebra over a field, Gr\"obner theory usually begins
with monomial term orders.  We now introduce these for the semiring of
tropical polynomials.

\begin{definition}
A total order $\prec$ on the monomials in $\Rbar[x_1,\dots,x_n]$ is a
monomial term order if
$$\bfx^{\mathbf{u}} \prec \bfx^{\mathbf{u}'} \text{   implies   } \bfx^{\mathbf{u+v}} \prec \bfx^{\mathbf{u'+v}} \text{  for all monomials } \bfx^{\mathbf{v}},$$
and 
$$\bfx^{\mathbf{u}} \prec  \bfx^\mathbf 0$$ for all monomials $\bfx^{\mathbf{u}} \neq \bfx^\mathbf 0$.
\end{definition}

The direction of the inequality in the second condition is to make
this compatible with the $\min$ convention for initial ideals that we
use in this theory.  It is the opposite of the usual order, but is {\defbold
  not} the tropical analogue of a local order in usual Gr\"obner
theory.

\begin{example}
Two central examples of monomial term orders are the lexicographic and
reverse-lexicographic term orders.  The {\defbold lexicographic} term order
$\prec_{lex}$ on $\Rbar[x_1,\dots,x_n]$ has $\bfx^{\mathbf{u}}
\prec_{lex} \bfx^{\mathbf{v}}$ if the first nonzero entry of
$\mathbf{u}-\mathbf{v}$ is positive.
The {\defbold reverse-lexicographic} term order $\prec_{revlex}$ on
$\Rbar[x_1,\dots,x_n]$ has $\bfx^{\mathbf{u}} \prec_{revlex}
\bfx^{\mathbf{v}}$ if $\deg(\bfx^{\mathbf{u}}) > \deg(\bfx^{\mathbf{v}})$, or
$\deg(\bfx^{\mathbf{u}}) = \deg(\bfx^{\mathbf{v}})$ and the last nonzero
entry of $\mathbf{u} - \mathbf{v}$ is negative.
Note that these are the reverse of the usual orders, to be compatible
with the $\min$ convention.  For example, 
$$ x_1^2 \,\prec_{lex}\, x_1 x_2^2 \,\prec_{lex}\, x_2 \,\prec_{lex}\, x_3^2 \,\prec_{lex}\, 0,$$
and 
$$x_3^3 \,\prec_{revlex}\, x_2^2 \,\prec_{revlex}\, x_1x_3 \,\prec_{revlex}\, 0,$$
where $0$ denotes the constant monomial $x_1^0x_2^0x_3^0$.
\end{example}

\begin{definition}
Let $\prec$ be a monomial term order on $\Rbar[x_1,\dots,x_n]$.  The
{\defbold initial term} of a tropical polynomial $f = \bigoplus c_{\mathbf{u}} 
\bfx^{\mathbf{u}} \in \Rbar[x_1,\dots,x_n]$ is $\inn_{\prec}(f) =
\bfx^{\mathbf{v}}$, where $\bfx^{\mathbf{v}} = \min_{\prec} \{
\bfx^{\mathbf{u}} : c_{\mathbf{u}} \neq \infty \}$.  The {\defbold initial ideal} of an ideal 
 $I \subseteq
\Rbar[x_1,\dots,x_n]$ is the monomial ideal 
$$\inn_{\prec}(I) := \langle \inn_{\prec}(f) : f \in I \rangle.$$
\end{definition}

As with traditional Gr\"obner bases, one use of monomial initial
ideals is that they give distinguished bases for the matroids
associated to a tropical ideal.

\begin{lemma} \label{l:stdmonosbasis}
Let $I \subseteq \Rbar[x_1,\dots,x_n]$ be a tropical ideal and let
$\prec$ be a monomial term order.  Then for any finite collection $E$ of
monomials in $\Rbar[x_1,\dots,x_n]$ the set of monomials in $E \setminus \{
\inn_{\prec}(f) : f \in I|_E \}$ is a
basis of $\uMat(I|_E)$. 
In particular, if $I$ is a homogeneous tropical ideal, $H_I(d) = H_{\inn_\prec(I)}(d)$ for all 
$d \geq 0$.
\end{lemma}

\begin{proof}  
If the set $B:= E \setminus \{\inn_{\prec}(f) : f \in I|_{E} \}$ 
were not an independent set of the matroid $\uMat(I|_E)$, then
there would be $f \in I$ with support in $B$.  We would then have
$\inn_{\prec}(f) \in B$, contradicting the definition of $B$.
  To show that $B$ is a
basis we show that for all $\mathbf{x}^{\mathbf{u}} \in E \setminus B$ there is
$f \in I$ supported in $B \cup \{ \mathbf{x}^{\mathbf{u}} \}$.  To see this,
fix $\bfx^{\mathbf{u}} \in E \setminus B$.  We have $\mathbf{x}^{\mathbf{u}} =
\inn_{\prec}(f)$ for $f \in I|_E$.
 We may assume that $f$ has been chosen so that the smallest
 $\mathbf{x}^{\mathbf{v}} \in \supp(f) \setminus (B \cup
 \{\mathbf{x}^{\mathbf{u}}\} )$ (if any exists) with respect to
 $\prec$ is as large as possible; this is possible because $E$ is
 finite.  For such a minimal $\mathbf{x}^{\mathbf{v}} \in \supp(f)
 \setminus (B \cup \{\bfx^{\mathbf{u}} \})$, there is a polynomial
 $f_{\mathbf{v}} \in I|_E$ with $\inn_{\prec}(f_{\mathbf{v}}) =
 \mathbf{x}^{\mathbf{v}}$, where we may assume that the coefficient of
 $\mathbf{x}^{\mathbf{v}}$ in $f_{\mathbf{v}}$ is $0$.  Since
 $\bfx^{\mathbf{u}} \prec \mathbf{x}^{\mathbf{v}}$, we have
 $\mathbf{x}^{\mathbf{u}} \not \in \supp(f_{\mathbf{v}})$.  Let
 $\alpha$ be the coefficient of $\mathbf{x}^{\mathbf{v}}$ in $f$, and
 let $f' \in I|_E$ be an elimination
 $\elim{f'}{\mathbf{x}^{\mathbf{v}}}{f}{\alpha 
   f_{\mathbf{v}}}$.  We have $\inn_{\prec}(f') =
 \mathbf{x}^{\mathbf{u}}$, and the smallest monomial in $\supp(f')
 \setminus (B \cup \{\mathbf{x}^{\mathbf{u}} \})$ is larger than
 $\mathbf{x}^{\mathbf{v}}$, which contradicts the construction of $f$.
 We thus conclude that there is $f \in I$ with $\inn_{\prec}(f) =
 \mathbf{x}^{\mathbf{u}}$, and $\supp(f) \subseteq B \cup
 \{\bfx^{\mathbf{u}} \}$, as claimed.
 
For a homogeneous ideal $I$ and $d\geq 0$, take $E$ to be the
 collection of monomials of degree $d$.  We then have that the set of
 monomials of degree $d$ not in $\inn_{\prec}(I)$ is a basis for
 $\uMat(I_d)$, as $\inn_{\prec}(I_d) = (\inn_{\prec}(I))_d$ for homogeneous ideals.  This 
 implies the equality of Hilbert functions $H_{\inn_{\prec}(I)}(d) =
 H_I(d)$.
\end{proof}

The following lemma states that for a {\defbold fixed} tropical ideal
$I$ every initial ideal with respect to a monomial term order is also
an initial ideal with respect to a weight vector $\mathbf{w} \in
\mathbb R^n$. 
The proof is very similar to the classical case; see
\cite{GBCP}*{Proposition 1.11}.

The {\defbold recession cone} of a polyhedron $P$ is the largest cone
$C$ for which the Minkowski sum $P+C \subseteq P$.  Equivalently, the
recession cone of a nonempty polyhedron $\{ \mathbf{x} : A \bfx \leq \mathbf b \}$ is the cone  $\{ \mathbf{x} : A \bfx \leq \bf0 \}$.

\begin{lemma} \label{l:weightvectortermorder}
Let $I$ be a homogeneous tropical ideal in $\Rbar[x_0,\dots,x_n]$, and let $\prec$ be a monomial term order.  
There is a nonempty polyhedron $C_{\prec} \subseteq \RR^{n+1}$ with an $(n+1)$-dimensional recession cone, with the property that $\inn_{\bfw}(I) = \inn_{\prec}(I)$ for all $\bfw$ in the interior of $C_{\prec}$. 
\end{lemma}

\begin{proof}
Let $\bfx^{\mathbf{u}_1},\dots,\bfx^{\mathbf{u}_s}$ be the minimal
generators of the monomial ideal $\inn_{\prec}(I)$. Write
$B_{|\mathbf{u}_i|}$ for the monomials in $\Rbar[x_0,\dots,x_n]$ of
degree $|\mathbf{u}_i|$ not in $\inn_{\prec}(I)$, which form a basis for $\Mat(I_d)$ by Lemma~\ref{l:stdmonosbasis}.  
For any $1 \leq i \leq s$, there exists a homogeneous
polynomial $g_i \in I$ with $\supp(g_i) \subseteq
B_{|\bfu_i|} \cup \{\bfx^{\mathbf{u}_i}\}$, corresponding to the
fundamental circuit of $\bfx^{\mathbf{u}_i}$ over $B_{|\bfu_i|}$.
After scaling, we can write $g_i
= \bfx^{\mathbf{u}_i} \tplus \bigoplus_{\mathbf{v} \in 
B_{|\bfu_i|}} c_{i\mathbf{v}} \bfx^{\mathbf{v}}$.  Let $C_{\prec}$ be
the closure of the set
$$C_{\prec}^\circ = \{ \mathbf{w} \in \mathbb R^{n+1} : \mathbf{w} \cdot \mathbf{u}_i < c_{i\mathbf{v}} + \mathbf{w} \cdot
\mathbf{v} \text{ for all } 1 \leq i \leq s \text{ and }
\bfx^{\mathbf{v}} \in \supp(g_i) \setminus \{\bfx^{\mathbf{u}_i} \} \}.$$
For any $\mathbf{w} \in C_{\prec}^\circ$ we have $\inn_{\mathbf{w}}(g_i) =
\bfx^{\mathbf{u}_i}$ for all $i$, and so $\inn_{\prec}(I) \subseteq
\inn_{\mathbf{w}}(I)$.  
We have $H_{\inn_{\mathbf{w}}(I)}(d) = H_I(d) = H_{\inn_{\prec}(I)}(d)$
for all $d \geq 0$ by \cite{TropicalIdeals}*{Corollary 3.6} and Lemma \ref{l:stdmonosbasis}. 
Thus we cannot have $\inn_{\prec}(I)$ properly contained in
$\inn_{\mathbf{w}}(I)$, as otherwise we would have a proper
containment of the sets of cycles of two matroids with the same
rank \cite{Oxley}*{Corollary 7.3.4}. 

It thus remains to show that $C_{\prec}$ is nonempty and has a full-dimensional recession cone.  Form the matrix
$U$ with $n+1$ columns whose $\ell$ rows are the vectors $-\mathbf{v}
+\mathbf{u}_i$ for $1 \leq i \leq s$ and $\bfx^{\mathbf{v}} \in \supp(g_i) \setminus \{\bfx^{\mathbf{u}_i} \}$.  
If $C_{\prec}^\circ$ is empty, then there is no $\mathbf{w} \in
\mathbb R^{n+1}$ for which $U\mathbf{w} < \mathbf{c}$, where the
$(i,\mathbf{v})$th entry of $\mathbf{c}$ is $c_{i\mathbf{v}}$, and the
inequality is coefficientwise.  There is thus also no $\mathbf{w} \in \mathbb
R^{n+1}_{\leq 0}$ with $U \mathbf{w} \leq \mathbf{c}'$,
where $\mathbf{c}'_i = \min(c_i,0)-1$.  Let $U'$ be the $(\ell+n+1)
\times (n+1)$ matrix with first $\ell$ rows equal to $U$, and the last
$n+1$ rows an identity matrix.  There is thus no $\mathbf{w} \in
\mathbb R^{n+1}$ with $U' \mathbf{w} \leq (\mathbf{c}',\mathbf{0})^T$.
By the Farkas lemma (\cite{Ziegler}*{Proposition 1.7}) there is thus $\mathbf{b}
\in \mathbb R^{\ell+n+1}_{\geq 0}$ with $\mathbf{b} \neq 0$ and
$\mathbf{b}^T U' = \mathbf{0}$.  Since $U'$ has integral entries, we
may choose $\mathbf{b} \in \mathbb N^{\ell+n+1}$.  Write
$\mathbf{b}_{i,\mathbf{v}}$ for the component of $\mathbf{b}$
corresponding to the row $-\mathbf{v}+\mathbf{u}_i$ of $U'$.  Then
since $\mathbf{b} \geq \mathbf{0}$, we must have $\sum_{i,\mathbf{v}} \mathbf{b}_{i,\mathbf{v}} (-\mathbf{v}+\mathbf{u}_i) \leq \mathbf{0}$.  This means that
$\prod_{i, \mathbf{v}} (\bfx^{\mathbf{u}_i})^{\mathbf{b}_{i,\mathbf{v}}}$ divides
$\prod_{i, \mathbf{v}} (\bfx^{\mathbf{v}})^{\mathbf{b}_{i,\mathbf{v}}}$,
so $\prod_{i, \mathbf{v}}
(\bfx^{\mathbf{u}_i})^{\mathbf{b}_{i,\mathbf{v}}} \succeq \prod_{i,
  \mathbf{v}} (\bfx^{\mathbf{v}})^{\mathbf{b}_{i,\mathbf{v}}}$. But
$\bfx^{\mathbf{u}_i} \prec \bfx^{\mathbf{v}}$ for all $1 \leq i \leq
s$ and all $\mathbf{x}^{\mathbf{v}} \in \supp(g_i) \setminus \{
\mathbf{x}^{\mathbf{u}_i} \}$, so
$\textstyle \prod_{i, \mathbf{v}} (\bfx^{\mathbf{u}_i})^{\mathbf{b}_{i,\mathbf{v}}} \prec \prod_{i, \mathbf{v}} (\bfx^{\mathbf{v}})^{\mathbf{b}_{i,\mathbf{v}}}$.
From this contradiction we conclude that $C_{\prec}^\circ$ is nonempty, and thus $C_{\prec}$ is nonempty as well.

Finally, note that the argument in the previous paragraph applies verbatim substituting $\mathbf c$ by $\bf0$ to show that the open cone 
$$\{ \mathbf{w} \in \mathbb R^{n+1} : \mathbf{w} \cdot \mathbf{u}_i < \mathbf{w} \cdot
\mathbf{v} \text{ for all } 1 \leq i \leq s \text{ and }
\bfx^{\mathbf{v}} \in \supp(g_i) \setminus \{\bfx^{\mathbf{u}_i} \} \}$$
is nonempty.  The recession cone of $C_\prec$ is the closure of this cone, so it is full dimensional.
\end{proof}

\begin{example}
Let $I \subseteq \Rbar[x_0,x_1,x_2]$ be the ideal of the point
$[0:0:0] \in \trop(\mathbb P^2)$.  This is tropicalization of the
ideal $\langle x_1-x_0,x_2-x_0 \rangle \subseteq K[x_0,x_1,x_2]$ for
any field $K$.  Let $\prec$ be the reverse lexicographic term order
with $x_0 \prec x_1 \prec x_2$.  Then 
$\inn_{\prec}(I) = \langle x_0,x_1 \rangle$.
The cone
$C_{\prec}^{\circ}$ from the proof of
Lemma~\ref{l:weightvectortermorder} is
$C_{\prec}^{\circ} = \{ \mathbf{w} \in \mathbb R^3 : w_0<w_2,
w_1<w_2 \}$. Note that while $\inn_{\mathbf{w}}(I)=\inn_{\prec}(I)$
for all $\mathbf{w} \in C_{\prec}^{\circ}$, we do not have
$\inn_{\mathbf{w}}(f) = \inn_{\prec}(f)$ for all $f \in I$ and
$\mathbf{w} \in C_{\prec}^{\circ}$.  For example,
$\inn_{\prec}(x_0 \tplus x_1) = x_0$, while $\inn_{(1,0,2)}(x_0 \tplus
x_1) = x_1$.
\end{example}

\subsection{Varieties of tropical ideals} \label{ss:varietiestropicalideals}

The {\defbold variety} of a tropical ideal $I \subseteq \Rbar[x_1,\dots,x_n]$ is 
$$V(I) := \{ \mathbf{w} \in \Rbar^n : f(\mathbf{w}) = \infty, \text{ or
  the minimum in } f(\mathbf{w}) \text{ is attained at least twice}
  \}.$$ 
The variety of an ideal $I \subseteq \Rbar[x_1^{\pm 1},\dots,x_n^{\pm
    1}]$ is defined similarly:
$$V(I) := \{ \mathbf{w} \in \mathbb R^n : \text{the minimum in }
f(\mathbf{w}) \text{ is attained at least twice}\}.$$ 
For a homogeneous ideal $I \subseteq \Rbar[x_0,\dots,x_n]$, we can think of its variety as a
subset of $$\trop(\mathbb P^n) = (\Rbar^{n+1} \setminus
(\infty,\dots,\infty) )/\mathbb R(1,\dots,1),$$
namely
$$ V(I) := \{ [\mathbf{w}] \in \trop(\mathbb P^n) \colon f(\mathbf{w})= \infty, \text{ or the min in } f(\mathbf{w}) \text{ is attained at least twice}  \}.$$
See \cite{TropicalIdeals}*{\S 4} for
more on this. 

Theorem 5.11 of \cite{TropicalIdeals} proves that if $I$ is a tropical
ideal in $\Rbar[x_1^{\pm 1},\dots,x_n^{\pm 1}]$,
$\Rbar[x_1,\dots,x_n]$, or a homogeneous tropical ideal in
$\Rbar[x_0,\dots,x_n]$, the variety $V(I)$ is the support of a finite
$\mathbb R$-rational polyhedral complex in either $\mathbb R^n$,
$\Rbar^n$, or $\trop(\mathbb P^n)$ respectively.  Here by {\defbold
$\mathbb R$-rational} we mean that every polyhedron in it has a rational
normal fan (but not necessarily rational vertices).  One source of
this polyhedral complex structure in the homogeneous case is the {\defbold Gr\"obner
complex} of $I$.  This is the finite $\mathbb R$-rational polyhedral
complex for which $\mathbf{w}$ and $\mathbf{w}'$ live in the same
relatively open polyhedron if and only if $\inn_{\mathbf{w}}(I)
= \inn_{\mathbf{w}'}(I)$; see
\cite{TropicalIdeals}*{\S 5}.

We have the following relationships between the varieties of ideals.

\begin{lemma} \label{l:varietyaffine}
\leavevmode
\begin{enumerate}
\item \, \label{enum:varietytorusaffine} Let $I \subseteq \Rbar[x_1^{\pm 1},\dots,x_n^{\pm 1}]$ be a tropical ideal, and let $J = I \cap \Rbar[x_1,\dots,x_n]$.  Then $V(J) \cap \mathbb R^n = V(I)$.

\item \, \label{enum:varietyaffineproj} Let $I \subseteq
  \Rbar[x_1,\dots,x_n]$ be a tropical ideal.  Then
$$ V(I^h) \cap \{[\mathbf{w}] : w_0=0 \} = \{ [0,\mathbf{w}'] :
  \mathbf{w}' \in V(I) \}.$$
\end{enumerate}
\end{lemma}

\begin{proof}

\begin{enumerate}
\item 
 Since every polynomial in $J$ is also in $I$, we have the inclusion
$V(I) \subseteq V(J) \cap \mathbb R^n$.  Now suppose $\mathbf{w} \in
\mathbb R^n$ is not in $V(I)$.  Then there is $f \in I$ with
$\inn_{\mathbf{w}}(f)$ a monomial.  Choose $\mathbf{x}^{\mathbf{u}}$
with $\mathbf{x}^{\mathbf{u}}f \in J$.  Then
$\inn_{\mathbf{w}}(\mathbf{x}^{\mathbf{u}} f) =
\mathbf{x}^{\mathbf{u}} \inn_{\mathbf{w}}(f)$ is also a monomial, so
$\mathbf{w} \not \in V(I)$.

\item For any $\bfw \in \RR^{n+1}$, write $\mathbf{w}'$ for the projection of $\mathbf{w}$ onto the last $n$ coordinates.
Then for every $f \in I$ and $\bfw \in \RR^{n+1}$ with $w_0=0$,
the minimum in $\tilde{f}(\mathbf{w})$ is attained twice 
if and only if the minimum in $f(\mathbf{w}')$ is attained twice, and so 
$\mathbf{w} \in V(I^h)$ if and only $\mathbf{w}' \in V(I)$.\qedhere
\end{enumerate}
\end{proof}

For realizable tropical ideals the variety of an initial ideal with
respect to $\mathbf{w}$ is the {\em star} of the variety at $\bfw$.  We now
extend this to all tropical ideals.

Let $\Sigma$ be a polyhedral complex in $\RR^n$, and let $\sigma$ be a cell of $\Sigma$.
The {\defbold linear span} of $\sigma$ is the linear subspace 
\[ \spann(\sigma) := \spann \{ \bfx - \bfy : \bfx, \bfy \in \sigma \}. \] 
The {\defbold star} $\starr_\Sigma(\sigma)$ of $\Sigma$ at $\sigma$ is
a polyhedral fan whose cones are indexed by the cells $\tau$ of
$\Sigma$ containing $\sigma$. The cone indexed by such a $\tau$ is the
convex cone
$\overline \tau := \cone \{ \bfx - \bfy : \bfx \in \tau \text{ and } \bfy \in \sigma \}$.
Equivalently, if $\bfw \in \relint(\sigma)$, we have
\[\overline \tau = \{ \bfv \in \RR^n : \bfw + \epsilon \bfv \in \tau  
\text{ for all } 0 < \epsilon \ll 1 \}.\]
The fan $\starr_\Sigma(\sigma)$ has lineality space equal to $\spann(\sigma)$.

If $\bfw \in \RR^n$ lies in the support of $\Sigma$, we set 
\[ \starr_{\Sigma} (\mathbf{w}) := \starr_\Sigma (\sigma), \]
where $\sigma$ is the cell of $\Sigma$ for which $\bfw \in \relint(\sigma)$.
If $\bfw$ is not in the support of $\Sigma$ we set $\starr_{\Sigma} (\mathbf{w}) = \emptyset$.

\begin{proposition} \label{p:star}
Fix a tropical ideal $I \subseteq \Rbar[x_1^{\pm 1},\dots,x_n^{\pm 1}]$, and $\bfv,
\bfw \in \RR^n$. Then we have $$\inn_{\bfv}(\inn_{\bfw}(I)) =
\inn_{\bfw + \epsilon \bfv}(I)$$ for $0 < \epsilon \ll 1$, and thus
$$V(\inn_{\mathbf{w}}(I)) = \starr_{V(I)}(\mathbf{w}).$$
\end{proposition}
\begin{proof}
For any $f \in \Rbar[x_1^{\pm 1}, \dots, x_n^{\pm 1}]$, 
$\inn_\bfv(\inn_{\bfw}(f)) = \inn_{\bfw+\epsilon \bfv}(f)$ for small
enough $\epsilon>0$.  Let $I^h \subseteq \Rbar[x_0, \dots, x_n]$
denote the homogenization of the ideal $I \cap \Rbar[x_1,\dots,x_n]$,
and consider $\widetilde{\bfw} := (0,\bfw) \in \mathbb R^{n+1}$ and
$\tilde{\bfv} := (0,\bfv) \in \mathbb R^{n+1}$.  Since the
Gr\"obner complex of $I^h$ is a finite polyhedral complex, there is
$\epsilon>0$ for which the ideal $\inn_{\tilde{\bfw}+\epsilon'
  \tilde{\bfv}}(I^h)$ is constant for all $0 < \epsilon' < \epsilon$.
If  $\inn_{\tilde{\bfw}+\epsilon' \tilde{\bfv}}(I^h)$ is different
from $\inn_{\tilde{\bfv}}(\inn_{\tilde{\bfw}}(I^h))$ then the two ideals differ
in some  degree $d$. Their degree $d$ parts are generated by
the corresponding initial forms of the (finitely many) circuits of
$I^h_d$,  and we can take $\epsilon'$ small enough so that for any such
circuit $f$ we have $\inn_{\tilde\bfv}(\inn_{\tilde\bfw}(f)) =
\inn_{\tilde\bfw+\epsilon' \tilde\bfv}(f)$, which is a contradiction.
Finally, we have that $I = I^h|_{x_0 = 0}$, and so by
Part~\ref{enum:initialaffineproj} of Lemma~\ref{l:basicinitial} we
get that for any $0 < \epsilon' < \epsilon$,
\[\inn_{\bfv}(\inn_{\bfw}(I)) = \inn_{\tilde{\bfv}}(\inn_{\tilde{\bfw}}(I^h))|_{x_0=0} = 
\inn_{\tilde{\bfw}+\epsilon' \tilde{\bfv}}(I^h)|_{x_0=0} = \inn_{\bfw+\epsilon' \bfv}(I).\]
The fact that $V(\inn_{\mathbf{w}}(I)) = \starr_{V(I)}(\mathbf{w})$ then follows 
 directly from the definitions.
\end{proof}
Later in Proposition \ref{p:starweighted} we show that the equality 
$V(\inn_{\mathbf{w}}(I)) = \starr_{V(I)}(\mathbf{w})$ is in fact an equality
of \emph{weighted} polyhedral fans.

Any $\bfv \in \ZZ^n$ induces a grading of the semiring $\Rbar[x_1^{\pm
    1}, \dots, x_n^{\pm 1}]$, by setting $\deg(x_i) = v_i$, so the
    degree of a term $c \bfx^\bfu$ is $\bfv \cdot \bfu \in \ZZ$.  
If $L \subseteq \RR^n$ is a rational $d$-dimensional linear subspace, fixing a
basis $\bfv_1, \bfv_2, \dots, \bfv_d$ of $L$ with $\bfv_i \in \ZZ^n$ for all $i$ 
gives then rise to a $\mathbb Z^d$-grading on $\Rbar[x_1^{\pm 1}, \dots, x_n^{\pm 1}]$, where
the degree of a term $c \bfx^\bfu$ is $(\bfv_1 \cdot \bfu, \dots, \bfv_d \cdot \bfu) \in \ZZ^d$.

\begin{corollary}\label{c:homogeneous}
Let $I \subseteq \Rbar[x_1^{\pm 1}, \dots, x_n^{\pm 1}]$ be a tropical
ideal, and let $\bfw \in V(I)$ lie in the relative interior of a cell
$\sigma$ of the Gr\"obner complex of $I$. Then $\inn_{\bfw}(I)$ is
homogeneous with respect to the grading by $\mathbf{v}$ for any
$\bfv \in \spann(\sigma)$.  Thus $\inn_{\mathbf{w}}(I)$ is homogeneous
with respect to a $\mathbb Z^{\dim(\sigma)}$-grading induced by $\spann(\sigma)$.
\end{corollary}
\begin{proof}
By Proposition \ref{p:star}, if $\bfv \in \spann(\sigma)$ then
for $0 < \epsilon \ll 1$ we have
$\inn_\bfw(I) = \inn_{\bfw+\epsilon \bfv}(I) = \inn_\bfv(\inn_\bfw(I))$.
This shows that $\inn_\bfw(I)$ is generated by polynomials of the
form $\inn_\bfv(f)$ with $f \in \BB[x_1^{\pm 1},\dots,x_n^{\pm 1}]$,
and thus homogeneous with respect to the grading by $\mathbf{v}$.
\end{proof}

There is a tight connection between the tropicalization of a classical variety over the same field with a nontrivial and trivial valuation.  We now extend this to tropical ideals.
Let $\varphi : \Rbar \to \mathbb B$ be the semiring homomorphism
defined by $\varphi(a) = 0$ if $a \neq \infty$, and $\varphi(\infty) =
\infty$.  This induces a semiring homomorphism $$\varphi :
\Rbar[x_1, \dotsc, x_n] \to \mathbb \BB[x_1, \dotsc, x_n].$$  The image
$\varphi(I)$ of $I$ is a tropical ideal in $\BB[x_1, \dotsc, x_n]$,
called the {\defbold trivialization} of $I$; in fact, we have
$\Mat(\varphi(I)|_E) = \uMat(I|_E)$ for any finite collection $E$ of monomials. 
For a monomial term order $\prec$
we have $\inn_{\prec}(I) = \inn_{\prec}(\varphi(I))$.  The same
notions apply to ideals in the Laurent polynomial semiring
$\Rbar[x_1^{\pm 1}, \dots, x_n^{\pm 1}]$.

The set of the
recession cones of all polyhedra in a polyhedral complex $\Sigma$ is
not always a fan, as the cones may not intersect correctly; see, for
example, \cite{BurgosGilSombra}.  However, when $X$ is a subvariety of
$(K^*)^n$ and $\Sigma$ is a polyhedral complex structure on
$\trop(X)$, then this set is a rational polyhedral fan
\cite{TropicalBook}*{Theorem 3.5.6}.  We now show that this generalizes to
tropical ideals.

We will make use of the following notation. 
The {\defbold normal complex} $\mathcal N(f)$ of a polynomial $f \in
\Rbar[x_0,\dotsc,x_n]$ is the $\mathbb R$-rational polyhedral complex
in $\mathbb R^{n+1}$ whose polyhedra are the closures of the sets
$C[\mathbf{w}] = \{ \mathbf{w}' \in \mathbb R^{n+1} :
\inn_{\mathbf{w}'}(f) = \inn_{\mathbf{w}}(f) \}$ for $\mathbf{w} \in
\mathbb R^{n+1}$.

\begin{proposition}\label{p:recessionfan} 
Let $I \subseteq \Rbar[x_0,\dots,x_n]$ be a homogeneous tropical ideal.  
The Gr\"obner complex in $\RR^{n+1}$ of the trivialization $\varphi(I)$ is the recession fan of the Gr\"obner complex of $I$ in $\RR^{n+1}$.
The maximal cones of this fan correspond to the monomial initial ideals of $I$ of the form $\inn_\prec(I)$ with $\prec$ a monomial term order. 
In addition, if $\Sigma$ is a polyhedral complex with $|\Sigma|=V(I)$,
then the support of the recession fan of $\Sigma$ is $V(\varphi(I))$.
\end{proposition}

\begin{proof}
  For any $d \geq 0$, denote by $\mon_d$ the set of monomials in
  $\Rbar[x_0,\dots,x_n]$ of degree $d$, and let $p : \binom{\mon_d}{r_d} \to \Rbar$ be the basis valuation function of
  the rank $r_d$ valuated matroid $\Mat(I_{d})$.  Consider the polynomial
$$ F_d := \bigoplus_{B \text{ basis of } \uMat(I_{d})} p(B) 
\ttimes 
 \left( \prod_{\mathbf{x}^\mathbf{u} \in \mon_d \setminus B} 
\mathbf{x}^\mathbf{u} \right) \quad \in \Rbar[x_0,\dots,x_n]. $$
Theorem 5.6 of \cite{TropicalIdeals} shows that for $D \gg 0$, the
Gr\"obner complex of $I$ is equal to the normal complex $\mathcal
N(F)$ of the polynomial $F = \prod_{d\leq D} F_d$.  In a similar way,
the Gr\"obner fan of the trivialization $\varphi(I)$ is the normal
complex $\mathcal N(G)$ of the polynomial $G = \prod_{d\leq D}
G_d \in \BB[x_0,\dots,x_n]$, where
$$G_d := \bigoplus_{B \text{ basis of } \uMat(I_d)} \left( \prod_{\mathbf{x}^\mathbf{u} \in \mon_d \setminus B} \mathbf{x}^\mathbf{u} \right) \quad \in \BB[x_0,\dots,x_n]$$
and $D \gg 0$. Note that $G = \varphi(F)$. The statement that the Gr\"obner fan of $\varphi(I)$ is the recession fan of the Gr\"obner complex of $I$ follows then from the fact that for any polynomial $f$, the normal complex of $\varphi(f)$ is the recession fan of the normal complex of $f$. Indeed, the normal complex of $f$ is a polyhedral complex dual to the regular subdivision of the Newton polytope $\mathrm{NP}(f)$ of $f$ induced by the coefficients of $f$, and its recession fan is the normal fan of $\mathrm{NP}(f)$, which is the normal complex of $\varphi(f)$.

We now show that the maximal cones of the Gr\"obner fan $\Sigma$ of $\varphi(I)$ correspond to monomial initial ideals $\inn_\prec(I)$ with $\prec$ a monomial term order. 
Lemma \ref{l:weightvectortermorder} ensures that any monomial initial ideal $\inn_\prec(I) = \inn_{\prec}(\varphi(I))$ is equal to $\inn_{\bfw}(\varphi(I))$ for $\bfw$ in the relative interior of a maximal cone of $\Sigma$. 
Conversely, suppose $C$ is a maximal cone of $\Sigma$, and take $\bfw \in C$ with all its coordinates linearly independent over $\QQ$.
Since $\varphi(I)$ is homogeneous, we can subtract a large  multiple of $(1,\dots,1)$ and assume that all the entries of $\bfw$ are negative.
The ordering on monomials given by $\bfx^\bfu \prec \bfx^\bfv$ if $\bfw \cdot \bfu \leq \bfw \cdot \bfv$ is then a total order, and it satisfies the two conditions for it to be a monomial term order. 
By definition, we have $\inn_\prec(I) = \inn_\bfw(\varphi(I))$, and thus the cone $C$ corresponds to the monomial initial ideal $\inn_\prec(I)$.

Finally, to prove the claim about the tropical varieties, we first
observe the analogous claim for a single tropical polynomial $f \in
\Rbar[x_0,\dots,x_n]$.  The variety $V(f)$ is the codimension-one
skeleton of the normal complex to the subdivision of the Newton
polytope $\mathrm{NP}(f)$ of $f$ induced by the coefficients of $f$.
Maximal cells of $V(f)$ are dual to edges of this subdivision.  These
cells are unbounded, so have a nontrivial recession cone, only if the
dual edge is part of an edge of the Newton polytope $\mathrm{NP}(f)$.
In that case the recession cone of the cell is the normal cone to the
edge.  Since $V(\varphi(f))$ is the codimension-one skeleton of the
normal fan of $\mathrm{NP}(f)$, and the maximal cells of
$\mathrm{NP}(f)$ are the normal cones to edges; this proves the claim
for a single tropical polynomial.

For the general case, by \cite{TropicalIdeals}*{Theorem 5.9} there exists 
a finite collection $f_1,\dots,f_s$ of polynomials in
$I$ that form a tropical basis for $I$ and for which
$\varphi(f_1),\dots,\varphi(f_s)$ form a tropical basis for $\varphi(I)$, meaning that
$V(I) = \bigcap_{i=1}^s V(f_i)$ and $V(\varphi(I))= \bigcap_{i=1}^s
V(\varphi(f_i))$. The result then follows from the fact that the recession cone of
the intersection of two polyhedra is the intersection of the two
recession cones, and so the recession fan of the intersection of the
complexes $V(f_i)$ for $1 \leq i \leq s$ equals the intersection of the fans $V(\varphi(f_i))$, as required.
\end{proof}

\begin{example} \label{e:recession}
Let $J= \langle xy+xz+yz+2z^2 \rangle \subseteq \mathbb Q[x,y,z]$, where
$\mathbb Q$ has the $2$-adic valuation, and let $I = \trop(J)$.  
The Gr\"obner complex of $I$ is the normal complex 
$\mathcal N(xy \tplus xz \tplus yz \tplus 1 \ttimes z^2)$, shown on the left of Figure~\ref{f:recessionfan} with the lineality
space $\mathbb R (1,1,1)$ quotiented out.  
The Gr\"obner complex of the trivialization $\varphi(I)$ is the normal
complex $\mathcal N(xy \tplus xz \tplus yz \tplus z^2)$, shown on the right of
Figure~\ref{f:recessionfan}.  Note that the second complex is the
recession fan of the first.
\begin{figure}
\centering
\includegraphics[scale=1]{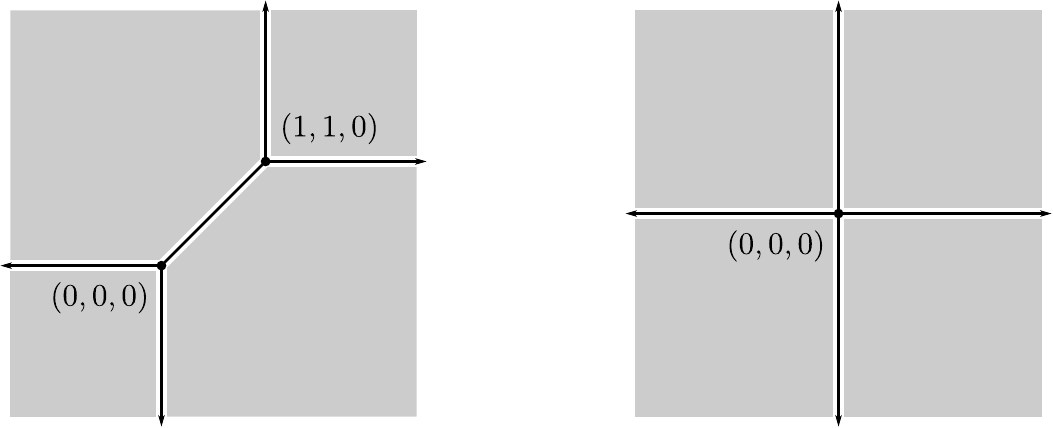}
\caption{\label{f:recessionfan} The Gr\"obner complex and its recession fan from Example~\ref{e:recession}.}
\end{figure}
\end{example}


\section{Specialization} \label{s:specialization}

In this section we prove that the class of tropical ideals 
is closed under specialization of the variables
(Theorem~\ref{t:restrictiontheorem}).

\begin{definition}
If $f \in \Rbar[x_1,\dotsc,x_{n},y_1,\dots,y_m]$ and
$\mathbf{a}=(a_1,\dots,a_m) \in \Rbar^m$, we write 
$f|_{\mathbf{y}=\mathbf{a}} := f(x_1,\dotsc,x_{n},a_1,\dots,a_m) \in \Rbar[x_1,\dotsc,x_{n}]$.
For an ideal $I \subseteq \Rbar[x_1,\dotsc,x_{n},y_1,\dots,y_m]$, we denote by
$I|_{\mathbf{y}=\mathbf{a}}$ the set
\[ I|_{\mathbf{y}=\mathbf{a}} := \{ f|_{\mathbf{y}=\mathbf{a}} : f
\in I \} \subseteq 
\Rbar[x_1,\dotsc,x_{n}], \] and call it the
     {\defbold{specialization}} of $I$ at $\mathbf{y}=\mathbf{a}$.  Note that $I|_{\mathbf{y}=\mathbf{a}}$
     is an ideal in $\Rbar[x_1,\dotsc,x_{n}]$.
\end{definition}

For any polynomial $f \in \Rbar[x_1,\dotsc,x_{n},y_1,\dots,y_m]$ and
any monomial $\bf x^\bfu$ in the variables $x_1,\dotsc,x_{n}$, we
denote by $f_\bfu \in \Rbar[y_1,\dots,y_m]$ the coefficient of
$\mathbf{x}^{\mathbf{u}}$ in $f$ viewed as a polynomial in
$x_1,\dots,x_{n}$, so
\[
f(x_1,\dots,x_{n},y_1,\dots,y_m) = \bigoplus_{\bfu \in \NN^{n}} f_\bfu(\mathbf{y}) \ttimes \bf x^\bfu.
\]

Our main result in this section is that if $I$ is a tropical ideal
then any specialization of $I$ is also a tropical ideal.  The proof
boils down to the following lemma.  For tropical polynomials $f, g$ 
we write $f \geq g$ if the inequality holds
coefficientwise.

\begin{lemma} \label{l:maineliminationstep}
Let $I$ be a tropical ideal in $\Rbar[x_1,\dotsc,x_n,y_1,\dots,y_m]$.
Suppose that $\bfx^\bfu, \bfx^\bfv$ are monomials in the variables 
$x_1,\dotsc,x_{n}$.  If $f, g \in I$ satisfy $f_{\mathbf{u}}(\bfzero) \leq
g_{\mathbf{u}}(\bfzero)$, and $f_{\mathbf{v}}(\bfzero) > g_{\mathbf{v}}(\bfzero)$, then
there is $h \in I$ with $h_{\mathbf{u}} = \infty$, $h_{\mathbf{v}}(\bfzero)
= g_{\mathbf{v}}(\bfzero)$, $h|_{\mathbf{y}=\bfzero} \geq f|_{\mathbf{y}=\bfzero}
\tplus g|_{\mathbf{y}=\bfzero}$, and 
$\inn_{\mathbf{0}}(h_{\mathbf{v}}) =
\inn_{\mathbf{0}}(g_{\mathbf{v}})$.
\end{lemma}

\begin{proof}
If $g_{\mathbf{u}}=\infty$ then we can take $h=g$, so henceforth we
assume that $g_{\mathbf{u}} \neq \infty$.  The proof is by induction
on $m$.  In the case $m=0$ the polynomials $f_{\mathbf{u}},
f_{\mathbf{v}}, g_{\mathbf{u}}$, and $g_{\mathbf{v}}$ are all
constants, and the claim is true for an elimination
$\elim{h}{\mathbf{x}^{\mathbf{u}}}{(g_{\mathbf{u}}-f_{\mathbf{u}})
  \ttimes f}{g}$.
 We now assume that $m>0$ and that the lemma is true for smaller $m$.

 Throughout the proof, we  make use of the induction hypothesis and apply the lemma 
 by regarding the variable $y_m$ as an $x$ variable.  For a polynomial $h \in I$ we write $h_{(\mathbf{u},j)} \in \Rbar[y_1,\dots,y_{m-1}]$ for the coefficient of $\mathbf{x}^{\mathbf{u}}y_m^j$ in $h$ viewed as a polynomial in $x_1,\dots,x_n,y_m$, so 
 $$ h(x_1,\dots,x_{n},y_1,\dots,y_m) = \bigoplus_{\substack{\bfu \in \NN^{n} \\ j \in \NN}} h_{(\bfu,j)}(y_1,\dots,y_{m-1}) \ttimes \bfx^{\bfu} y_m^j.$$
 Note that for any $\bfu \in \NN^n$ we have
 $$ h_\bfu(y_1, \dots, y_m) = \bigoplus_{j \in \NN} h_{(\bfu,j)}(y_1,\dots,y_{m-1}) \ttimes y_m^j.$$
 We denote by $\deg_{y_m}(h_\bfu)$ the maximum $j$ such that $h_{(\bfu,j)} \neq \infty$. 

 Now, suppose the lemma is not true for the tropical ideal $I$ and the
  monomials $\bfx^\bfu, \bfx^\bfv$.    
   Choose a counterexample $f,g \in I$
  which is minimal in the sense that $(\deg_{y_m}(g_\bfu),
  \deg_{y_m}(f_\bfu)) \in \NN \times \NN$ is
  lexicographically minimal among all counterexamples $f,g \in I$.  
Consider the sets
\begin{align*}
\F &:= \{ h \in I : h|_{\mathbf{y}=\bfzero} \geq  f|_{\mathbf{y}=\bfzero} \tplus g|_{\mathbf{y}=\bfzero} \text{ and } h_\bfv(\bfzero) > g_{\bfv}(\bfzero) \},\\
\G &:= \{ h \in I : h|_{\mathbf{y}=\bfzero} \geq f|_{\mathbf{y}=\bfzero} \tplus g|_{\mathbf{y}=\bfzero} \text{ and } h_\bfv(\bfzero) = g_{\bfv} (\bfzero) \}.
\end{align*}
Note that $f \in \F$ and $g \in \G$. 
Moreover, regarding $y_m$ as an $x$ variable, any application of the
lemma to monomials $\mathbf{x}^{\mathbf{u}}y_m^j$ and 
$\mathbf{x}^{\mathbf{v}}y_m^k$ in polynomials $h_1 \in \F$ and $h_2 \in \G$ with $(h_2)_{(\mathbf{v},k)}(\bfzero) =
  g_{\mathbf{v}}(\bfzero)$ yields a polynomial $h_3 \in \G$.  
Similarly, 
any application of the lemma to any two monomials
in polynomials $h_1, h_2 \in \F$ gives a polynomial $h_3 \in \F$.

Set 
$$d := \deg_{y_m}(f_\bfu) \in \NN \quad \text{ and } \quad \alpha := f_{(\bfu,d)}(\bfzero) \in \RR,$$ 
and similarly 
$$e:=\deg_{y_m}(g_\bfu) \in \NN \quad \text{ and } \quad \beta := g_{(\bfu,e)}(\bfzero) \in \RR.$$

\begin{claim}\label{claimnog}
There is no $h \in \G$ with $c := \deg_{y_m}(h_\bfu) \leq d$ and $\gamma := h_{(\bfu,c)}(\bfzero) < \alpha$.
\end{claim}
\noindent Suppose that such an $h \in \G$ exists.
Set $h' = (\alpha-\gamma)\ttimes y_m^{d-c}\ttimes h \in \F$, so that both
$f_\bfu$ and $h'_\bfu$ are polynomials of degree $d$ in $y_m$ 
and $h'_{(\bfu,d)}(\bfzero) = \alpha$. 
Since $f_{\mathbf{u}}(\bfzero) \leq h_{\mathbf{u}}(\bfzero) \leq \gamma < \alpha$, there is $j<d$
with $f_{(\mathbf{u},j)}(\bfzero) = f_{\mathbf{u}}(\bfzero)$. 
We have $h'_{(\bfu,d)}(\bfzero) = \alpha = f_{(\bfu,d)}(\bfzero)$ and 
$h'_{(\bfu,j)}(\bfzero) > h_{(\bfu,j)}(\bfzero) \geq h_{\mathbf{u}}(\bfzero) \geq f_{\mathbf{u}}(\bfzero) = f_{(\bfu,j)}(\bfzero)$, 
so we can apply the lemma to the monomials $\mathbf{x}^{\mathbf{u}}y_m^d$ and
$\mathbf{x}^{\mathbf{u}}y_m^j$ in the polynomials $h'$ and $f$ to obtain a polynomial $f' \in \F$
with $f'_{(\bfu,d)} = \infty$ and $f'_{\bfu}(\bfzero) = f'_{(\bfu,j)}(\bfzero) = f_{\bfu}(\bfzero)$.  
Since $\deg_{y_m}(f'_\bfu) < d$, our
minimality assumption implies that $f',g$ satisfy the statement of the
lemma, so there is $h \in I$ with $h_\bfu = \infty$, $h|_{\mathbf{y}=\bfzero}
\geq f'|_{\mathbf{y}=\bfzero} \tplus g|_{\mathbf{y}=\bfzero} \geq f|_{\mathbf{y}=\bfzero} \tplus g|_{\mathbf{y}=\bfzero}$
and $\inn_{\mathbf{0}}(h_{\mathbf{v}}) = \inn_{\mathbf{0}}(g_{\mathbf{v}})$.
But then $h$ contradicts the assumption that $f$ and $g$ were a
counterexample to the lemma, finishing the proof of the claim.

\begin{claim}\label{claimnof}
There is no $h \in \F$ with $c := \deg_{y_m}(h_\bfu) \leq e$ and $\gamma := h_{(\bfu,c)}(\bfzero) \leq  \beta$.
\end{claim}
\noindent Suppose that such an $h \in \F$ exists.
Set $h' = y_m^{e-c}\ttimes h \in \F$, so that both $h'_\bfu$ and $g_\bfu$ are
polynomials of degree $e$ in $y_m$.
Fix $j$ such that $g_{(\mathbf{v},j)}(\bfzero) = g_{\mathbf{v}}(\bfzero)$.
We have $h'_{(\bfu,e)}(\bfzero) = \gamma \leq g_{(\bfu,e)}(\bfzero)$ and 
$h'_{(\bfv,j)}(\bfzero) \geq h'_{\bfv}(\bfzero) = h_{\bfv}(\bfzero) > g_{\mathbf{v}}(\bfzero) = g_{(\bfv,j)}(\bfzero)$, 
so we can apply the lemma to the monomials $\mathbf{x}^{\mathbf{u}}y_m^e$ and
$\mathbf{x}^{\mathbf{v}}y_m^j$ in the polynomials $h'$ and $g$ to obtain a polynomial $g^j \in \G$
with $g^j_{(\bfu,e)}(\bfzero) = \infty$, $g^j_{(\bfv,j)}(\bfzero) = g_{(\bfv,j)}(\bfzero) = g_{\bfv}(\bfzero)$, and $\inn_{\mathbf{0}}({g^j}_{(\mathbf{v},j)}) = \inn_{\mathbf{0}}(g_{(\mathbf{v},j)})$.
Let $g' \in \mathcal G$ be the sum of
 all $g^j$ over all such choices of $j$. 
Then $g'_{(\bfu,e)}(\bfzero) = \infty$, and $g'_{\bfv}(\bfzero) =
g_{\bfv}(\bfzero)$.  Since $\inn_{\mathbf{0}}(g_{\mathbf{v}}) =
\bigoplus_j \inn_{\mathbf{0}}(g_{(\mathbf{v},j)})$, this implies that
$\inn_{\mathbf{0}}(g'_{\mathbf{v}})=
\inn_{\mathbf{0}}(g_{\mathbf{v}})$.
  Since $\deg_{y_m}(g'_\bfu) < e$, our minimality
  assumption implies that $f, g'$ satisfy the statement of the lemma,
  so there is $h \in I$ such that $h_{\mathbf{u}} = \infty$,
  $h_{\mathbf{v}} (\bfzero) = g'_{\mathbf{v}}(\bfzero) = g_{\mathbf{v}}(\bfzero)$,
  $h|_{\bfy=\bfzero} \geq f|_{\bfy=\bfzero} \tplus g'|_{\bfy=\bfzero} \geq f|_{\bfy=\bfzero} \tplus
  g|_{\bfy=\bfzero}$, and $\inn_{\mathbf{0}}(h_{\mathbf{v}}) =
  \inn_{\mathbf{0}}(g'_{\mathbf{v}}) =
  \inn_{\mathbf{0}}(g_{\mathbf{v}})$.
Such an $h$ contradicts the assumption that $f$ and $g$ were a
counterexample to the lemma, so this finishes the proof of the claim.

Since $g \in \G$, Claim \ref{claimnog} implies that if $
\alpha>\beta$ then $d < e$.  Also, since $f \in \F$, Claim
\ref{claimnof} implies that if $\alpha \leq \beta$ then $d > e$.  We
now show that both of these cases are impossible, which leads to a
contradiction to our original assumption that the counterexample
exists.

$\bullet$ {\bf Case $\alpha > \beta$ and $d<e$}. 
We inductively construct an infinite sequence of polynomials
$f_1, f_2, \dotsc \in \F$ satisfying the following conditions:
\begin{enumerate}
\item For all $i$ we have  $(f_i)_\bfu(\bfzero)=f_{\mathbf{u}}(\bfzero)$.
\item Set $d_i := \deg_{y_m}((f_i)_\bfu)$ and $\alpha_i := (f_i)_{(\bfu,d_i)}(\bfzero)$.  We have 
 $\alpha_i > \beta$ and $d_i < e$ for all $i$.  
\item Set $l_i := \max \{ j  \colon (f_i)_{(\mathbf{u},j)}(\bfzero) = f_{\mathbf{u}}(\bfzero) \}$.
  We have $l_1 < l_2 < \dotsb $.
\end{enumerate}
Set $f_1 := f$, and suppose that we have constructed $f_i \in \F$. To
construct $f_{i+1}$, set $f'_i = y_m^{e-d_i} \ttimes f_i \in \F$ and
$g'_i = (\alpha_i-\beta) \ttimes g \in \F$, so that both $(f'_i)_\bfu$
and $(g'_i)_\bfu$ are polynomials of degree $e$ in $y_m$  
and $(f'_i)_{(\bfu,e)}(\bfzero) = \alpha_i = (g'_i)_{(\bfu,e)}(\bfzero)$.
Since $(f'_i)_{\mathbf{u}}(\bfzero) = (f_i)_\bfu(\bfzero) = f_{\mathbf{u}}(\bfzero)$, 
there is $k$ such that $(f'_i)_{(\mathbf{u},k)}(\bfzero) = f_{\mathbf{u}}(\bfzero)$.
Take $k$ to be the largest possible such value, which is equal to $(e - d_i) + l_i$.
We have $(g'_i)_{(\bfu,e)}(\bfzero) = (f'_i)_{(\bfu,e)}(\bfzero)$ and 
$(g'_i)_{(\bfu,k)}(\bfzero) > g_{(\bfu,k)}(\bfzero) \geq g_{\bfu}(\bfzero) \geq f_{\mathbf{u}}(\bfzero) = (f'_i)_{(\bfu,k)}(\bfzero)$, 
so $k \neq e$, and we can apply the lemma to the monomials $\mathbf{x}^{\mathbf{u}}y_m^e$ and
$\mathbf{x}^{\mathbf{u}}y_m^k$ in the polynomials $g'_i$ and $f'_i$ to obtain a polynomial $f_{i+1} \in \F$
with $(f_{i+1})_{(\bfu,e)} = \infty$ and $(f_{i+1})_{\bfu}(\bfzero) = (f_{i+1})_{(\bfu,k)}(\bfzero) = f_{\bfu}(\bfzero)$.
We have $l_{i+1} := \max \{ j  : (f_{i+1})_{(\mathbf{u},j)}=f_{\mathbf{u}}(\bfzero) \} = k = 
 (e-d_i) + l_i$,
so $l_{i+1} > l_i$.
Furthermore, note that $d_{i+1} := \deg_{y_m}((f_{i+1})_\bfu) < e$, so by Claim \ref{claimnof} we must also have $\alpha_{i+1} > \beta$.
This all shows that $f_{i+1}$ satisfies the desired properties.
We conclude this case by noting that the sequence $l_1 < l_2 < \dotsb $
is a strictly increasing infinite sequence of integers
bounded above by $e$, which is a contradiction.

$\bullet$ {\bf Case $\alpha \leq \beta$ and $d>e$}. 
We inductively construct a sequence of polynomials
$g_1, g_2, \dotsc \in \G$ satisfying the following conditions:
\begin{enumerate}
\item Set $e_i := \deg_{y_m}((g_i)_\bfu)$ and $\beta_i := (g_i)_{(\bfu,e_i)}(\bfzero)$.  We have 
$\alpha \leq \beta_i$ and $d > e_i$ for all $i$.
\item \label{enum:increasing} Let $l_i := \min \{ j : (g_i)_{(\mathbf{v},j)} = g_\bfv(\bfzero) \}$.
We have $l_1 < l_2 < \dotsb $.
\end{enumerate}
Set $g_1 := g$, and suppose that we have constructed $g_i \in \G$. To
construct $g_{i+1}$, set $g'_i = y_m^{d-e_i} \ttimes g_i \in \G$, so that both
$f_\bfu$ and $(g'_i)_\bfu$ are polynomials of degree $d$ in $y_m$. 
Set $k := (d - e_i) + l_i$, which is the minimum value
satisfying $(g'_i)_{(\mathbf{v},k)}(\bfzero)=
g_{\mathbf{v}}(\bfzero)$.  
We have $f_{(\bfu,d)}(\bfzero) = \alpha \leq \beta_i = (g'_i)_{(\bfu,d)}(\bfzero)$ and 
$f_{(\bfv,k)}(\bfzero) \geq f_{\bfv}(\bfzero) > g_{\mathbf{v}}(\bfzero) = (g'_i)_{(\bfv,k)}(\bfzero)$, 
so we can apply the lemma to the monomials $\mathbf{x}^{\mathbf{u}}y_m^d$ and
$\mathbf{x}^{\mathbf{v}}y_m^k$ in the polynomials $f$ and $g'_i$ to obtain a polynomial $g_{i+1} \in \G$
with $(g_{i+1})_{(\bfu,d)}(\bfzero) = \infty$, and $(g_{i+1})_{(\bfv,k)}(\bfzero) = (g'_i)_{(\bfv,k)}(\bfzero) = g_{\bfv}(\bfzero)$.
We have $l_{i+1} := \min \{ j  : (g_{i+1})_{(\mathbf{v},j)}=g_{\mathbf{v}}(\bfzero) \} = k = 
 (d-e_i) + l_i$,
so $l_{i+1} > l_i$.
This all shows that $g_{i+1}$ satisfies the desired properties.
Now, as the sequence of degrees $e_1, e_2, \dotsc $ is an infinite sequence of integers bounded above 
by $d$, it must contain an infinite constant subsequence $e_{i_1} = e_{i_2} = \dotsb $.
This contradicts the following claim.

\begin{claim} There is no infinite sequence of polynomials $g_1, g_2, \dotsc \in \G$
satisfying the condition $l_1<l_2 < \dotsb$ of \eqref{enum:increasing} with $\deg_{y_m}((g_1)_\bfu) = \deg_{y_m}((g_2)_\bfu) = \dotsb $.
\end{claim}
\noindent To prove the claim, suppose such a sequence $g_1, g_2,
\dotsc \in \G$ exists, and assume $c := \deg_{y_m}((g_1)_\bfu) =
\deg_{y_m}((g_2)_\bfu) = \dotsb$ is minimal among all such sequences.
Set $\beta_i := (g_i)_{(\bfu,c)}(\bfzero)$. For any $j>0$, let $g'_{2j-1} =
\max(0, \beta_{2j}-\beta_{2j-1}) \ttimes g_{2j-1}$ and $g'_{2j} =
\max(0, \beta_{2j-1}-\beta_{2j}) \ttimes g_{2j}$, so that both
$(g'_{2j-1})_\bfu$ and $(g'_{2j})_\bfu$ are polynomials 
of degree $c$ in $y_m$ satisfying $(g'_{2j-1})_{(\bfu,c)}(\bfzero) = (g'_{2j})_{(\bfu,c)}(\bfzero)$.
Set $l = l_{2j-1}$ if $\beta_{2j-1} \geq \beta_{2j}$, and $l = l_{2j}$ if 
$\beta_{2j-1} < \beta_{2j}$.  Note that $l$ is 
 the minimum value for which either 
 $(g'_{2j-1})_{(\mathbf{v},l)}(\bfzero)=g_{\mathbf{v}}(\bfzero)$ or
$(g'_{2j})_{(\mathbf{v},l)}(\bfzero)=g_{\mathbf{v}}(\bfzero)$, 
and that exactly one of these equalities holds. 
Since $(g'_{2j-1})_{(\bfu,c)}(\bfzero) = (g'_{2j})_{(\bfu,c)}(\bfzero)$ and 
$(g'_{2j-1})_{(\bfv,l)}(\bfzero) \neq (g'_{2j})_{(\bfv,l)}(\bfzero)$, 
so $c \neq l$ and we can apply the lemma to the monomials $\mathbf{x}^{\mathbf{u}}y_m^c$ and
$\mathbf{x}^{\mathbf{v}}y_m^l$ in the polynomials $g'_{2j-1}$ and $g'_{2j}$ (possibly in the reverse order) to obtain a polynomial $g''_j \in I$
with $(g''_j)_{(\bfu,c)}(\bfzero) = \infty$, $(g''_j)_{(\bfv,l)}(\bfzero) = \min((g'_{2j-1})_{(\bfv,l)}(\bfzero), (g'_{2j})_{(\bfv,l)}(\bfzero)) = g_{\bfv}(\bfzero)$,
and $g''_j|_{\mathbf{y}=\bfzero} \geq g'_{2j-1}|_{\mathbf{y}=\bfzero} \tplus g'_{2j}|_{\mathbf{y}=\bfzero} \geq f|_{\mathbf{y}=\bfzero} \tplus
g|_{\mathbf{y}=\bfzero}$.
As $\min \{s  : (g''_j)_{(\mathbf{v},s)} = g_\bfv(\bfzero) \} = l$, 
which is equal to either $l_{2j-1}$ or $l_{2j}$, it follows that 
$g''_1, g''_2, \dotsb$ is a sequence of polynomials in $\G$
satisfying condition \eqref{enum:increasing}.

Now, since the sequence of degrees $\deg_{y_m}((g''_1)_\bfu),
\deg_{y_m}((g''_2)_\bfu), \dotsc$ is an infinite sequence of
non-negative integers strictly less than $c$, it must contain an
infinite constant subsequence $\deg_{y_m}((g''_{i_1})_\bfu) =
\deg_{y_m}((g''_{i_2})_\bfu) = \dotsb $.  The sequence
$g''_{i_1}, g''_{i_2}, \dotsc$ is then an infinite sequence of polynomials
in $\G$ satisfying condition \eqref{enum:increasing} such that $c >
\deg_{y_m}((g''_{i_1})_\bfu) = \deg_{y_m}((g''_{i_2})_\bfu) = \dotsb
$, which contradicts the minimality of $c$, showing that no such
sequence exists.
\end{proof}

We now use Lemma \ref{l:maineliminationstep} to prove 
the key specialization theorem. We will also need Lemma \ref{l:maineliminationstep}
in its full strength in the proof of the projection theorem in Section \ref{d:dimensionprojection}.

\begin{theorem} \label{t:restrictiontheorem}
Let $I \subseteq \Rbar[x_1,\dots,x_{n}]$ be a tropical ideal. 
For any $a \in \Rbar$, the ideal 
$I|_{x_n=a} \subseteq \Rbar[x_1,\dots,x_{n-1}]$ is also a tropical ideal.
\end{theorem}

\begin{proof}
If $a=\infty$, the monomial elimination axiom for $I|_{x_n=a}$ follows
directly from the monomial elimination axiom for $I$. 
Indeed, if $f,g \in I|_{x_n=\infty}$ then there are $F,G \in I$ such that
$f = F|_{x_n=\infty}$ and $g = G|_{x_n=\infty}$. Suppose 
$\bfx^\bfu$ is a monomial such that $[f]_{\bfx^\bfu} = [g]_{\bfx^\bfu} \neq \infty$. 
As evaluating the variable $x_n$ at $\infty$ does not change the coefficients
of the monomials not divisible by $x_n$, we also have 
$[F]_{\bfx^\bfu} = [G]_{\bfx^\bfu} \neq \infty$.
Since $I$ is a tropical ideal, there exists $H \in I$ with
$\elim{H}{\bfx^\bfu}{F}{G}$.
This implies that $h := H|_{x_n=\infty} \in I|_{x_n=\infty}$ is an elimination
$\elim{h}{\bfx^\bfu}{f}{g}$, showing that $I|_{x_n=\infty}$
satisfies the monomial elimination axiom.

Suppose now that $a \neq \infty$. Consider the ideal
$I':= \{ f(x_1,\dots,x_{n-1},a x_n) : f \in I\} \subseteq \Rbar[x_1,\dots,x_{n}]$.
Since $I'$ is obtained from $I$ by doing an invertible scaling of
the variable $x_n$, the fact that $I$ is a tropical ideal implies that
$I'$ is also a tropical ideal.  Note that $I|_{x_n=a} = I'|_{x_n=0}$, and
thus we may assume that $a=0$.

To show that $I|_{x_n=0}$ is a tropical ideal, 
fix two polynomials $f,g \in I|_{x_n=0}$ and a monomial
$\mathbf{x}^{\mathbf{u}}$ with $[f]_{\bfx^\bfu} =
[g]_{\bfx^\bfu} \neq \infty$.
Choose $F,G \in I$ such that $f = F|_{x_n=0}$ and $g = G|_{x_n=0}$.
For any monomial $\bfx^\bfv$ for which $F_\bfv(0) \neq
G_\bfv(0)$, we can use Lemma \ref{l:maineliminationstep}
to construct a polynomial $H^{\bfv} \in I$ such that 
$H_\bfu = \infty$, $H_\bfv(0) = \min(F_\bfv(0),G_\bfv(0))$, 
and $H|_{x_n=0} \geq F|_{x_n=0} \tplus G|_{x_n=0}$.
Let $H \in I$ be the sum of all such polynomials $H^{\bfv}$. 
Then $h := H|_{x_n=0} \in
I|_{x_n=0}$ satisfies $[h]_{\bfx^\bfu}=\infty$ and $[h]_{\bfx^\bfv} 
\geq \min( [f]_{\bfx^\bfv}, [g]_{\bfx^\bfv} )$ 
for all monomials $\bfx^\bfv$, with the equality holding
whenever $[f]_{\bfx^\bfv} \neq [g]_{\bfx^\bfv}$. 
This shows that $I|_{x_n=0}$ satisfies the monomial elimination axiom,
and so it is a tropical ideal.
\end{proof}

\begin{remark} \label{r:realizableslicing}
If $K$ is a valued field with an uncountable residue field $\K$, $I
\subseteq K[x_1,\dots,x_n]$ is an ideal, and $a \in \mathbb R$ is an
element of the value group of $K$, then 
$\trop(I)|_{x_n=a} = \trop(I|_{x_n=\alpha})$ where $\alpha$ is a
sufficiently generic element of $K$ with valuation $a$.  To see this,
first fix $\alpha_0 \in K$ of valuation $a$.  For $f \in I$, write $f$
as a polynomial in $x_1,\dots,x_{n-1}$ with coefficients in $K[x_n]$.
We claim that for each such coefficient $g= \sum c_i x_n^i$ and
$\alpha$ of valuation $a$ we have $ \val(g(\alpha)) \geq \trop(g)(a)$,
with equality for all but finitely values of
$\overline{\alpha/\alpha_0} \in \K$.  The inequality is immediate from the
valuation axioms, so we only need justify the equality condition.  Fix
$j$ with $\val(c_j\alpha^j) = \trop(g)(a)$, and note that $g =
c_j\alpha_0^j \sum_i (c_i\alpha_0^{i-j}/c_j ) (x_n/\alpha_0)^i$.  Set
$b_i = c_i \alpha_0^{i-j}/c_j$ and $g' = \sum b_i y^i$.  By
construction $\val(g(\alpha))>\trop(g)(a)$ if and only if
$\val(g'(\alpha/\alpha_0)) > \trop(g')(0)=0$.  This occurs if and only
if $\overline{\alpha/\alpha_0}$ is one of the finitely many roots of
$\overline{g'} = \sum_i \overline{b_i} y_i \in \K[x_1,\dots,x_n]$.

The ideal $\trop(I)|_{x_n=a}$ is
generated by its circuits, which are specializations at $x_n=a$ of
circuits of $\trop(I)$.  These are tropicalizations of polynomials in
$I$, and up to scaling, there are a countable number of them.  There are thus, up to scaling,
 a countable
number of polynomials in $\overline{g}' \in \K[x_n]$ where $g$ is a
coefficient of a circuit.  We have $\trop(I)|_{x_n=a} =
\trop(I|_{x_n=\alpha})$ for all $\alpha \in K$ with $\val(\alpha)=a$
and the property that $\alpha/\alpha_0$ is not a
root of any of these polynomials in $\K[x_n]$.

To see that some hypothesis on the field is necessary, consider the
trivial valuation on $\mathbb Z/2\mathbb Z$, and the ideal $I= \langle
y^2+y+x\rangle \subseteq \mathbb Z/2\mathbb Z[x,y]$.  Then $\trop(I) \subseteq \mathbb B[x,y]$ contains $y^2 \tplus y \tplus x$
but does not contain any polynomial in $\mathbb B[y]$, and is saturated
with respect to $x$.  It follows that $\trop(I)|_{y=0}$ contains $x \tplus 0$, but does
not contain $x$.  However the only element of $\mathbb Z/2 \mathbb Z$
with valuation $0$ is $1$, so $I|_{y=1} = \langle x \rangle$, and thus $\trop(I|_{y=1})$ 
is  not equal to $\trop(I)|_{y=0}$.

On the other hand, $\trop(I)|_{x_n=\infty} = \trop(I|_{x_n=0})$ holds
without any hypothesis on the field.
\end{remark}

\begin{corollary} \label{c:Laurentworkstoo}
Let $I \subseteq \Rbar[x_1^{\pm 1},\dots,x_n^{\pm 1}]$ be a tropical
ideal, and let $J = I \cap \Rbar[x_1,\dots,x_n]$.  For any $a \in
\mathbb R$ the ideal 
$$I|_{x_n=a} := \{ f|_{x_n =a} : f \in I \} \subseteq
\Rbar[x_1^{\pm 1},\dots,x_{n-1}^{\pm 1}]$$ equals $J|_{x_n=a}
\Rbar[x_1^{\pm 1},\dots,x_{n-1}^{\pm 1}]$.  Thus $I|_{x_n=a}$ is a
tropical ideal.
\end{corollary}

\begin{proof}
The equality $I|_{x_n=a} =
J|_{x_n=a} \Rbar[x_1^{\pm 1},\dots,x_{n-1}^{\pm 1}]$ follows from
$I = J \Rbar[x_1^{\pm 1},\dots,x_{n}^{\pm 1}]$, 
and from the fact that for any polynomial $f \in J$ and any
Laurent monomial $\bfx^\bfu$ we have 
$(f\bfx^\bfu)|_{x_n=a} = f|_{x_n=a} \ttimes \bfx^\bfu|_{x_n=a}$.
Part~\ref{enum:torustoaffine} of Lemma~\ref{l:basicproperties} 
and Theorem~\ref{t:restrictiontheorem} thus imply that $I|_{x_n=a}$ 
is a tropical ideal.
\end{proof}

The following example shows that polynomials of degree at most $d$ in
a specialization are not necessarily specializations of
polynomials of degree at most $d$ in the ideal.

\begin{example} \label{e:nonhomog}
Let $J = \langle x_1-1, x_2-x_3 \rangle \subseteq \mathbb
C[x_1,x_2,x_3]$, and let $I = \trop(J)$.  Consider the specialization $I' = I|_{x_3=0}$.
We have $x_1 \tplus 0 \in I$, and thus in $x_1 \tplus 0 \in I'$.  Also, $x_2 \tplus x_3
\in I$, so $x_2 \tplus 0 \in I'$.  Since $I'$ is a tropical ideal, the
monomial elimination axiom implies that $x_1 \tplus x_2
\in I'$.  However there is no polynomial with tropicalization $x_1
\tplus x_2$ in $J$; we need to set $x_3$ to $0$ in the polynomial
$x_1x_3 \tplus x_2 = \trop(x_3(x_1-1)-(x_2-x_3)) \in I$.  
\end{example}

Example~\ref{e:nonhomog} can be homogenized to
show that if $I$ is a homogeneous tropical ideal, the operation of
specializing $x_n = x_0$, which is an ideal in
$\Rbar[x_0,\dots,x_{n-1}]$, is not always a tropical ideal. This
can be fixed, however, by saturating appropriately.

\begin{proposition}\label{prop:homogeneousspecialization}
Let $I \subseteq \Rbar[x_0,\dots, x_n]$ be a homogeneous tropical ideal, 
and denote by $I|_{x_0 = 0} \subseteq \Rbar[x_1,\dots, x_n]$ its dehomogenization.
For any $a \in \Rbar$, we have
\[ (I|_{x_n = \, a x_0}:x_0^{\infty}) = ((I|_{x_0 = 0})|_{x_n = a})^h \quad \subseteq \Rbar[x_0,\dots, x_{n-1}].\]
In particular, $(I|_{x_n = \, a x_0}:x_0^{\infty})$ is a homogeneous tropical ideal.
\end{proposition}
\begin{proof}
For any homogeneous polynomial $f \in \Rbar[x_0,\dots, x_n]$, we have, for some $r \geq 0$, 
the equality 
$f|_{x_n = \, a x_0} = x_0^r \ttimes (f|_{x_0 = 0, \, x_n = a})^h \in \Rbar[x_0,\dots, x_{n-1}]$.
This implies that $I|_{x_n = \, a x_0} \subseteq ((I|_{x_0 = 0})|_{x_n = a})^h \subseteq (I|_{x_n = \, a x_0}:x_0^{\infty})$. 
As $((I|_{x_0 = 0})|_{x_n = a})^h$ is saturated with respect to $x_0$, the desired equality must hold.
\end{proof}

We finish this section with some observations about the effect of
specialization on initial ideals and varieties.  Recall that the
trivialization of a tropical ideal $I$ is the image of $I$ under the 
homomorphism induced by $\varphi \colon \Rbar \rightarrow \mathbb B$
defined by $\varphi(a) = 0 $ if $a \neq \infty$ and $\varphi(\infty) =
\infty$.

\begin{lemma} \label{l:basicfacts}
  Let $I$ be a tropical ideal in $\Rbar[x_1,\dots,x_n]$ or $\Rbar[x_1^{\pm 1},\dots,x_n^{\pm 1}]$.
\leavevmode
\begin{enumerate}
\item \label{item:trivialization} For all $a \in \Rbar$ we have
  $\varphi(I|_{x_n=a})=\varphi(I)|_{x_n=\varphi(a)}$.
\item \label{item:initial} For all $\mathbf{w} \in \mathbb R^{n-1}$ and all $a \in \mathbb R$ 
we have $\inn_{\mathbf{w}}(I|_{x_n=a}) = \inn_{(\mathbf{w},a)}(I)|_{x_n=0}$.
\end{enumerate}
\end{lemma}

\begin{proof}
In both cases, Part \ref{item:trivialization} is expressing the fact that applying $\varphi$ 
to a polynomial commutes with specializing a variable, as $\varphi$ is a semiring homomorphism.

To prove Part \ref{item:initial} for $I \subseteq \Rbar[x_1,\dots,x_n]$, note that $\inn_{\mathbf{w}}(I|_{x_n=a})$ is equal to the set of 
polynomials of the form $\inn_{\mathbf{w}}(f|_{x_n=a})$ with $f \in I$, while $\inn_{(\mathbf{w},a)}(I)|_{x_n=0}$
is equal to the set of polynomials of the form $\inn_{(\mathbf{w},a)}(f)|_{x_n=0}$ with $f \in I$.
It thus suffices to show that  
\begin{equation}\label{eq:initialspecialization}
\inn_{\mathbf{w}}(f|_{x_n=a}) = \inn_{(\mathbf{w},a)}(f)|_{x_n=0}
\end{equation}
for any polynomial $f = \bigoplus c_\bfu \bfx^\bfu \in \Rbar[x_1,\dots,x_n]$.
For $\bfu \in \mathbb N^n$, denote by $\bfu'$ its projection 
onto the first $n-1$ coordinates. 
The initial term $\inn_{(\bfw,a)}(f)$ is equal to the sum of those monomials
$\bfx^\bfu$ for which $c_\bfu + (\bfw, a) \cdot \bfu$ is smallest, and so
$\inn_{(\bfw,a)}(f)|_{x_n=0}$ is the sum of those
monomials $\bfx^{\bfu'}$ for which there exists $u_n$ with $c_{(\bfu',u_n)} + \bfw \cdot \bfu'
+ a \, u_n$ equal to the minimum value $f(\bfw, a)$.  
Now, the coefficient of
$\mathbf{x}^{\mathbf{u}'}$ in $f|_{x_n=a}$ is 
$\min(c_{(\mathbf{u}',u_n)} + u_n \, a) =: b_{\mathbf{u}'}$.  The minimum value of
$b_{\mathbf{u}'}+\mathbf{w} \cdot \mathbf{u}'$ is thus also $f(\bfw, a)$,
and $\inn_{\mathbf{w}}(f|_{x_n=a})$ is the sum of those monomials
$\mathbf{x}^{\mathbf{u}'}$ achieving this minimum. 
The desired equality follows from the fact that $\mathbf{u}'$ achieves the minimum
if and only if any choice of $u_n$ such that
$\min(c_{(\mathbf{u}',u_n)} + u_n \, a) = b_{\mathbf{u}'}$ satisfies
$c_{(\bfu',u_n)} + \bfw \cdot \bfu' + a \, u_n = f(\bfw, a)$.

The case that $I \subseteq \Rbar[x_1^{\pm 1},\dots,x_n^{\pm 1}]$ 
follows from the above argument using
Part~\ref{enum:initialaffinetorus} of Lemma~\ref{l:basicinitial},
since $I|_{x_n=a} = (I \cap \Rbar[x_1,\dots,x_n])|_{x_n=a} \mathbb
R[x_1^{\pm 1},\dots,x_n^{\pm 1}]$.
\end{proof}

Recall that two polyhedral complexes $\Sigma_1, \Sigma_2$ in $\mathbb
R^n$ {\defbold intersect transversely} at $\mathbf{w} \in \mathbb R^n$ if
$\mathbf{w}$ lies in the relative interior of $\sigma_1 \in \Sigma_1$
and $\sigma_2 \in \Sigma_2$, and $\spann(\sigma_1) + \spann(\sigma_2) = \RR^n$.
We say that $\Sigma_1$ and $\Sigma_2$ intersect transversely if they
intersect transversely at any $\bfw \in \Sigma_1 \cap \Sigma_2$.

\begin{proposition} \label{p:transverseintersection}
Fix a tropical ideal $I \subseteq \Rbar[x_1^{\pm 1},\dots,x_n^{\pm
    1}]$, and let $a \in \mathbb R$.  Then $V(I|_{x_n=a}) \subseteq
\mathbb R^{n-1}$ in contained in $\pi(V(I) \cap \{x_n=a \})$, where
$\pi : \mathbb R^n \rightarrow \mathbb R^{n-1}$ is the projection onto
the first $n-1$ coordinates.  Moreover if $\mathbf{w}$ lies in a
closed cell $\sigma$ of a polyhedral complex $\Sigma$ with
$|\Sigma|=V(I)$ with the property that $\spann(\sigma) \not \subseteq
\{x_n=0 \}$, then $\pi(\mathbf{w}) \in V(I|_{x_n=a})$.  Thus if the
intersection of $V(I)$ and $\{x_n=a \}$ is transverse at $\mathbf{w}
\in \mathbb R^n$ then $\pi(\mathbf{w}) \in V(I|_{x_n=a})$.
\end{proposition}

\begin{proof}
Fix $\mathbf{w}' \in V(I|_{x_n=a}) \subseteq \mathbb R^{n-1}$.  Then for any $f
\in I$, the minimum in $f|_{x_n=a}(\mathbf{w}')$ is achieved at least
twice, say at monomials $\mathbf{x}^{\mathbf{u}}$ and
$\mathbf{x}^{\mathbf{v}}$.  Write $i,j$ for exponents of $x_n$ at
which the minimum in the univariate polynomials $f_{\mathbf{u}}(a)$ and $f_{\mathbf{v}}(a)$ is
achieved.  Then the minimum in $f(\mathbf{w}',a)$ is achieved at the
terms $\mathbf{x}^{\mathbf{u}}x_n^i$ and
$\mathbf{x}^{\mathbf{v}}x_n^j$, and so $(\bfw',a) \in V(I)$.

Now suppose that $\mathbf{w}$ lies in a closed cell $\sigma$ of a
polyhedral complex $\Sigma$ with $|\Sigma|=V(I)$ with the property
that $\spann(\sigma) \not \subseteq \{x_n=0 \}$, but assume that
$\pi(\bfw) \notin V(I|_{x_n=a})$.  Then there is $f \in I|_{x_n=a}$
with $\inn_{\pi(\bfw)}(f)$ equal to a monomial
$\mathbf{x}^{\mathbf{u}}$.  Fix $F \in I$ with $f = F|_{x_n=a}$.  Then
$\inn_{\mathbf{w}}(F)$ equals $\mathbf{x}^{\mathbf{u}}$ times a
polynomial $g \in \mathbb{B}[x_n^{\pm 1}]$ with more than one term, by
\eqref{eq:initialspecialization}.  This means that for any
$\mathbf{w}' \in V(\inn_{\mathbf{w}}(I))$ we must have $w'_n=0$.  By
Proposition~\ref{p:star}, the variety of the initial ideal
$\inn_{\mathbf{w}}(I)$ is the star of $V(I)$ at $\mathbf{w}$.  But
this contradicts the assumption that $\mathbf{w} \in \sigma$ with
$\spann(\sigma) \not \in \{ x_n=0\}$.  The claim about transverse
intersection is the special case that $\mathbf{w} \in
\relint(\sigma)$.
\end{proof}

In the case that the variety $V(I)$ is the support of a pure polyhedral complex, 
we show later in Proposition \ref{p:stableintersection} that $V(I|_{x_n=a})$ is the 
{\em stable intersection} of $V(I)$ and $\{x_n=a\}$.


\section{Dimension and Projections} \label{d:dimensionprojection}

In this section we prove several fundamental results about the dimension
(Theorem~\ref{t:dimiscorrect}) and projections
(Theorem~\ref{t:projection}) of varieties of tropical ideals.

For a homogeneous tropical ideal $I \subseteq \Rbar[x_0,\dots,x_n]$, the
Hilbert function $H_I(d) = \rk(\uMat(I_d))$ agrees with a
polynomial for $d \gg 0$, called the {\defbold Hilbert polynomial} of $I$
\cite{TropicalIdeals}*{Proposition 3.8}. The {\defbold dimension} of $I$ is defined to be the
degree of this polynomial.  We can extend this definition to tropical
ideals in $\Rbar[x_1,\dots,x_n]$ and $\Rbar[x_1^{\pm 1},\dots,x_n^{\pm
    1}]$, by setting $\dim(I)= \dim(I^h)$ for $I \subseteq
\Rbar[x_1,\dots,x_n]$, and $\dim(I) = \dim(I \cap
\Rbar[x_1,\dots,x_n])$ for $I \subseteq \Rbar[x_1^{\pm
    1},\dots,x_n^{\pm 1}]$.

\begin{proposition} \label{p:dimgoesdown}
Let $I$ be a proper tropical ideal in $\Rbar[x_1^{\pm 1},\dots,x_n^{\pm 1}]$ or $\Rbar[x_1,\dots,x_n]$.  
For any $a \in \mathbb R$ we have 
$$\dim(I|_{x_n=a}) \leq \dim(I)-1.$$
\end{proposition}
\begin{proof}
If $I$ is a tropical ideal in $\Rbar[x_1^{\pm 1},\dots,x_n^{\pm 1}]$,
let $J = I \cap \Rbar[x_1, \dots, x_n]$. By definition, $\dim(J) =
\dim(I)$. We also have $J|_{x_n=a} = (I|_{x_n=a}) \cap
\Rbar[x_1,\dots,x_{n-1}]$, and so
$\dim(J|_{x_n=a})= \dim(I|_{x_n=a})$.
 Replacing $I$ by $J$, we see
that we can reduce to proving the statement for ideals in
$\Rbar[x_1,\dots,x_n]$.

Suppose that $I$ is a tropical ideal in $\Rbar[x_1,\dots,x_n]$. 
Since the dimension of a tropical ideal depends only on its trivialization, 
by Part~\ref{item:trivialization} of Lemma~\ref{l:basicfacts} the dimension of $I|_{x_n = a}$ does not depend on the value of $a \in \RR$.  

For any $\bfw \in \RR^n$, write $\mathbf{w}'$ for the projection of 
$\mathbf{w}$ onto the first
$n-1$ coordinates. 
By Part~\ref{item:initial} of Lemma~\ref{l:basicfacts}, 
we have $\inn_{\mathbf{w}'}(I|_{x_n=w_n}) =
(\inn_{\mathbf{w}}(I))|_{x_n=0}$. 
In particular, if $\inn_{\mathbf{w}}(I)$ is a
monomial ideal then $\inn_{\mathbf{w}'}(I|_{x_n = w_n})$ is
a monomial ideal as well. 

We now prove that there is a $\bfw \in \RR^n$ such that $\inn_\bfw(I)$ 
is a monomial ideal
with the property that $(\inn_{\bfw'}(I|_{x_n=w_n}))_{\leq d} =
\inn_{\bfw'}((I|_{x_n=w_n})_{\leq d})$ for all $d \geq 0$.  Consider
the homogenization $I^h \subseteq \Rbar[x_0, \dots, x_n]$ of $I$. The 
Gr\"obner complex of $I^h$ is a finite polyhedral complex in 
$\RR^{n+1}$ whose maximal cells correspond to monomial initial ideals,
and so there exists $w_n \in \RR$ such that the set of $(w_0, \dots,
w_{n-1}) \in \RR^n$ for which $\inn_{(w_0,\dots,w_{n-1}, w_n)}(I^h)$
is not a monomial ideal is a polyhedral complex in $\RR^n$ of
dimension at most $n-1$.  Set $I' = I|_{x_n=w_n}$.  Let $\prec$ be a
reverse-lexicographic order on $\Rbar[x_0,\dots,x_{n-1}]$ with the
variables ordered so that $x_i \prec x_0$ for all $i = 1, \dots, n-1$.
By Proposition \ref{p:recessionfan}, the set of $(w_0,\dots,w_{n-1})
\in \RR^{n}$ such that $\inn_{(w_0,\dots,w_{n-1})}((I')^h) =
\inn_{\prec}((I')^h)$ is an $n$-dimensional open polyhedron in $\RR^n$, and so
our assumption on $w_n$ implies we can pick one such
$(w_0,\dots,w_{n-1})$ with the additional property that
$\inn_{(w_0,\dots,w_{n-1}, w_n)}(I^h)$ is a monomial ideal.  Adding a
suitable multiple of $(1,\dots,1)$, we can assume that $w_0 = 0$.  We
claim that $\bfw = (w_1,\dots,w_n)$ chosen in this way satisfies the
desired properties. First, by Part \ref{enum:initialaffineproj} of
Lemma \ref{l:basicinitial}, we have $\inn_\bfw(I) =
(\inn_{(0,\bfw)}(I^h))|_{x_0=0}$, which is a monomial ideal.  It
remains to be checked that $(\inn_{\bfw'}(I'))_{\leq d} =
\inn_{\bfw'}(I'_{\leq d})$. The inclusion $(\inn_{\bfw'}(I'))_{\leq d}
\supseteq \inn_{\bfw'}(I'_{\leq d})$ holds for any ideal $I'$ and any
vector $\mathbf{w}'$. For the reverse inclusion, take $\bfx^\bfu$ a
generator of the monomial ideal $\inn_{\bfw'}(I')$ satisfying
$\deg(\bfx^\bfu) \leq d$.  Again by Part
\ref{enum:initialaffineproj} of Lemma \ref{l:basicinitial}, we have
$\inn_{\bfw'}(I') = \inn_{(0,\bfw')}((I')^h)|_{x_0=0}$, and so there
is some $k \geq 0$ such that $x_0^k \bfx^\bfu \in
\inn_{(0,\bfw')}((I')^h) = \inn_{\prec}((I')^h)$.  Let $m = k +
|\mathbf{u}|$. By Lemma \ref{l:stdmonosbasis}, the monomials of degree $m$ not
in $\{ \inn_{\prec}(f) : f \in (I')^h_m \}$ form a basis $B$ for
$\underline{\Mat}((I')^h_m)$, and so there is a polynomial $f \in
(I')^h_m$ such that $\supp(f) \cap \inn_{\prec}((I')^h) = \{x_0^k
\bfx^\bfu\}$, corresponding to the fundamental circuit of $x_0^k
\bfx^\bfu$ over $B$.
We then have $\inn_{(0,\bfw')}(f) = \inn_{\prec}(f) = x_0^k \bfx^\bfu$. 
Since $\prec$ is reverse-lexicographic with $x_0$ the last variable, it follows that all monomials in $\supp(f)$ are divisible by $x_0^k$. As $(I')^h$ is saturated with respect to $x_0$, this implies that $f = x_0^k g$ with $g \in (I')^h$. 
We conclude that $\bfx^\bfu = \inn_{(0,\bfw')}(g) \in \inn_{(0,\bfw')}((I')^h_{\leq d})$, and thus $\bfx^\bfu \in \inn_{\bfw'}(I'_{\leq d})$, as desired.

Now, for any $d \geq 0$, let $B'_d$ be the set of monomials of degree at most $d$ not in 
$\inn_{\mathbf{w}'}(I')$. Since  $(\inn_{\bfw'}(I'))_{\leq d} = \inn_{\bfw'}(I'_{\leq d})$, 
by \cite{TropicalIdeals}*{Lemma 3.3} the set $B'_d$ is a basis of the matroid $\underline{\Mat}(I'_{\leq d})$.
Consider the set of monomials 
$B_d := \{ \bfx^\bfu x_n^k : \bfx^\bfu \in B'_d \text{ and } 0 \leq k \leq d - |\bfu|\}$.
As $\inn_{\mathbf{w}'}(I') = (\inn_{\mathbf{w}}(I))|_{x_n=0}$, none of the monomials
in $B_d$ can be contained in $\inn_\bfw(I)$, and thus $B_d$ is an independent set of the matroid $\uMat(I_{\leq d})$.
Note that $B_{d+1}$ is the disjoint union $B_{d+1}= B'_{d+1} \sqcup x_n B_d$. We then have
$H_{I'}(d+1) = |B'_{d+1}| = |B_{d+1}| - |B_d|$.
For $d \gg 0$, the function $H_{I'}(d)$ agrees with a polynomial on $d$ of degree $\dim(I')$ by \cite{TropicalIdeals}*{Proposition 3.8}.  
It follows that for $d \gg 0$, the sequence $(|B_d|)_{d \geq 0}$ agrees with a polynomial of degree $\dim(I') + 1$.
Since $H_I(d) \geq |B_d|$, the Hilbert polynomial of $I$ is a polynomial of degree at least $\dim(I') + 1$,
and thus $\dim(I) \geq  \dim(I') + 1$, as claimed.
\end{proof}

\begin{remark}
The strict inequality 
$\dim(I|_{x_n=a})<\dim(I)-1$ is possible.  An example is given by $I=
\trop(\langle x_1-1, x_2-1 \rangle \cap \langle x_3-1 \rangle) \subseteq \Rbar[x_1^{\pm 1},x_2^{\pm 1},x_3^{\pm 1}]$, which has
dimension two.  By Remark \ref{r:realizableslicing}, 
for any $a \in \RR$ we have $I|_{x_3=a} = \trop(\langle x_1-1,x_2-1 \rangle)
\subseteq \Rbar[x_1^{\pm 1},x_2^{\pm 1}]$, which is zero-dimensional.
\end{remark}

We now prove that the dimension of a tropical ideal agrees with the dimension
of its variety.

\begin{theorem} \label{t:dimiscorrect}
Let $I$ be a tropical ideal in $\Rbar[x_1^{\pm 1},\dots,x_n^{\pm 1}]$
of dimension $d$. Then $V(I) \subseteq \RR^n$ is a $d$-dimensional
polyhedral complex.
\end{theorem}

We note that this complex need not be pure, so may have maximal cells of dimension less than $d$.

\begin{proof}
We first prove the equality
\begin{equation}\label{eq:dimension}
 \dim(V(I)) = \max \left( |S| : S \subseteq \{1,\dotsc,n\} \text{ and } 
I \cap \Rbar[x_j^{\pm 1} : j \in S] = \{ \infty \} \right),
\end{equation}
where by convention we set this maximum 
to be $-1$ if $I \cap \Rbar = \Rbar$.
Denote by $e$ the expression on the right hand side.
To show that $\dim(V(I)) \leq e$, take a cell $\sigma$ of $V(I)$ of maximal dimension.
There exists $S \subseteq \{1,\dotsc,n\}$ of size $\dim(\sigma) = \dim(V(I))$ 
such that the projection $\pi(\sigma)$ of $\sigma$ onto the coordinate subspace
$\RR^S \subseteq \RR^n$ is injective. 
It follows that $I$ does not contain any polynomial in
$\Rbar[x_j^{\pm 1}:j\in S]$ other than $\infty$,
as otherwise $\pi(\sigma) \subseteq \RR^S$ would have codimension
at least 1. 

We now prove that $\dim(V(I)) \geq e$ by induction on $\dim(V(I))$.  
If $\dim(V(I)) = -1$ then $V(I) = \emptyset$, and by the weak
Nullstellensatz for tropical ideals~\cite{TropicalIdeals}*{Corollary 5.17},
 we also have $e = -1$.
For the induction step, suppose $\dim(V(I)) \geq 0$, 
and fix $S \subseteq \{1,\dotsc,n\}$ of size $\dim(V(I))+1$ and $i \in S$. 
Choose $a \in \RR$ such that the hyperplane $\{x_i=a\}$ intersects
$V(I)$ transversely at all points of intersection (or their intersection is empty).  
This is possible
since $V(I)$ is the support of a finite polyhedral complex.  By
Proposition~\ref{p:transverseintersection} we have $V(I|_{x_i=a}) =
\pi(V(I) \cap \{x_i=a\})$, where $\pi$ is the projection onto the
coordinates other than $i$, and so $V(I|_{x_i=a})$ has dimension at
most $\dim(V(I))-1$. As $|S \setminus \{i\}| = \dim(V(I))$, the
induction hypothesis then guarantees that there is a polynomial $f
\neq \infty$ in $I|_{x_i=a} \cap \Rbar[x_j^{\pm 1} : j \in S \setminus
  \{i\}]$.  The polynomial $f$ must be the specialization $f=
g|_{x_i=a}$ for a polynomial $g \in I \cap \Rbar[x_j^{\pm 1} : j \in
  S]$, which shows that $I \cap \Rbar[x_j^{\pm 1} : j \in S] \neq
\{\infty\}$.  This completes the proof of \eqref{eq:dimension}.

We now show that \eqref{eq:dimension} implies that $\dim(I) =
\dim(V(I))$.  Take $S \subseteq \{1,\dotsc,n\}$ of size $\dim(V(I))$
such that $I \cap \Rbar[x_j^{\pm 1} : j \in S] = \{ \infty \}$.  Let
$J \subseteq \Rbar[x_0, \dots, x_n]$ be the homogenization of 
$I \cap \Rbar[x_1,\dots,x_n]$.  Fix $\prec$ be a term order on $\Rbar[x_0, \dots, x_n]$ with
the property that $x_i \prec \mathbf{x}^{\mathbf{u}}$ whenever both $i \not \in
\{0\} \cup S$ and $u_j=0$ for all $j \not \in \{0\} \cup S$.
Any polynomial in $J$ involves a variable $x_i$ with $i \not \in \{0\} \cup S$,
and thus any monomial
$\bfx^\bfu \in \inn_{\prec}(J)$ must involve a variable $x_i$ with $i \not \in \{0\} \cup S$. 
This means that $\inn_{\prec}(J) \cap \Rbar[x_j : j \in \{0\} \cup S] = \{ \infty \}$,
and so 
\[ \max \left( |S'| : S' \subseteq \{0,\dotsc,n\} \text{ and} 
\inn_{\prec}(J) \cap \Rbar[x_j : j \in S'] = \{ \infty \} \right) \geq \dim(V(I)) +1.\]
In fact, this last inequality must be an equality, as any subset $S'$ such that
$\inn_{\prec}(J) \cap \Rbar[x_j : j \in S'] = \{ \infty \}$ must
also satisfy $J \cap \Rbar[x_j : j \in S'] = \{ \infty \}$.
It follows that $\dim(I) = \dim(V(I))$, as $\dim(I) = \dim(J) = \dim(\inn_{\prec}(J))$ 
by Lemma \ref{l:stdmonosbasis},
and the (projective) dimension of a monomial ideal $M$ is the size of the largest
subset $S'$ with $M \cap \Rbar[x_i : i \in S'] = \{ \infty\}$ minus one.
\end{proof}

\begin{remark}
One consequence of Theorem~\ref{t:dimiscorrect} is that the definition
of dimension we use here, as the degree of the Hilbert polynomial,
essentially agrees with a naive notion of Krull dimension, as in
\cite{JooMincheva2}.  This follows from the proof of Theorem 7.2.1 of
\cite{KalinaThesis}, which shows that for an arbitrary ideal in
$\Rbar[x_1,\dots,x_n]$, if $V(I)$ is the support of an $\mathbb
R$-rational polyhedral complex then 
$\dim(\Rbar[x_1,\dots,x_n]/\mathcal B(I) )$ is one more than the maximal
dimension of a cell in the complex.  The ``one more'' comes from the
fact that $\dim(\Rbar)=1$ in this theory.  However, as shown in
\cite{JooMincheva}, varieties of prime ideals are not flexible enough
to play the role of irreducible varieties in a tropical scheme theory.
\end{remark}

We next consider the effects of changes of coordinates on
$\Rbar[x_1^{\pm 1},\dots,x_n^{\pm 1}]$.  This is essentially identical
to the realizable case (see \cite{TropicalBook}*{Lemma 2.6.10 and
  Corollary 3.2.13}).

Let $A \in \GL(n,\mathbb Z)$ be an $n \times n$ invertible matrix with
integer entries, and fix $\mathbf{\lambda} \in \mathbb R^n$.  Let
$\phi^* \colon \Rbar[x_1^{\pm 1},\dots,x_n^{\pm 1}] \rightarrow
\Rbar[x_1^{\pm 1},\dots,x_n^{\pm 1}]$ be the semiring homomorphism
given by $x_i \mapsto \lambda_i \mathbf{x}^{\mathbf{a}_i}$, where
$\mathbf{a}_i$ is the $i$th column of $A$.  Such homomorphisms are
precisely the automorphisms of $\Rbar[x_1^{\pm 1},\dots,x_n^{\pm 1}]$.
Write $\phi^*_A$ for the homomorphism given by $\phi^*_A(x_i)=\mathbf{x}^{\mathbf{a}_i}$.  
We denote by $\trop(\phi) \colon \mathbb
R^n \rightarrow \mathbb R^n$ the linear map given by
$\trop(\phi)(\mathbf{w}) = A^T \mathbf{w} + \mathbf{\lambda}$.

\begin{lemma} \label{l:changeofcoords}
Let $I \subseteq \Rbar[x_1^{\pm 1},\dots,x_n^{\pm 1}]$ be a tropical
ideal, and let $I'={\phi^*}^{-1}(I)$.  Then $I'$ is also a tropical
ideal, and
$$\phi^*_A(\inn_{\trop(\phi)(\mathbf{w})}(I')) =
\inn_{\mathbf{w}}(I)$$ for all $\mathbf{w} \in \mathbb R^n$.
As a consequence,  we have 
$$V(I') = \trop(\phi)(V(I)).$$
\end{lemma}

\begin{proof}

We first show that $I'$ is a tropical ideal.  Write $\phi^* =
\phi^*_{\lambda} \circ \phi^*_A$, where $\phi^*_{\lambda}(x_i)
=\lambda_i x_i$.  Let $E$ be a finite collection of monomials in
$\Rbar[x_1^{\pm 1},\dots,x_n^{\pm 1}]$, and let $E' = \phi^*_A(E)$.
Since $I$ is a tropical ideal, the polynomials in $I$ with support in
$E'$ are the vectors of a valuated matroid.  Applying the invertible
map ${\phi^*_{\lambda}}^{-1}$ to this collection produces an
equivalent valuated matroid on the ground set $E'$.  Since $\phi_A^*$
is an injection, the collection of polynomials in $I'$ with support in
$E$ define the same matroid.  This shows that $I'$ is a tropical
ideal.

Fix $f = \bigoplus c_\mathbf{u} \mathbf{x}^\mathbf{u} \in I'$, so $\phi^*(f) = \bigoplus
c_{\mathbf{u}} \ttimes  (\mathbf{\lambda} \cdot \mathbf{u}) \ttimes \mathbf{x}^{A\mathbf{u}} \in I$.  We have 
\begin{align} \label{eqtn:tropphi} \inn_{\mathbf{w}}(\phi^*(f)) 
& = \bigoplus_{\mathbf{u} \,:\, c_\mathbf{u} + \lambda \cdot \mathbf{u} + \mathbf{w} \cdot (A\mathbf{u}) \text{ is minimal}} \mathbf{x}^{A\mathbf{u}} \\ &
  = \bigoplus_{\mathbf{u} \,:\, c_\mathbf{u} + (
\lambda + A^T\mathbf{w}) \cdot \mathbf{u} \text{ is minimal}} \mathbf{x}^{A\mathbf{u}}
\nonumber  \\
 & = \bigoplus_{\mathbf{u} \,:\, c_\mathbf{u} + \trop(\mathbf{\mathbf{w}}) \cdot \mathbf{u} \text{ is
      minimal}} \mathbf{x}^{A\mathbf{u}} \nonumber \\ 
& = \phi^*_A(\inn_{\trop(\phi)(\mathbf{w})}(f)). \nonumber
\end{align}
This implies that $\phi_A^*(\inn_{\trop(\phi)(\mathbf{w})}(I')) \subseteq
\inn_{\mathbf{w}}(I)$.  As $A$ is invertible, we also have
${\phi}_{A^{-1}}^*(\inn_{\mathbf{w}}(I)) \subseteq
\inn_{\trop(\phi)(\mathbf{w})}(I')$, since $\mathbf{w} =
\trop(\phi^{-1})(\trop(\phi)(\mathbf{w}))$, so
$\inn_{\trop(\phi)(\mathbf{w})}(I') = \inn_{\mathbf{w}}(I)$.  Thus
  $V(I') = \trop(\phi)(V(I))$, as required.
\end{proof}

We conclude this section with Theorem~\ref{t:projection}, which shows
that the projection of the variety of a tropical ideal is the variety
of the elimination ideal, as is the case for realizable tropical
ideals.  For this we first require the following result.  
Recall from \cite{TropicalIdeals}*{Theorem 5.9} that any
tropical ideal $I \subseteq \Rbar[x_1^{\pm 1},\dots,x_n^{\pm 1}]$ has
a finite tropical basis.  This is a set $\{f_1,\dots,f_s\} \subset I$
with the property that $V(I) = \bigcap_{i=1}^s V(f_i)$.

\begin{lemma}\label{l:tropicalbasis}
Let $I$ be a tropical ideal in $\Rbar[x_1^{\pm 1},\dots,x_n^{\pm 1}]$. 
If $\bfv \in \RR^n$ is not in $V(I)$ then there exists a finite tropical basis 
$\mathcal T \subset I$ such that for all $f \in \mathcal T$ we have $\bfv \notin V(f)$.
\end{lemma}
\begin{proof}
Fix $f \in I$ such that $\bfv \notin V(f)$.  The set $\RR^n \setminus
V(f)$ decomposes naturally as a union of finitely many (full-dimensional) open polyhedra,
so $\bfv$ lies in the interior of one such open polyhedron $P$.
Let $\mathcal T$ be a finite tropical basis for $I$.  For each
$g \in \mathcal T$ we can decompose $\mathbb R^n \setminus V(g)$ as a
union of open polyhedra.  After subdividing, we can thus write
$\mathbb R^n \setminus (V(I) \cup P) = \bigcup P_i$, where for each $P_i$ the 
closure $\overline{P_i}$ is an $n$-dimensional
pointed polyhedron, there is $g \in \mathcal T$ such that $P_i
\subseteq \mathbb R^n \setminus V(g)$ with $g(\mathbf{w}) =
\mathbf{a} \cdot \mathbf{w} +b$ for $\mathbf{w} \in P_i$, and $f$ is is linear on $P_i$.

Fix $i$, and set $Q= P_i$, with $Q \subseteq \mathbb R^n \setminus
V(g)$ for $g \in \mathcal T$.  Since $\overline{Q}$ is a pointed
polyhedron not containing $\mathbf{v}$, there is a
hyperplane $\{ \mathbf{w} \in \mathbb R^n : \mathbf{c} \cdot
\mathbf{w} +d= 0 \}$ with $\mathbf{c} \cdot \mathbf{w} +d > 0$ for
$\mathbf{w} \in \overline{Q}$ and $\mathbf{c} \cdot \mathbf{v}+d < 0$.  We may
assume that the face of $\overline{Q}$ minimizing $\mathbf{c}$ is a
vertex $\mathbf{p}$.

Consider the polynomial $g' := g \tplus
(d\bfx^{\mathbf c})^N f \in I$.
We claim that for sufficiently large $N$, the minimum in
$g'(\mathbf{w})$ is achieved at terms in $g$ for $\mathbf{w} \in Q$,
and at terms in $(d\bfx^{\mathbf c})^N f$ for $\mathbf{w}=\mathbf{v}$.  This implies that $Q
\subseteq \mathbb R^n \setminus V(g')$, and $\mathbf{v} \not \in
V(g')$.  It follows that the collection of $g'$ for $g \in \mathcal
T$, together with $f$, form the desired tropical basis for $I$.

To see the claim, since $f$ is linear on $Q$, there are $\mathbf{a}'$
and $b$ with $f(\mathbf{w}) = \mathbf{a}' \cdot \mathbf{w} +b'$ for
$\mathbf{w} \in Q$.  For $N \gg 0$ the face of $\overline{Q}$
minimizing $N\mathbf{c} + \mathbf{a}'-\mathbf{a}$ is still
$\mathbf{p}$, so $(N \mathbf{c} + \mathbf{a}' - \mathbf{a}) \cdot
\mathbf{w} \geq (N \mathbf{c} + \mathbf{a}' - \mathbf{a}) \cdot
\mathbf{p}$ for $\mathbf{w} \in Q$.   Since $\mathbf{c}  \cdot \mathbf{p}+d >0$, for $N \gg 0$ we have
\begin{align*} (N \mathbf{c} + \mathbf{a}' - \mathbf{a}) \cdot \mathbf{w} + Nd+b'-b
  & \geq (N \mathbf{c} + \mathbf{a}' - \mathbf{a}) \cdot \mathbf{p} + Nd+b'-b\\
  & = N(\mathbf{c} \cdot \mathbf{p}+d) + (\mathbf{a}'-\mathbf{a}) \cdot \mathbf{p} + b'-b \\
  & >0
\end{align*}  
for all $\mathbf{w} \in Q$. For such $N$,
$$g(\mathbf{w}) = \mathbf{a} \cdot \mathbf{w} + b < N(\mathbf{c} \cdot
\mathbf{w}+d) + \mathbf{a}' \cdot \mathbf{w}+b' = N(\mathbf{c} \cdot \mathbf{w}+d) + f(\mathbf{w}) $$ for
all $\mathbf{w} \in P$, which shows the first part of the claim.  For the
second part, note that any $N > (g(\mathbf{v})-f(\mathbf{v}))/(\mathbf{c}
\cdot \mathbf{v} +d)$ suffices.
\end{proof}

\begin{theorem} \label{t:projection}
Let $I$ be a tropical ideal in $\Rbar[x_1^{\pm 1},\dots,x_n^{\pm 1}, y_1^{\pm 1},\dots,y_m^{\pm 1}]$.
Then $$V(I \cap \Rbar[y_1^{\pm 1},\dots,y_{m}^{\pm 1}]) = \pi(V(I)),$$ 
where $\pi \colon \mathbb R^n \times \RR^m \rightarrow
\RR^m$ is the projection onto the second factor.
\end{theorem}

\begin{proof}
Since $V(I) \subseteq V(f)$ for all $f \in I \cap \Rbar[y_1^{\pm 1},\dots,y_{m}^{\pm 1}]$, 
where we regard $f$ as a polynomial in $\Rbar[x_1^{\pm 1},\dots,x_n^{\pm 1}, y_1^{\pm 1},\dots,y_m^{\pm 1}]$, we have $\pi(V(I)) \subseteq V(I \cap \Rbar[y_1^{\pm 1},\dots,y_{m}^{\pm 1}])$. 

To prove the reverse inclusion, suppose $\bfw \in
\mathbb R^m$ satisfies $\bfw \not \in \pi(V(I))$.
By tropically scaling the $y$ variables, we may assume that $\bfw = \bfzero$.
As $(\bfzero, \bfzero) \in \RR^n \times \RR^m$ is not in $V(I)$, 
by Lemma \ref{l:tropicalbasis} the ideal $I$ has a finite tropical
basis $\mathcal T \subset I$ with $(\bfzero, \bfzero) \notin V(f)$ for all $f \in \mathcal T$.
After multiplying each polynomial in $\mathcal T$ by a monomial we may assume that 
$\mathcal T$ consists of polynomials in $\Rbar[x_1,\dots,x_n,y_1,\dots,y_m]$.

Now consider the specializations $f|_{\bfy=\bfzero} \in \Rbar[x_1,\dots,x_n]$ 
for $f \in \mathcal T$. Since $\bfzero \not \in \pi(V(I))$, we claim that the
tropical prevariety $\bigcap_{f \in \mathcal T} V(f|_{\bfy=\bfzero}) \subseteq
\mathbb R^n$ is empty.  Indeed, if this prevariety contained a
point $\mathbf{w}' \in \RR^n$, then $(\bfzero,\mathbf{w}') \in V(f)$ for all $f \in
\mathcal T$, and so $(\bfzero,\mathbf{w}') \in V(I)$.
As $\bigcap_{f \in \mathcal T} V(f|_{\bfy=\bfzero})$ is empty, we can apply
the Tropical Primary Nullstellensatz of Grigoriev and Podolski
\cite{GrigorievPodolskii}*{Theorem 10}.  This says that there is a
tropical polynomial combination 
$$g = \bigoplus_{i=1}^s c_i \mathbf{x}^{\mathbf{u}_i} \ttimes 
{f_i}|_{\bfy=\bfzero} \in \Rbar[x_1,\dots,x_n]$$
where $f_i \in \mathcal T$ for all $i$, and for all monomials
$\mathbf{x}^{\mathbf{v}} \in \supp(g)$ there is a unique
$i=i(\mathbf{v})$ with the coefficient of $\mathbf{x}^{\mathbf{v}}$ in
$c_i \mathbf{x}^{\mathbf{u}_i} \ttimes {f_i}|_{\bfy=\bfzero}$  less than the
coefficient of $\mathbf{x}^{\mathbf{v}}$ in $c_j
\mathbf{x}^{\mathbf{u}_j} \ttimes {f_j}|_{\bfy=\bfzero}$ for $j \neq i(\mathbf{v})$,
and for $\mathbf{v} \neq \mathbf{v}'$ we have $i(\mathbf{v}) \neq
i(\mathbf{v}')$.  Note that these conditions imply that we can take
$s$ to be the number of terms of $g$.

Set $g_i = c_i \mathbf{x}^{\mathbf{u}_i} f_i \in
\Rbar[x_1,\dots,x_n,y_1,\dots,y_m]$.  We may regard $g_i$ as a polynomial in
$x_1,\dots,x_n$ with coefficients in $\Rbar[y_1,\dots,y_m]$.  Write
${g_i}_{\mathbf{v}} \in \Rbar[y_1,\dots,y_m]$ for the coefficient of
$\mathbf{x}^{\mathbf{v}}$ in $g_i$.  The monomials $\bfx^\bfv$ appearing in $g_i$
are a subset of the monomials appearing in $g$.  Fix an order on the
monomials $\mathbf{x}^{\mathbf{v}_1},\dots, \mathbf{x}^{\mathbf{v}_s}$
appearing in $g$.  After reordering, we may assume that the lowest
coefficient appearing in ${g_i}|_{\bfy=\bfzero}$ for any $i$ appears in
$g_s$, and is the coefficient of $\mathbf{x}^{\mathbf{v}_s}$, and in
general the lowest coefficient of $\mathbf{x}^{\mathbf{v}_j}$ in any
${g_i}|_{\bfy=\bfzero}$ occurs in ${g_j}|_{\bfy=\bfzero}$ for all $1 \leq j \leq s$.
Since the coefficient of $\mathbf{x}^{\mathbf{v}_s}$ in
${g_s}|_{\bfy=\bfzero}$ is the lowest coefficient of any ${g_j}|_{\bfy=\bfzero}$, it
is the lowest coefficient appearing in ${g_s}|_{\bfy=\bfzero}$.  The
assumption that $(\bfzero,\bfzero) \not \in V(f)$ for all $f \in \mathcal T$
then implies that $\bfzero \notin V({g_s}_{\mathbf{v}_s})$.

We now repeatedly apply Lemma~\ref{l:maineliminationstep} to do a form
of Gaussian elimination on this system of polynomials.  Applying
Lemma~\ref{l:maineliminationstep} to $g_1$ and $g_i$ for $i>1$, with
$\mathbf{u}=\mathbf{v}_1$ and $\mathbf{v}=\mathbf{v}_i$, we get a new
polynomial $g'_i \in I$ with ${g'_i}_{\mathbf{v}_1} = \infty$,
${g'_i}_{\mathbf{v}_i}(\bfzero) = {g_i}_{\mathbf{v}_i}(\bfzero)$, and
${g'_i}|_{\bfy=\bfzero} \geq {g_1}|_{\bfy=\bfzero} \tplus {g_i}|_{\bfy=\bfzero}$.  
Additionally, when $i = s$ we get that $\inn_{\mathbf{0}}({g'_s}_{\mathbf{v}_s}) =
\inn_{\mathbf{0}}({g_s}_{\mathbf{v}_s})$,
and so $\bfzero \notin V({g'_s}_{\mathbf{v}_s})$.
Note that the smallest coefficient in any $g'_i|_{\bfy=\bfzero}$ still occurs
in $g'_s$, and the coefficient of $\mathbf{v}_i$ in ${g'_i}|_{\bfy=\bfzero}$
is still smaller than that coefficient in ${g'_j}|_{\bfy=\bfzero}$ for $j>i$.  

We now replace $g_i$ by $g'_i$ for $i>1$, and iterate, with $g_2$
playing the role of $g_1$.  After $s-1$ iterations we obtain $g_s =
\mathbf{x}^{\mathbf{v}_s} h$, where $h \in \Rbar[y_1,\dots,y_m]$ satisfies 
$\bfzero \notin V(h)$.  Since $g_s \in I \subseteq \Rbar[x_1^{\pm
    1},\dots,x_n^{\pm 1}, y_1^{\pm 1},\dots,y_m^{\pm 1}]$, 
    we have $h \in I$, so $\bfzero \not \in V(I \cap
\Rbar[y_1^{\pm 1},\dots,y_m^{\pm 1}])$, as desired.  This proves the reverse inclusion.
\end{proof}

\begin{example}
Theorem~\ref{t:projection} is not true verbatim with
$\Rbar[x_1^{\pm 1},\dots,x_n^{\pm 1},y_1^{\pm 1},\dots,y_m^{\pm 1}]$ replaced by
$\Rbar[x_1,\dots,x_n,y_1,\dots,y_m]$.  
For example,  consider $I = \trop(\langle
x_0x_3-x_1x_2 \rangle ) \subseteq \Rbar[x_0,x_1,x_2,x_3]$.  Then $V(I)
= \{ \mathbf{w} \in \Rbar^4 : w_0+w_3=w_1+w_2 \}$.  We have $I
\cap \Rbar[x_0,x_1,x_2] = \{ \infty \}$, and so $V(I
\cap \Rbar[x_0,x_1,x_2]) = \Rbar^3$, but $\{
(\infty,w_1,w_2) : w_1, w_2 \in \mathbb R \}$ is not in $\pi(V(I))$.
In this example we have $V(I
\cap \Rbar[x_0,x_1,x_2])$ equal to the closure of $\pi(V(I))$, as in
the classical case.
\end{example}


\section{Degrees of zero-dimensional ideals} \label{s:degreezero}

For a zero-dimensional subscheme of $\mathbb P^n$, its degree equals
the sum of the multiplicities of the points in its variety.  
In this section we extend this to
tropical ideals.  The proof is more complicated than in the classical case, as we do not (yet?) have primary decomposition available as a tool.

\begin{definition} \label{d:degree}
  The {\defbold degree} of a $d$-dimensional homogeneous tropical ideal
  $I \subseteq \Rbar[x_0,\dots,x_n]$ is $d!$ times the leading coefficient of the
  Hilbert polynomial of $I$.
  \end{definition}

When $I$ is zero-dimensional, the Hilbert polynomial of $I$ is a constant, 
and the degree is equal to that
constant.  For a zero dimensional tropical ideal $I$ in
$\Rbar[x_1,\dots,x_n]$ or $\Rbar[x_1^{\pm 1},\dots,x_n^{\pm 1}]$ we
define the degree to be the degree of the homogenization $I^h$, or of
the homogenization of $I \cap \Rbar[x_1,\dots,x_n]$ respectively.  We
make the same definitions with $\Rbar$ replaced by $\mathbb B$.  This
has the following equivalent formulation.

\begin{lemma} \label{l:degreeforzerodim}
Let $I$ be a zero-dimensional tropical ideal in $\Rbar[x_1,\dots,x_n]$ or
$\Rbar[x_1^{\pm 1},\dots,x_n^{\pm 1}]$.  Then the degree of $I$ is the
maximum size of a finite collection $E$ of monomials not containing the support
of any polynomial in $I$: $$\deg(I) = \max(\rk(\underline{\Mat}(I_E))).$$
\end{lemma}

\begin{proof}
The degree of $I \subseteq \Rbar[x_1^{\pm 1},\dots,x_n^{\pm 1}]$ is by
definition the same as the degree of $J= I \cap \Rbar[x_1,\dots,x_n]$.
If $E$ is a collection of monomials in $\Rbar[x_1^{\pm
    1},\dots,x_n^{\pm 1}]$, there is a polynomial in $I \subseteq
\Rbar[x_1^{\pm 1},\dots,x_n^{\pm 1}]$ with support contained in $E$ if
and only if there is some monomial $\mathbf{x}^{\mathbf{u}}$ with
$\mathbf{x}^{\mathbf{u}}E \subseteq \Rbar[x_1,\dots,x_n]$ and a polynomial in
$J := I \cap \Rbar[x_1,\dots,x_n]$ with support in
$\mathbf{x}^{\mathbf{u}}E$. So it suffices to consider $J \subseteq
\Rbar[x_1,\dots,x_n]$.

The degree of $J$ is the degree of $J^h$, which is the rank of
$\uMat(J^h_d)$ for $d \gg 0$.  If $E$ is a collection of monomials in
$\Rbar[x_1,\dots,x_n]$, let $E^d = \{
x_0^{d-|\mathbf{u}|}\mathbf{x}^{\mathbf{u}} : \mathbf{x}^{\mathbf{u}}
\in E \}$ for $d \gg \max_{\mathbf{x}^{\mathbf{u}} \in E}
|\mathbf{u}|$.  If there is no polynomial in $J$ with support in $E$,
then there is no polynomial in $J^h$ with support in $E^d$, so
$\deg(J^h) \geq |E|$.  Conversely, let $E'$ be a basis for
$\uMat(J^h_d)$ for fixed $d \gg 0$, so there is no polynomial in $J^h$
with support in $E'$.  Set $E = E'|_{x_0=0}$.  Then $|E'| = |E|$, 
and since $J^h|_{x_0 = 0} = J$, we have that $E$
is a collection of monomials in $\Rbar[x_1,\dots,x_n]$ not containing
the support of any polynomial in $J$, thus the maximum such $|E|$
is at least $\deg(J^h)$.
\end{proof}

The notion of {\em multiplicity} is central in tropical geometry.  We
now define it for zero-dimensional ideals.  A general definition
follows in Definition~\ref{d:balancing}. 
Note that any initial ideal of a zero-dimensional tropical ideal is 
also zero-dimensional; this follows, for instance, from Theorem \ref{t:dimiscorrect}
and Proposition \ref{p:star}.

\begin{definition} \label{d:zerodimmult}
Let $I$ be a zero-dimensional tropical ideal in $\Rbar[x_1^{\pm
    1},\dots,x_n^{\pm 1}]$, and let $\mathbf{w} \in V(I)$.  The {\defbold
  multiplicity} $\mult_{V(I)}(\bfw)$ of $V(I)$ at $\mathbf{w}$ is the degree of
the zero-dimensional initial ideal $\inn_{\mathbf{w}}(I) \subseteq \mathbb B[x_1^{\pm 1},\dots,x_n^{\pm 1}]$.  

When $I$ is a zero-dimensional tropical ideal in
$\Rbar[x_1,\dots,x_n]$ or a zero-dimensional homogeneous tropical ideal in
$\Rbar[x_0,\dots,x_n]$, we define the multiplicity of $V(I)$
at $\mathbf{w} \in \mathbb R^n$ or $[\mathbf{w}]\in \trop(\mathbb
P^n)$ with $\mathbf{w} \in \mathbb R^{n+1}$ to be the degree of the
saturation $(\inn_{\mathbf{w}}(I):(\prod_{i=0}^n x_i)^{\infty})$.  We will not
need the case where $\mathbf{w}$ has infinite coordinates in this
paper.
\end{definition}

We first note some basic properties of the  multiplicity and degree.
Recall from \S \ref{d:dimensionprojection} that an automorphism of $\Rbar[x_1,\dots,x_n]$ is given by a
function $\phi^*$ 
of the form $\phi^*(x_i) = \lambda_i
\mathbf{x}^{\mathbf{a}_i}$, where $\lambda_i \in \RR$ for all $i$
and the $n \times n$ matrix $A$ with
columns $\mathbf{a}_i$ has determinant $\pm 1$.  This has
tropicalization $\trop(\phi)$ given by $\trop(\phi)(\mathbf{w}) = A^T
\mathbf{w} + \lambda$.

\begin{proposition}

  \label{p:basicdeg}
  \leavevmode
\begin{enumerate}
\item \label{enum:degmap} Let $\phi^* \colon \Rbar[x_1^{\pm
    1},\dots,x_n^{\pm 1}] \rightarrow \Rbar[y_1^{\pm 1},\dots,y_n^{\pm
    1}]$ be the semiring homomorphism given by $\phi^*(x_i) =
  \mathbf{y}^{\mathbf{a}_i}$, where the $n \times n$ matrix $A$ with
  columns $\mathbf{a}_i$ has rank $n$. If $J$ is a tropical ideal in
  $\Rbar[x_1^{\pm 1},\dots,x_n^{\pm 1}]$ then $J_{\phi} :=
  \phi^*(J)\Rbar[y_1^{\pm 1},\dots,y_n^{\pm 1}]$ is a tropical ideal
  as well. In addition, if $J$ is zero-dimensional then so is
  $J_\phi$, and $\deg(J_{\phi})=|\det(A)|\deg(J)$.  The same holds with $\Rbar$ replaced by $\mathbb B$.

\item \label{enum:automorphismdeg} When $I$ is a zero-dimensional
  ideal in $\Rbar[x_1^{\pm 1},\dots,x_n^{\pm 1}]$, the degree of $I$
  and the multiplicity of a point in $V(I)$ are invariant under
  automorphisms of $\Rbar[x_1^{\pm 1},\dots,x_n^{\pm 1}]$.
  Explicitly, if $\phi^* \colon \Rbar[x_1^{\pm 1},\dots,x_n^{\pm 1}]
  \rightarrow \Rbar[x_1^{\pm 1},\dots,x_n^{\pm 1}]$ is an
  automorphism and $I' = {\phi^*}^{-1}(I)$, then
  $\deg(I') = \deg(I)$, and $\mult_{V(I)}(\mathbf{w}) =
  \mult_{V(I')}(\trop(\phi)(\mathbf{w}))$.

\item \label{enum:homogenizationmultiplicity} For a zero-dimensional 
tropical ideal $I \subseteq
  \Rbar[x_1^{\pm},\dots,x_n^{\pm 1}]$, let $J \subseteq
  \Rbar[x_0,\dots,x_n]$ be the homogenization of $I \cap
  \Rbar[x_1,\dots,x_n]$.  Then for any $\mathbf{w} \in V(I)$, we have
  $\mult_{V(I)}(\mathbf{w}) = \mult_{V(J)}([0:\mathbf{w}])$.
\end{enumerate}
\end{proposition}
\begin{proof}

\begin{enumerate}

\item We first show that the ideal $\phi^*(J) \Rbar[y_1^{\pm
    1},\dots,y_n^{\pm 1}]$ is a tropical ideal.  Note that this ideal
  is homogeneous with respect to the grading of $\Rbar[y_1^{\pm
      1},\dots,y_n^{\pm 1}]$ by the cokernel $G$ of the matrix $A$,
  which is a finite abelian group of size $|\det(A)|$.  Explicitly, let
  $L$ be the sublattice of $\mathbb Z^n$ spanned by the columns
  $\mathbf{a}_i$ of $A$.  Grade $\Rbar[y_1^{\pm 1},\dots,y_n^{\pm 1}]$
  by $\deg(y_i)= \mathbf{e}_i + L \in G = \mathbb Z^n/L$.  Since
  $\phi^*(f)$ has degree zero for every $f \in J$, the ideal $J_{\phi}$ is
  homogeneous with respect to this grading.  In addition, note that if
  $\deg(\mathbf{y}^{\mathbf{u}}) = \mathbf{0}$ then
  $\mathbf{y}^{\mathbf{u}} = \phi^*(\mathbf{x}^{\mathbf{v}})$ for some
  monomial $\mathbf{x}^{\mathbf{v}}$, so
  $\mathbf{y}^{\mathbf{u}}\phi^*(f) = \phi^*(\mathbf{x}^{\mathbf{v}}f)$,
  and thus $(J_{\phi})_{\mathbf{0}} = \phi^*(J)$.  This also shows that
  $(J_{\phi})_{\gamma} = \mathbf{y}^{\mathbf{u}} \phi^*(J)$ for any $\gamma \in G$
  and $\mathbf{y}^{\mathbf{u}}$ with $\deg(\mathbf{y}^{\mathbf{u}}) =
  \gamma$.

Now suppose that $f, g \in J_{\phi}$ with
$[f]_{\mathbf{y}^{\mathbf{u}}} = [g]_{\mathbf{y}^{\mathbf{u}}} <
\infty$.  Since $J_{\phi}$ is homogeneous with respect to the
$G$-grading, to prove that it is a tropical ideal it is enough to
show the elimination axiom holds for homogeneous polynomials $f$ and
$g$ of the same degree $\gamma := \deg(\bfy^\bfu)$.  By factoring
out $\mathbf{y}^{\mathbf{u}}$, we may assume that
${\mathbf{u}} = \mathbf 0$ and $\deg(f)=\deg(g)=\bfzero$.  This means
that $f= \phi^*(f')$ and $g = \phi^*(g')$ for some $f',g' \in J$, with
the coefficient of ${\bfx^\mathbf 0}$ equal in $f'$ and $g'$.  Since $J$ is a tropical
ideal, there is $h' \in J$ satisfying $\elim{h'}{\bfx^{\mathbf 0}}{f'}{g'}$.
Set $h = \phi^*(h')$.  Then $\elim{h}{\bfy^{\bf0}}{f}{g}$,
since all monomials appearing in $f,g$, and $h$ have degree $\bf0$. 
This shows that $J_{\phi}$ is a tropical ideal.  

Since $J_{\phi}$ is homogeneous, for any collection of monomials $E$
in $\Rbar[y_1^{\pm 1},\dots,y_n^{\pm 1}]$ the matroid
$\Mat(J_{\phi}|_E)$ is the direct sum $\bigoplus_{\gamma \in G}
\Mat(J_{\phi}|_{E_\gamma})$, where $E_\gamma$ is the collection of
monomials in $E$ of degree $\gamma$.  In addition, if $E$ is a
collection of monomials of degree $\gamma$, so $E =
\mathbf{y}^{\mathbf{v}}\phi^*(E')$ for a collection of monomials $E'$
in $\Rbar[x_1^{\pm 1},\dots,x_n^{\pm 1}]$, then $\Mat(J_{\phi}|_E)$ is
isomorphic to $\Mat(J|_{E'})$.  Lemma~\ref{l:degreeforzerodim} then
implies that the largest set of monomials $E$ not supporting a
polynomial in $J_\phi$ has size $|G|\deg(J) = |\det(A)| \deg(J)$.  The
proof is verbatim the same in the case that $\Rbar$ is replaced by
$\mathbb B$.

\item Let $\phi^* \colon \Rbar[x_1^{\pm 1},\dots,x_n^{\pm 1}]
  \rightarrow \Rbar[x_1^{\pm 1},\dots,x_n^{\pm 1}]$ be an
  automorphism, given by $\phi^*(x_i) = \lambda_i
  \mathbf{x}^{\mathbf{a}_i}$ for $1 \leq i \leq n$.  Set $I' =
         {\phi^*}^{-1}(I)$.  The fact that $\deg(I')=\deg(I)$ follows
         from Lemma~\ref{l:degreeforzerodim}.  By
         Lemma~\ref{l:changeofcoords}, if $\mathbf{w} \in V(I)$ then
         $\trop(\phi)(\mathbf{w}) \in V(I')$.
         In addition, $\inn_{\trop(\phi)(\mathbf{w})}(I') =
         \phi^{*-1}_A(\inn_{\mathbf{w}}(I))$.
         Lemma~\ref{l:degreeforzerodim} then implies that the degree
         of $\inn_{\trop(\phi)(\mathbf{w})}(I')$ and
         $\inn_{\mathbf{w}}(I)$ are the same, as the automorphism
         $\phi^{*-1}_A$ does not change the maximal size of a
         collection of monomials not containing the support of a
         polynomial in either ideal.  As these degrees are the
         multiplicities by definition, the result follows.

\item Set $I' = I \cap \Rbar[x_1,\dots,x_n]$, so that $J = (I')^h$.
  By definition, $\mult_{V(J)}([0:\mathbf{w}])$ is the degree of
  $J_{in}:=(\inn_{(0,\mathbf{w})}(J): (\prod_{i=0}^n x_i)^{\infty})$.
  Since $J_{in}$ is saturated with respect to $x_0$, it is the
  homogenization with respect to $x_0$ of $(J_{in})|_{x_0=0} \subseteq \Rbar[x_1,\dots,x_n]$, so
  $\deg(J_{in}) = \deg(J_{in}|_{x_0=0})$ by the definition of degree.
  Note that $J_{in}|_{x_0=0} = (\inn_{(0,\mathbf{w})}(J)|_{x_0=0} :
  (\prod_{i=1}^n x_i)^{\infty})$.  Thus, by
  Part~\ref{enum:initialaffineproj} of Lemma~\ref{l:basicinitial},
  $J_{in}|_{x_0=0} = (\inn_{\mathbf{w}}(I') : (\prod_{i=1}^n
  x_i)^{\infty})$.  Part~\ref{enum:initialtorustoaffine} of
  Lemma~\ref{l:basicinitial} then implies that $(\inn_{\mathbf{w}}(I')
  : (\prod_{i=1}^n x_i)^{\infty}) = \inn_{\mathbf{w}}(I) \cap \mathbb
  B [x_1,\dots,x_n]$.  Thus $\deg(J_{in}) = \deg(\inn_{\mathbf{w}}(I))
  = \mult_{V(I)}(\mathbf{w})$ as required. \qedhere
\end{enumerate}
\end{proof}

We now show that Definition~\ref{d:zerodimmult} agrees with the multiplicity of a root of a 
univariate polynomial.  Many different tropical
polynomials give rise to the same function from $\Rbar$ to $\Rbar$,
but given a polynomial $f$ there is a minimum possible choice for the
coefficients of a polynomial giving rise to the same function. 
We call the polynomial with these coefficients the
{\defbold convexification} of $f$.  

\begin{definition}
For a polynomial $f \in \Rbar[x^{\pm 1}]$ or $\Rbar[x]$, we can factor
the convexification of $f$ as $\alpha \ttimes \prod_{i=1}^s (x \tplus
w_i)^{m_i}$, with $\alpha \in \mathbb R$, $w_i \in \Rbar$, and $m_i
\in \mathbb N$; see, for example \cite{GriggManwaring}.  
The {\defbold multiplicity}  $\mult_{w_i}(f)$ of $f$ at $w_i$ 
equals $m_i$.  For $w \in \Rbar$ with $w \neq w_i$ for all $i$, we set
$\mult_w(f)=0$.
\end{definition}

\begin{example}
The convexification of $f = x^3 \tplus 1 \ttimes x^2 \tplus x \tplus
1$ is equal to $x^3 \tplus x^2 \tplus x \tplus 1 = (x \tplus 0)^2 \ttimes (x
\tplus 1)$.  The multiplicity of $f$ at $0$ thus equals $2$.
\end{example}

The following proposition shows that, while tropical ideals in one variable 
might not be finitely generated, they still behave like 
principal ideals.

\begin{proposition} \label{p:principalideal} 
Let $I$ be a tropical ideal in $\Rbar[x^{\pm 1}]$.
There exists $h \in I$ such that $V(I) = V(h)$, 
and the multiplicity of $V(I)$ at any point $w \in \RR$ equals
$\mult_w(h)$. In addition, for every $f \in I$ the 
convexification of $h$ divides the convexification of $f$.
\end{proposition}

\begin{proof}
Fix $h \in I \cap \Rbar[x]$ with $\deg(h)$ minimal. 
We first prove that for any $f \in I$ and $w \in \RR$, the
multiplicity of $V(f)$ at $w$ is at least the multiplicity of $V(h)$
at $w$. 
Denote by 
$r_1<r_2<\dots<r_j$ the points of $V(h)$, which we call the roots of $h$.  
 The claim is trivially true for $w < r_1$, and indeed
for any $w$ not equal to one of the roots $r_i$.  Suppose now that the
claim is true for all $w<r_l$ for some root $r_l$.  After replacing
$x$ by $r_l \ttimes x$ we may assume that $r_l=0$.  If the claim is
false for $r_l$, then there is $g \in I \cap \Rbar[x]$ with $g \neq \infty$ and
$\mult_0(g)<\mult_0(h)$.  Choose such a $g$ of minimal degree.  We
have $\deg(g)>\deg(h)$, as otherwise we could eliminate $x^{\deg(h)}$
from $h$ and a suitable multiple of $g$ to get a polynomial in $I$ of lower
degree than $\deg(h)$.  Write $h = \bigoplus a_i x^i$, and $g = \bigoplus b_i x^i$.  After
scaling we may assume that $h(0)=g(0)=0$.  This means that $a_i,b_i
\geq 0$.  Let $j= \max \{ i : a_i=0 \}$, and $k= \max \{ i: b_i = 0
\}$.  We have $\deg(h) = j+ \sum_{i=1}^{l-1} \mult_{r_i}(h)$.  The
leading coefficient of $h$ is
$\lc(h) =  - \sum_{i=1}^{l-1} r_i \mult_{r_i}(h),$ where if $l=1$ we
have the empty sum, so the coefficient is $0$. 
An analogous statement holds for $g$.  Our assumption implies that 
$\mult_{r_i}(g)\geq \mult_{r_i}(h)$ for all $i<l$.  This means that
$\deg(g)-k \geq \deg(h) - j$, and $\lc(g) \geq \lc(h)$.

Set $h' = {(\lc(g)-\lc(h)) \ttimes x^{\deg(g)-\deg(h)} \ttimes h}$,
and let $f = \bigoplus c_ix^i \in I$ be an elimination $\elim{f}{x^{\deg(g)}}{g}{h'}$.
We have $\deg(f)<\deg(g)$.  If $\lc(g)>\lc(h)$, then
$c_i=0$ if and only if $b_i=0$, so $\mult_0(f)=\mult_0(g)<\mult_0(h)$,
contradicting our choice of $g$ to have minimal degree.  If
$\lc(g)=\lc(h)$, then we also have $\deg(g)-k=\deg(h)-j$.  Note that
$c_{k-\mult_0(h)}=0$, as $0=a_{j-\mult_0(h)}<b_{k-\mult_0(h)}$, and
$c_i>0$ for $i<k-\mult_0(h)$.  Let $l=\max \{ i: c_i =0 \}$.  By the argument above applied to $f$ we have $\deg(f)-l \geq \deg(h)-j = \deg(g)-k$, so since
$\deg(f)<\deg(g)$, we must have $l<k$.  But this means that
$\mult_0(f)<\mult_0(h)$, contradicting our choice of $g$ to have
minimal degree.  From these contradictions we conclude that $g$ does
not exist, so the claim is true for $w=r_l$, and thus for all $w$.
This shows that the 
convexification of $h$ divides the convexification of every other
polynomial in $I$.  

The multiplicity of $V(I)$ at $w$ is the smallest degree of a
polynomial in $\inn_w(I) \cap \mathbb B[x]$.  For a polynomial $f \in
I$ with $\mult_w(f)=m$, the convexification of $\inn_w(f)$ has the
form $x^{b} (x \tplus 0)^m$, for some $b \leq \deg(f)-m$, so
$(x \tplus 0)^m \in \inn_w(I)$.  Thus the minimal degree polynomial in
$\inn_w(I) \cap \mathbb B[x]$ has degree $\min \{ \mult_w(f) : f \in I
\}$.  The above argument shows that this equals $\mult_w(h)$.
\end{proof}

The following lemma is a key technical tool in the proof of Theorem~\ref{t:degissummult}.

\begin{lemma} \label{l:twovariable}
Let $I$ be a zero-dimensional homogeneous tropical ideal in
$\Rbar[x_0,\dots,x_n]$.  Then for all $0 \leq i<j\leq n$ there is a
nontrivial polynomial $f_{ij} \in I \cap \Rbar[x_i,x_j]$ that factors
as a product of linear factors. If $I$ is saturated with respect to
the product of the variables, then $f_{ij}$ can be chosen to have
degree at most $2\deg(I)$.
\end{lemma}

\begin{proof} 
  We can assume without loss of generality that $i = 0$ and $j = 1$.
  Fix $d \gg 0$, so that the Hilbert function of $I$ in degree $d$
  equals $\deg(I)$.  The set of monomials $E=\{ x_0^{d-i}x_1^i : 0
  \leq i \leq \deg(I) \}$ cannot be an independent set of $\uMat(I_d)$,
  so there is a polynomial in $I$ with support in $E$, and thus a
  polynomial $f \in I \cap \Rbar[x_0,x_1]$.  If $I$ is saturated
  with respect to the product of the variables, then we may divide by
  a power of $x_0$ to assume that $\deg(f) \leq \deg(I)$.
Let $f'$ be the convexification of $f$.  As described in \cite{TropicalIdeals}*{Example 4.13},
we have
$g := (f')^2 = f' \ttimes f \in I \cap \Rbar[x_0,x_1]$.  Since this polynomial is equal to its convexification, it factors into linear polynomials, so is the required polynomial.
\end{proof}

\begin{remark} \label{r:varietyintorus}
If $I \subseteq \Rbar[x_0, \dots, x_n]$ is a homogeneous zero-dimensional tropical ideal
saturated with respect to the product of the variables, then $V(I)
\subseteq \trop(\mathbb P^n)$ is contained in the tropical torus
$\mathbb R^{n+1}/\mathbb R \mathbf{1}$.  Indeed, suppose that
$[\mathbf{w}] \in V(I)$, where we may assume, after relabelling the
coordinates if necessary, that $w_0=0$.  For all $j>0$, the
polynomial $f_{0j}$ of Lemma~\ref{l:twovariable} can be chosen to
not be divisible by $x_j$, which implies that $w_j<\infty$, and thus
$\mathbf{w} \in \mathbb R^{n+1}$.
\end{remark}

The main result of this section is the following theorem, which is the
equivalent for tropical ideals of \cite{TropicalBook}*{Proposition
  3.4.13}.

\begin{theorem} \label{t:degissummult}
Let $I \subseteq \Rbar[x_1^{\pm 1},\dots,x_n^{\pm 1}]$ be a zero-dimensional tropical ideal.  Then $$\deg(I) = \sum_{\mathbf{w} \in V(I)} \mult_{V(I)}(\mathbf{w}).$$
\end{theorem}
Note that by Theorem \ref{t:dimiscorrect}, the sum on the right hand
side is finite.  The overall approach of the proof of
Theorem~\ref{t:degissummult} is to prove that the sums of the
multiplicities of points in certain polyhedral regions are the degrees
of related ideals.  The proof requires the following technical lemma,
which will be used to gradually increase the size of
the regions.

\begin{lemma} \label{l:lemmafordegissummult}
 Let $I \subseteq \Rbar[x_0,\dots,x_n]$ be a zero-dimensional
 homogeneous tropical ideal that is saturated with respect to the
 product of the variables.  Fix $\mathbf{w} \in \mathbb R^{n+1}$,
 $\sigma \subseteq \{1,\dots,n\}$ with $\sigma \neq \emptyset$, and $i
 \in \sigma$.  Write $m_{\sigma} = \prod_{l \in \{0,\dots,n\} \setminus \sigma} x_l$.

 For $\epsilon$ sufficiently small we have 
 $$\deg(\inn_{\mathbf{w}}(I):m_\sigma^{\infty}) =
 \deg(\inn_{\mathbf{w}}(I) : (m_{\sigma}x_i)^{\infty}) + \deg(\inn_{\mathbf{w}+\epsilon \mathbf{e}_i}(I) : m_{\sigma}^{\infty})$$ and
 $$\deg(\inn_{\mathbf{w}}(I):m_{\sigma}^{\infty}) =
 \deg(\inn_{\mathbf{w}-\epsilon \mathbf{e}_i}(I):m_{\sigma}^{\infty}).$$
\end{lemma}

\begin{proof}
Let $\prec$ be the reverse lexicographic term order on the monomials
 in $\Rbar[x_0,\dots,x_n]$ with $x_i$ largest (so initial terms of
 homogeneous polynomials are divisible by the lowest power of $x_i$)
 and $x_0$ second-largest, and let $\prec'$ be the lexicographic term
 order with $x_i$ smallest (so initial terms of all polynomials are
 divisible by the largest power of $x_i$) and $x_0$ largest.  For  $0 <
 \epsilon \ll 1$ we define the monomial ideals
\begin{enumerate}
\item $J^+ = \inn_{\prec}(\inn_{\mathbf{w}}(I) : m_{\sigma}^{\infty})$,
\item $J^+_{\mathbf{\epsilon}} = \inn_{\prec}(\inn_{\mathbf{w+\epsilon \mathbf{e}_i}}(I)
  :m_{\sigma}^{\infty})$,
\item $J_{\mathbf{w}} = \inn_{\prec}(\inn_{\mathbf{w}}(I) :
  (m_{\sigma} x_i)^{\infty})$, 
\item $J^- = \inn_{\prec'}(\inn_{\mathbf{w}}(I) : m_{\sigma}^{\infty})$,
\item $J^-_{\mathbf{\epsilon}} = \inn_{\prec'}(\inn_{\mathbf{w-\epsilon \mathbf{e}_i}}(I)
  :m_{\sigma}^{\infty})$.
\end{enumerate}
Since the Hilbert function, and thus the degree, is preserved on
passing to a monomial initial ideal by Lemma~\ref{l:stdmonosbasis},
it suffices to prove that 
for $\epsilon$ sufficiently small we have 
$$\deg(J^+) = \deg(J_{\mathbf{w}})+\deg(J^+_{\epsilon}) 
\quad \text{and} \quad
\deg(J^-) = \deg(J^-_{\epsilon}).$$

We assume that $\epsilon$ has been chosen sufficiently small so that
$\inn_{\mathbf{w}+\epsilon
  \mathbf{e}_i}(I) = \inn_{\mathbf{e}_i}(\inn_{\mathbf{w}}(I))$ and 
  $\inn_{\mathbf{w}-\epsilon \mathbf{e}_i}(I) = \inn_{-\mathbf{e}_i}(\inn_{\mathbf{w}}(I))$, and furthermore
$\inn_{\mathbf{w}\pm \epsilon \mathbf{e}_i}(f) = \inn_{\pm
  \mathbf{e}_i}(\inn_{\mathbf{w}}(f))$ for all homogeneous $f \in I$ of
degree at most $2\deg(I)$.  This is possible by
Proposition~\ref{p:star}.

{\bf Step 1: }{\em $J^+ \subseteq J_{\mathbf{w}} \cap J^+_{\epsilon}$, and
$J^- \subseteq J^-_{\epsilon}$.}
Since $(\inn_{\mathbf{w}}(I):m_{\sigma}^{\infty})$ is contained in 
$(\inn_{\mathbf{w}}(I): (m_{\sigma}x_i)^{\infty}) = (
(\inn_{\mathbf{w}}(I): m_{\sigma}^{\infty}): x_i^{\infty})$, it follows that $J^+ \subseteq
J_{\mathbf{w}}$.  
Now, we have $\inn_{\bfe_i}(\inn_{\mathbf{w}}(I) : m_{\sigma}^{\infty}) \subseteq
(\inn_{\bfe_i}(\inn_{\mathbf{w}}(I)) : m_{\sigma}^{\infty})$.
Since $\inn_{\mathbf{e}_i}(\inn_{\mathbf{w}}(I)) = \inn_{\mathbf{w}+\epsilon \mathbf{e}_i}(I)$ 
and the term order $\prec$
is a refinement of the partial order given by $\mathbf{e}_i$,  we have
$J^+ \subseteq J^+_{\epsilon}$.  The same is true with $\epsilon$
replaced by $-\epsilon$, and $\prec$ replaced by $\prec'$, so
$J^- \subseteq J^-_{\epsilon}$.

{\bf Step 2: }{\em $(J^+:x_i^{\infty}) = J_{\mathbf{w}}$. } 
The containment $\inn_{\prec}(J : x_i^{\infty}) \subseteq
(\inn_{\prec}(J) : x_i^{\infty})$ holds for any ideal $J$.
Taking $J = (\inn_{\mathbf{w}}(I): m_{\sigma}^\infty)$, we obtain 
$J_{\mathbf{w}} \subseteq (J^+:x_i^{\infty})$.

Since $\prec$ is the reverse-lexicographic order and
$(\inn_{\mathbf{w}}(I):(m_{\sigma}x_i)^{\infty})$ is homogeneous and
saturated with respect to $x_i$, we claim that $J_{\mathbf{w}}$ is
also saturated with respect to $x_i$.  To see this, suppose
$\mathbf{x}^{\mathbf{v}} = \inn_{\prec}(f)$ is a minimal generator of
$J_{\mathbf{w}}$ with $f \in (\inn_{\mathbf{w}}(I) :
(m_{\sigma}x_i)^{\infty})$.  If $v_i>0$ then $u_i>0$ for all other
monomials $\mathbf{x}^{\mathbf{u}}$ in $f$, by the definition of the
reverse-lexicographic order, so $f/x_i \in (\inn_{\mathbf{w}}(I) :
(m_{\sigma}x_i)^{\infty})$, and thus $\mathbf{x}^{\mathbf{v}}/x_n \in
J_{\mathbf{w}}$, contradicting that $\mathbf{x}^{\mathbf{v}}$ was a
minimal generator.  We thus conclude that no generator of
$J_{\mathbf{w}}$ is divisible by $x_i$, so the claim follows.  This
means that $J_{\mathbf{w}} \subseteq (J^+:x_i^{\infty}) \subseteq
(J_{\mathbf{w}}:x_i^{\infty}) = J_{\mathbf{w}}$, so
$(J^+:x_i^{\infty}) = J_{\mathbf{w}}$.

{\bf Step 3: }{\em $J^+_{\epsilon} \subseteq (J^+ : x_0^{\infty})$, and
$J^-_{\epsilon} \subseteq (J^- : x_0^{\infty})$.}  The proof is the
same for both cases, so we give it for $J^+_{\epsilon} \subseteq (J^+
: x_0^{\infty})$.  The only changes needed are to replace each $\epsilon$ by $-\epsilon$, and $\prec$ by $\prec'$.

 If $\mathbf{x}^{\mathbf{u}} \in J^+_{\epsilon}$, then there is $f \in
 I$ of the form $f=m_{\sigma}^k(\mathbf{x}^{\mathbf{u}} \tplus f')\tplus f''$,
 where all terms $c_{\mathbf{v}} \mathbf{x}^{\mathbf{v}}$ in $f'$ have
 $c_{\mathbf{v}}+ \mathbf{w}\cdot \mathbf{v}+ \epsilon v_i =
 \mathbf{w} \cdot \mathbf{u}+ \epsilon u_i$ but
 $\mathbf{x}^{\mathbf{u}} \prec \mathbf{x}^{\mathbf{v}}$, and all
 terms in $f''$ have greater weight with respect to
 $\mathbf{w}+\epsilon \mathbf{e}_i$ than $m_{\sigma}^k\mathbf{x}^{\mathbf{u}}$.
  This means that $\inn_{\mathbf{w}+\epsilon \mathbf{e}_i}(f) =
  m_{\sigma}^k(\mathbf{x}^{\mathbf{u}} \tplus
  \bigoplus_{c_{\mathbf{v}}\mathbf{x}^{\mathbf{v}} \text{ a term of } f'}
  \mathbf{x}^{\mathbf{v}})$. We can write $f''=f''_1\tplus f''_2$, where the
  terms $\mathbf{x}^{\mathbf{v}'}$ in $f''_1$ have $c_{\mathbf{v}'}+\mathbf{w} \cdot \mathbf{v}'=
  \mathbf{w} \cdot \mathbf{u} + k(\sum_{i \not \in \sigma} w_i)$, and
  the equality is an inequality $>$ for the terms in $f''_2$.  Then
  $\inn_{\mathbf{w}}(f) = \inn_{\mathbf{w}+\epsilon \mathbf{e}_i}(f) \tplus
  \bigoplus_{c_{\mathbf{v}} \mathbf{x}^{\mathbf{v}} \text{ a term of }
    f''_1} \mathbf{x}^{\mathbf{v}}$.  We thus have $v_i>u_i$ for terms
  $c_{\mathbf{v}}\mathbf{x}^{\mathbf{v}}$ of $f''_1$.

We now show that, after multiplying $f$ by a sufficiently high power
of $x_0$ and doing some vector eliminations, we may assume that each
term of $f''_1$ is divisible by $m_{\sigma}^k$.  By
Lemma~\ref{l:twovariable}, since $I$ is saturated, for all $j>0$
there is a polynomial $f_{j} \in I \cap \Rbar[x_0,x_j]$ that factors
as a product of linear terms.  
Thus $\inn_{\mathbf{w}}(f_j) =x_0^{a_j}(x_0 \tplus x_j)^{b_j}
x_j^{c_j} \in \inn_{\mathbf{w}}(I)$ for some $a_j,b_j,c_j \geq 0$.  Fix $j>0$ with $j \not \in \sigma$.
We may assume that $b_j \neq 0$, as otherwise
$x_0^{a_j}x_j^{c_j} \in \inn_{\mathbf{w}}(I)$, so $0 \in
(\inn_{\mathbf{w}}(I):m_{\sigma}^{\infty})$, and so $J^+ = \langle 0
\rangle$, from which the lemma is immediate.  Since
$\inn_{\mathbf{w}}(I)$ is a tropical ideal, we may use the term
$x^{a_j+b_j}_0$ of $\inn_{\mathbf{w}}(f_j)$ to eliminate a term
$x_0^l\mathbf{x}^{\mathbf{v}}$ of $x_0^l f_1''$ from
$\inn_{\mathbf{w}}(x_0^lf)$, where $l \gg 0$.
This replaces the term
$x_0^l\mathbf{x}^{\mathbf{v}}$ with terms with the same exponent on
$x_i$, but higher exponents on $x_j$.  After iterating with all $f_j$,
as $j$ ranges over $\{1,\dots, n\} \setminus \sigma$, we may assume
that all terms of $\inn_{\mathbf{w}}(x_0^lf)$ are divisible by
$m_{\sigma}^k$.  Note that these elimination steps did not change
$\inn_{\mathbf{e}_i}(\inn_{\mathbf{w}}(x_0^lf))$, as they did not
change the exponent of $x_i$ in the affected terms, so did not change
$\inn_{\prec}(\inn_{\mathbf{w}}(x_0^lf)) = x_0^l m_{\sigma}^k
\mathbf{x}^{\mathbf{u}}$.  We thus conclude that $x_0^l
\mathbf{x}^{\mathbf{u}} \in J^+$, so $\mathbf{x}^{\mathbf{u}} \in
(J^+:x_0^{\infty})$.  This shows that $J_{\epsilon}^+ \subseteq
(J^+:x_0^{\infty})$ as required.

{\bf Step 4:} {\em $J^+$ contains a power of every variable except $x_0,
x_i$, and $J^-$ contains a power of every variable except $x_0$.}  From
the previous paragraph we have $x_0^{a_j}(x_0 \tplus x_j)^{b_j}
x_j^{c_j} \in \inn_{\mathbf{w}}(I)$ for all $j>0$, so $x_j^{b_j+c_j}
\in J^+$ for $j \neq i$, and $x_j^{b_j+c_j} \in J^-$ for all $j>0$.

{\bf Step 5:} {\em $J^+_d = (J_{\mathbf{w}})_d \cap
  (J^+_{\epsilon})_d$ and $J^-_d = (J^-_{\epsilon})_d$ for $d \gg 0$.}
Since all of $J^+$, $J^-$, $J_{\mathbf{w}}$, $J^+_{\epsilon}$, and
$J^-_{\epsilon}$ are monomial ideals, they have monomial primary
decompositions, viewed as monomial ideals living in a
polynomial ring over a field.  Since $J^+$ contains a power of every
variable except $x_0, x_i$ and $J^-$ contains a power of every
variable except $x_0$, and both have constant Hilbert polynomials,
the only possible associated primes of $J^+$
are $P_i = \langle x_j : j \neq i \rangle$, $P_0 = \langle x_j : j
\neq 0 \rangle$, and $\mathfrak{m} = \langle x_j : 0 \leq j \leq n
\rangle$, and the only possible associated primes of $J^-$ are $P_0$
and $\mathfrak{m}$.  This means that for $d \gg 0$,
\begin{equation} 
\label{eqtn:primarydecomp} J^+_d =
(J^+:x_0^{\infty})_d \cap (J^+:x_i^{\infty})_d,
\end{equation}
and 
\begin{equation} 
\label{eqtn:primarydecomp2} J_d^- =
(J^-:x_0^{\infty})_d. 
\end{equation}

Since $J^- \subseteq J^-_{\epsilon} \subseteq (J^- :x_0^{\infty})$, we
thus have $J^-_d = (J^-_{\epsilon})_d$ for $d \gg 0$, and so
$\deg(J^-)=\deg(J^-_{\epsilon})$.

Since $J^+ \subseteq J_{\mathbf{w}} \cap J^+_{\epsilon}$, we have
$J^+_d \subseteq (J_{\mathbf{w}})_d \cap (J^+_{\epsilon})_d \subseteq
(J^+ :x_i^{\infty})_d \cap (J^+:x_0^{\infty})_d = J^+_d$ for $d \gg
0$.  This means that $\deg(J^+)= \deg(J_{\mathbf{w}} \cap
J^+_{\epsilon})$.

{\bf Step 6: } {\em Equality of degrees.}
Since $\inn_{\mathbf{w}}(f_i)= x_0^{a_i}(x_0 \tplus
x_i)^{b_i}x_i^{c_i} \in \inn_{\mathbf{w}}(I)$, we have $(x_0 \tplus x_i)^{b_i}
\in (\inn_{\mathbf{w}}(I):(m_{\sigma}x_i)^{\infty})$, and so 
  $x_0^{b_i} \in J_{\mathbf{w}}$.  On the other hand, since $\deg(f_i) \leq  2 \deg(I)$ by Lemma~\ref{l:twovariable}, 
  $\inn_{\mathbf{w}+\epsilon \mathbf{e}_i}(f_i) =
  \inn_{\mathbf{e}_i}(\inn_{\mathbf{w}}(f_i)) = x_0^{a_i+b_i}x_i^{c_i}$,
  so $x_i^{c_i} \in J^+_{\epsilon}$.  This means that the monomials
  not in $J^+_{\epsilon}$ and not in $J_{\mathbf{w}}$ are disjoint in high
  degree, so $\deg(J^+) = \deg(J_{\mathbf{w}} \cap J^+_{\epsilon}) =
  \deg(J_{\mathbf{w}}) + \deg(J^+_{\epsilon})$.
\end{proof}

We are now ready to prove Theorem~\ref{t:degissummult}.

\begin{proof}[Proof of Theorem~\ref{t:degissummult}]
We first consider the related situation that $I$ is a zero-dimensional
homogeneous tropical ideal that is saturated with respect to the
product of the variables, and prove the following stronger result:
For any $[\mathbf{w}]$ in the tropical torus $\mathbb R^{n+1}/\mathbb R \mathbf{1}$ of $ \trop(\mathbb P^n)$,
\begin{equation}
\label{eqtn:degclaim}
\deg(\inn_{\mathbf{w}}(I): x_0^{\infty}) = \sum_{p \in
  (\mathbf{w}+\cone(\mathbf{e}_l : l >0)) \cap V(I)} \mult_{V(I)}(p).
\end{equation}

To see that this implies the theorem, we claim that when $\mathbf{w}$
is in the interior of an unbounded cell $C_{\mathbf{w}}$ of the
Gr\"obner complex of $I$ that contains points $\mathbf{w}'$ with $w'_l
\ll w'_0$ for all $l>0$, then $(\inn_{\mathbf{w}}(I):x_0^{\infty}) =
\inn_{\mathbf{w}}(I)$, so the result follows from the facts that
$\deg(\inn_{\mathbf{w}}(I)) = \deg(I)$ by \cite{TropicalIdeals}*{Corollary 3.6}, and that for such
$\mathbf{w}'$ all of $V(I)$ is contained in
$\mathbf{w}'+\cone(\mathbf{e}_l : l >0)$.  To see the claim, note that
$\inn_{\mathbf{w}}(I)$ is a monomial ideal, and
suppose that $\mathbf{x}^{\mathbf{u}}x_0^l \in \inn_{\mathbf{w}}(I)$.
There is $f = \bigoplus c_{\mathbf{v}} \mathbf{x}^{\mathbf{v}} \in
I$ with $\inn_{\mathbf{w}}(f) = \mathbf{x}^{\mathbf{u}}x_0^l$ and all
other monomials $\mathbf{x}^{\mathbf{v}}$ occurring in $f$ not in
$\inn_{\mathbf{w}}(I)$; such $f$ corresponds to the fundamental circuit
of $\mathbf{x}^{\mathbf{u}}x_0^l$ over the basis 
$B := \{ \bfx^\bfv : \bfx^\bfv \notin \inn_{\bfw}(I)_{|\bfu|+l} \}$ of the matroid
$\uMat(I_{|\bfu|+l})$.
Since $\inn_{\mathbf{w'}}(I) =
\inn_{\mathbf{w}}(I)$ for all $\mathbf{w}' \in C_{\mathbf{w}}$, we
must have $\inn_{\mathbf{w}'}(f) = \mathbf{x}^{\mathbf{u}}x_0^l$ for
such $\mathbf{w}'$.  By choosing such a $\mathbf{w}'$ with $w'_j$
sufficiently less than $w_0$ for all $1 \leq j \leq n$, we see that
the other terms of $f$ must also be divisible by $x_0^l$, so since $I$
is $(\prod x_j)$-saturated, $f/x_0^l \in I$, and so
$\mathbf{x}^{\mathbf{u}} \in \inn_{\mathbf{w}}(I)$.  This shows that
$(\inn_{\mathbf{w}}(I):x_0^{\infty}) = \inn_{\mathbf{w}}(I)$ as
required.

To prove Equation~\eqref{eqtn:degclaim}, we prove the following
stronger formula.  For $\sigma \subseteq \{1,\dots,n\}$, set
$m_{\sigma} = \prod_{l \in \{0,\dots,n\} \setminus \sigma} x_l$.  Then
for any $\sigma \subseteq \{1,\dots,n \}$ and any $\mathbf{w} \in
\mathbb R^{n+1}$,
\begin{equation}
\label{eqtn:strongdegclaim}
\deg(\inn_{\mathbf{w}}(I): m_{\sigma}^{\infty}) = \sum_{p \in
  (\mathbf{w}+\cone(\mathbf{e}_i : i \in \sigma)) \cap V(I)} \mult_{V(I)}(p).
\end{equation}
The case that $\sigma = \{1,\dots,n\}$ is
Equation~\eqref{eqtn:degclaim}.  We prove
Equation~\eqref{eqtn:strongdegclaim} by induction on $|\sigma|$.  When
$\sigma = \emptyset$, the right-hand side is
$\mult_{V(I)}(\mathbf{w})$, and the left-hand side is the definition
of this.

Now suppose that $|\sigma|>0$ and Equation~\eqref{eqtn:strongdegclaim}
holds for all $\sigma'$ with $|\sigma'|<|\sigma|$.  Write $\mathcal
C_{\sigma}$ for the locus of $\mathbf{w} \in \mathbb R^{n+1}$ for
which \eqref{eqtn:strongdegclaim} holds.  To prove the claim we need
to show that $\mathcal C_{\sigma} = \mathbb R^{n+1}$.  
Suppose $\mathcal C_{\sigma} \neq \mathbb R^{n+1}$, and fix
$\mathbf{w}_0 \in \mathbb R^{n+1} \setminus \mathcal C_{\sigma}$. Fix $i \in \sigma$, and consider the ray $\{ \mathbf{w} : \mathbf{w} = \bfw_0 + \lambda \bfe_i \text{ with } \lambda \geq 0\}$.  We observe that for
$\lambda$ large enough, $\bfw_0 + \lambda \bfe_i \in \C_{\sigma}$.
Indeed, by Lemma~\ref{l:twovariable} there is a polynomial $f \in
I \cap \Rbar[x_0,x_i]$ that factors as a product of linear terms, and
is saturated with respect to $x_0$ and $x_i$.  Thus for $\lambda \gg 0$,
we have $\inn_{\mathbf{w}_0+\lambda \bfe_i}(f) = x_0^m$ for some
$m>0$.  This means that $(\inn_{\mathbf{w}_0+\lambda
  \bfe_i}(I):m_{\sigma}^{\infty} )= \langle 0 \rangle$, so since $\mathbf{w}_0+\lambda \bfe_i+\cone({\mathbf e}_i : i \in \sigma) \cap V(I) = \emptyset$ for sufficiently large $\lambda$, 
\eqref{eqtn:strongdegclaim} holds for such  a point $\mathbf{w}_0+\lambda
\bfe_i$. 

We thus have that the set of $\lambda$ such that $\mathbf{w}_0+\lambda
\bfe_i \notin \mathcal C_{\sigma}$ is bounded above; let
$\tilde{\lambda}$ be its supremum, and set $\bfw =
\mathbf{w}_0+\tilde{\lambda} \bfe_i$.  By
Lemma~\ref{l:lemmafordegissummult} 
the left-hand side of \eqref{eqtn:strongdegclaim} does not change
when subtracting $\epsilon \mathbf{e}_i$ from $\mathbf{w}$ for
sufficiently small $\epsilon>0$, which implies that $\bfw \notin
\mathcal C_{\sigma}$.
Now, since $V(I)$ is finite and $\inn_\bfw(I)_{\leq \deg(I)}$ is generated by 
finitely many polynomials, we can fix $\epsilon > 0$ small enough so that: 
\begin{enumerate}
\item $\inn_{\mathbf{w}+\epsilon \mathbf{e}_i}(I) = 
\inn_{ \mathbf{e}_i}(\inn_{\mathbf{w}}(I))$, and
  $\inn_{\mathbf{w}+\epsilon \mathbf{e}_i}(f) =
  \inn_{\mathbf{e}_i}(\inn_{\mathbf{w}}(f))$ for all $f \in
  I$ of degree at most $2\deg(I)$, and
\item there are no points $p \in V(I)$ with $p \in
  \mathbf{w}+\cone(\mathbf{e}_j : j \in \sigma)$ and $w_i <
  p_i<w_i+\epsilon$.
\end{enumerate}

We now consider $(\inn_{\mathbf{w}}(I):(m_{\sigma}x_i)^{\infty})$ and $(\inn_{\mathbf{w}+\epsilon \mathbf{e}_i}(I):m_{\sigma}^{\infty})$.
The degree of $(\inn_{\mathbf{w}}(I):(m_{\sigma}x_i)^{\infty})$ is the 
left-hand side of \eqref{eqtn:strongdegclaim} for the set $\sigma'
    = \sigma \setminus \{i\}$.  By the induction hypothesis we know
    that this equals $\sum_{p \in (\mathbf{w} + \cone(\mathbf{e}_j : j
      \in \sigma')) \cap V(I)} \mult_{V(I)}(p)$.
Since $\mathbf{w}+\epsilon \mathbf{e}_i \in
    \mathcal C_{\sigma}$, 
    the degree
    of $(\inn_{\mathbf{w}+\epsilon \mathbf{e}_i}(I):m_{\sigma}^{\infty})$
    equals  $\sum_{p \in
      (\mathbf{w}+\epsilon \mathbf{e}_i +\cone(\mathbf{e}_j : j \in \sigma)) \cap V(I)}
    \mult_{V(I)}(p)$.  By our assumptions on $\epsilon$, we thus have
     that 
$$\deg((\inn_{\mathbf{w}}(I):(m_{\sigma}x_i)^{\infty})) +
     \deg((\inn_{\mathbf{w}+\epsilon \mathbf{e}_i}(I):
     m_{\sigma}^{\infty})) = \sum_{p \in (\mathbf{w}+\cone(\mathbf{e}_i
       : i \in \sigma)) \cap V(I)} \mult_{V(I)}(p),$$ which is the
     right-hand side of \eqref{eqtn:strongdegclaim}.
     Equation~\eqref{eqtn:strongdegclaim} then follows from
     Lemma~\ref{l:lemmafordegissummult}.

We now consider the case that $I$ is a zero-dimensional ideal in
$\Rbar[x_1^{\pm 1},\dots,x_n^{\pm 1}]$.  Let $I^h \subseteq
\Rbar[x_0,\dots,x_n]$ be the homogenization of $I \cap
\Rbar[x_1,\dots,x_n]$.  By definition we have $\deg(I^h) = \deg(I)$, and by Remark~\ref{r:varietyintorus} $V(I)$
equals $V(I^h)$ after the identification of $\mathbb R^{n+1}/\mathbb R
\mathbf{1}$ with $\mathbb R^n$ by choosing the representative with
first coordinate $0$.  In addition, by Part~\ref{enum:homogenizationmultiplicity} of
Proposition~\ref{p:basicdeg} we have $\mult_{V(I)}(\mathbf{w}) =
\mult_{V(I^h)}([0:\mathbf{w}])$ for all $\mathbf{w} \in V(I)$.  The
theorem for $I$ then follows from the result for homogeneous ideals
proved above.
\end{proof}

\begin{remark}
The proof of Theorem~\ref{t:degissummult} is simpler in the realizable
case, as we can use primary decomposition.  If $I \subseteq
K[x_0,\dots,x_n]$ is zero-dimensional, then $I = J \cap \bigcap_{p \in
  V(I)} Q_p$, where $Q_p$ is $I(p)$-primary, and $J$ is $\langle
x_0,\dots,x_n \rangle$-primary.  We then have $\deg(I) = \sum_{p \in
  V(I)} \deg(Q_p)$. The degree of an ideal $Q_p$ primary to the ideal of a
point $p$ equals its multiplicity, so the result follows.
\end{remark}


\section{Balancing} \label{s:balancing}

In this section we prove 
that the top-dimensional part of the variety of
a tropical ideal is {\defbold balanced} with respect to intrinsically
defined multiplicities.  

We first recall the definition of the balancing condition. 
A {\defbold weighted} polyhedral complex is a polyhedral
complex $\Sigma$ with a weight function that assigns a positive
integer to each maximal cell of the complex.

\begin{definition}
\label{d:balancing}
Let $\Sigma$ be a pure one-dimensional weighted rational polyhedral fan
with $s$ rays.  
 Let $\mathbf{u}_i$ be the first lattice point on
the $i$th ray of $\Sigma$, and let $m_i$ be the weight on that ray.
We say that $\Sigma$ is {\defbold balanced} if
$\sum_{i=1}^s m_i \mathbf{u}_i = \mathbf{0}$.

Let $\Sigma$ be a pure $d$-dimensional $\mathbb R$-rational weighted
polyhedral complex, and let $\sigma$ be a $(d-1)$-dimensional cell of
$\Sigma$.  The quotient of the star $\starr_{\Sigma}(\sigma)$ by the
subspace $\spann(\sigma)$ is a pure one-dimensional rational polyhedral fan
that inherits a weighting from $\Sigma$.  We say that $\Sigma$ is
{\defbold balanced} if this fan is balanced for {\em all}
$(d-1)$-dimensional cells of $\Sigma$.
\end{definition}

We now  define the multiplicities on the variety of a tropical ideal
that give $V(I)$ the structure of a weighted polyhedral complex.  This needs the following lemma.

Given a $\mathbb Z^d$-grading on the tropical Laurent polynomial semiring
$S = \Rbar[x_1^{\pm 1},\dots,x_n^{\pm 1}]$ or 
$S = \BB[x_1^{\pm 1},\dots,x_n^{\pm 1}]$, we denote by
$S_{\mathbf{0}}$ the subsemiring consisting of degree zero elements.
An ideal $I \subseteq S_{\mathbf{0}}$ is a tropical ideal if for any
finite collection of monomials $E \subseteq S_{\mathbf{0}}$ the
restriction $I|_{E}$ is the collection of vectors of a valuated
matroid on the set $E$; see \cite{TropicalIdeals}*{Definition 4.3}.
Note that $S_{\mathbf{0}}$ is isomorphic to a Laurent polynomial
semiring in fewer variables.  We can use this fact to define the
dimension and degree of $I$.  Explicitly, we say that $I$ is
zero-dimensional if there is an upper bound on the rank of
$\underline{\Mat}(I|_E)$ for $E \subseteq S_{\mathbf{0}}$ 
a finite collection of monomials, and in that case we set the degree of $I$ to
be the maximum possible such rank.

\begin{lemma} \label{l:maxinitialideal}
Let $I \subseteq \Rbar[x_1^{\pm 1},\dots,x_n^{\pm 1}]$ be a tropical
ideal, and fix $\mathbf{w} \in \mathbb R^n$ in the relative
interior of a maximal cell $\sigma$ of $V(I)$, where $V(I)$ is given the
Gr\"obner polyhedral complex structure.  By
Corollary~\ref{c:homogeneous}, $\inn_{\mathbf{w}}(I)$ is homogeneous with respect to a
$\mathbb Z^{\dim(\sigma)}$-grading of $S:= \mathbb B[x_1^{\pm 1},\dots,x_n^{\pm 1}]$. 
Let $S_{\mathbf{0}}$ be the degree-zero part of $S$ with respect to this grading. 
Then $\inn_{\mathbf{w}}(I) \cap S_{\mathbf{0}}$ 
is a zero-dimensional tropical ideal.
\end{lemma}

\begin{definition} \label{d:multiplicitydefn}
Under the same setup as in Lemma \ref{l:maxinitialideal}, 
we define the {\defbold multiplicity} of
$V(I)$ at $\mathbf{w}$ to be the degree of the 
zero-dimensional tropical ideal $\inn_{\mathbf{w}}(I) \cap S_{\mathbf{0}}$.  
\end{definition}

Note that this definition agrees with the one given for
zero-dimensional ideals in Definition~\ref{d:zerodimmult}, as the
grading is trivial when $\dim(I)= 0$.

\begin{proof}[Proof of Lemma \ref{l:maxinitialideal}]
 By Proposition~\ref{p:star}, we have $V(\inn_{\mathbf{w}}(I)) =
 \spann(\sigma)$, which is a subspace of $\mathbb R^n$ of dimension $d := \dim(\sigma)$.
 Using Lemma~\ref{l:changeofcoords} we may change coordinates to
 assume that $\spann(\sigma) =
 \spann(\mathbf{e}_1,\dots,\mathbf{e}_d)$.
 Explicitly, choose a basis for the lattice
 $\spann(\sigma) \cap \mathbb Z^n$, and let $A \in \GL(n,\mathbb Z)$
 be the inverse of an invertible $n \times n$ matrix with first $d$ rows equal
 to this basis.  Set $\phi^* \colon \Rbar[x_1^{\pm 1},\dots,x_n^{\pm
     1}] \rightarrow \Rbar[x_1^{\pm 1},\dots,x_n^{\pm 1}]$ to be
 $\phi^*(x_i) = \mathbf{x}^{\mathbf{a}_i}$, where $\mathbf{a}_i$ is
 the $i$th column of $A$.  By construction we then have
 $\trop(\phi)(\spann(\sigma)) = \spann(\mathbf{e}_1,\dots,\mathbf{e}_d)$.
 The ideal $I'={\phi^*}^{-1}(I)$ satisfies
 $ V(\inn_{\trop(\phi)(\mathbf{w})}(I')) =
 V({\phi^*}^{-1}(\inn_{\mathbf{w}}(I))) =
 \trop(\phi)(V(\inn_{\mathbf{w}}(I))) = \trop(\phi)(\spann(\sigma)) =
\spann(\mathbf{e}_1,\dots,\mathbf{e}_d)$.
 The $\mathbb Z^d$-grading induced by $\trop(\phi)(\sigma)$ on $S$
is given by $\deg(x_i) = \mathbf{e}_i$ for $ 1\leq i \leq d$ and
$\deg(x_i)=\mathbf{0}$ for $i > d$, so the degree $\bfzero$ part of $S$ is 
$\mathbb B[x_{d+1}^{\pm 1},\dots,x_n^{\pm 1}]$.
By Part \ref{enum:automorphismdeg} of Proposition \ref{p:basicdeg},
the tropical ideal $\inn_{\mathbf{w}}(I) \cap S_{\mathbf{0}}$ is zero-dimensional
if and only if the tropical ideal 
${\phi^*}^{-1}(\inn_{\mathbf{w}}(I)) \cap \BB[x_{d+1}^{\pm 1},\dots,x_n^{\pm 1}]$ is zero-dimensional.
We can thus assume that $\spann(\sigma) =
\spann(\mathbf{e}_1,\dots,\mathbf{e}_d)$, and $S_{\mathbf{0}}= \BB[x_{d+1}^{\pm 1},\dots,x_n^{\pm 1}]$. 

To prove that $\inn_{\mathbf{w}}(I) \cap \BB[x_{d+1}^{\pm 1},\dots,x_n^{\pm 1}]$ 
is zero-dimensional, let $J = \inn_{\mathbf{w}}(I) \Rbar[x_1^{\pm 1},\dots,x_n^{\pm 1}]$. 
Denote by $\pi \colon \mathbb R^n \rightarrow
\RR^{n-d}$ the projection onto the last $n-d$ coordinates. By Theorem \ref{t:projection} we have
$$V(\inn_{\mathbf{w}}(I) \cap \BB[x_{d+1}^{\pm 1},\dots,x_n^{\pm 1}]) = 
V(J \cap \Rbar[x_{d+1}^{\pm 1},\dots,x_n^{\pm 1}]) = \pi(V(J)) = 
\pi(V(\inn_{\mathbf{w}}(I))).$$
As $V(\inn_{\mathbf{w}}(I)) = \spann(\mathbf{e}_1,\dots,\mathbf{e}_d)$, the last
term in the equalities above is equal to $\{\bfzero\}$.
Theorem~\ref{t:dimiscorrect} then implies that 
$\inn_{\mathbf{w}}(I) \cap \BB[x_{d+1}^{\pm 1},\dots,x_n^{\pm 1}]$ is zero-dimensional, as claimed.
\end{proof}

Recall from \S\ref{s:Groebner} that if $\Sigma$ is a polyhedral
complex and $\bfw$ is in the interior of a cell $\sigma \in \Sigma$,
the star $\starr_\Sigma(\bfw)$ is a polyhedral fan with cones
$\overline{\tau}$ for any $\tau \in \Sigma$ containing $\sigma$.  If
$\Sigma$ is a weighted polyhedral complex, the star
$\starr_\Sigma(\bfw)$ inherits weights on its maximal cones, and thus
it is a weighted polyhedral fan.
 
\begin{proposition}\label{p:starweighted}
Suppose $I$ is a tropical ideal in $\Rbar[x_1^{\pm 1}, \dots, x_n^{\pm 1}]$ or $\BB[x_1^{\pm 1}, \dots, x_n^{\pm 1}]$, and $\bfw \in \RR^n$. Then
$$V(\inn_\bfw(I)) = \starr_{V(I)}(\bfw)$$
as \emph{weighted} polyhedral complexes, where $V(I)$ and $V(\inn_\bfw(I))$ are given their Gr\"obner complex structures.
\end{proposition}
\begin{proof}
The equality of these fans as polyhedral complexes was proved in Proposition \ref{p:star}.
To show that they have the same weighting, suppose $\bfv$ is in the 
interior of a maximal cone $\overline{\tau}$ of $V(\inn_\bfw(I)) = \starr_{V(I)}(\bfw)$. 
The linear subspace $\spann(\overline{\tau}) = \spann(\tau)$ 
induces a $\ZZ^{\dim(\tau)}$-grading on 
$S := \mathbb B[x_1^{\pm 1},\dots,x_n^{\pm 1}]$.
The multiplicity of the cone
$\overline \tau \in V(\inn_\bfw(I))$ is equal to the degree of the 
zero-dimensional tropical ideal 
$\inn_{\mathbf{v}}(\inn_\bfw(I)) \cap S_{\mathbf{0}} = 
\inn_{\bfw+\epsilon \bfv}(I) \cap S_{\mathbf{0}}$,
which is equal to the multiplicity of the cell $\tau$ in $V(I)$,
as desired.
\end{proof}

We now show the key special case of Theorem~\ref{t:balanced} when the
ideal $I$ is one-dimensional with coefficients in $\mathbb B$.

\begin{lemma} \label{l:curvesbalancing}
Let $I \subseteq \mathbb B[x_1^{\pm 1},\dots,x_n^{\pm 1}]$ be a one-dimensional tropical
ideal.  Then $V(I)$ is the support of a one-dimensional 
rational polyhedral fan that is balanced, with weights on the rays given 
by the multiplicities of Definition~\ref{d:multiplicitydefn}.
\end{lemma}
\begin{proof}
Since $I \subseteq \mathbb B[x_1^{\pm 1},\dots,x_n^{\pm 1}]$, we have
$I = \inn_{\mathbf{0}}(I)$ and thus $V(I) = \starr_{V(I)}(\mathbf{0})$
by Proposition~\ref{p:star}.  This implies that the
variety $V(I)$ is a fan, which is one-dimensional 
by Theorem~\ref{t:dimiscorrect}.   The same is true for the ideal $I' = I \Rbar[x_1^{\pm
    1},\dots,x_n^{\pm 1}]$, as the circuits are the same for both ideals.  For the
rest of the proof we replace $I$ by $I'$, so assume that $I \subseteq
\Rbar[x_1^{\pm 1},\dots,x_n^{\pm 1}]$.

Let $\mathbf{u}_1,\dots,\mathbf{u}_s$ be the first lattice points on
the rays of $V(I)$.  Set $m_i = \mult_{V(I)}(\mathbf{u}_i)$.  Let
$\mathbf{u} = \sum m_i \mathbf{u}_i$.  We will show that $\mathbf{v}
\cdot \mathbf{u} = 0$ for all $\mathbf{v} \in \mathbb Z^n$, which
implies that $\mathbf{u} = \mathbf{0}$, so $V(I)$ is balanced at the origin.
It
suffices to show that $\mathbf{v} \cdot \mathbf{u} =0$ for a basis of 
$\mathbb Z^n$, so up to reindexing we may assume that $\mathbf{v} = \mathbf{e}_1$.

 Let $J_1 = I|_{x_1=1} \subseteq \Rbar[x_2^{\pm 1},\dots,x_n^{\pm
     1}]$ and $J_2 = I|_{x_1=-1} \subseteq \Rbar[x_2^{\pm
     1},\dots,x_n^{\pm 1}]$.  By
 Proposition~\ref{p:transverseintersection}, since the intersections
 of $V(I)$ with with the two hyperplanes $\{x_1 =1 \}$ and $\{ x_1 = -1 \}$
 are transverse, these intersections equal $V(J_1)$ and $V(J_2)$
 respectively.
  By Proposition~\ref{p:dimgoesdown}, $J_1$ and $J_2$ are either the
  unit ideal $\langle 0 \rangle$ or are zero-dimensional.  By
  Part~\ref{item:trivialization} of Lemma~\ref{l:basicfacts}, $J_1=
  \langle 0 \rangle$ if and only if $J_2 = \langle 0 \rangle$.  In
  that case we have $V(J_1) = V(J_2) = \emptyset$, so $V(I) \cap
  \{ x_1 = \pm 1 \} = \emptyset$.  This
    means that $\mathbf{e}_1 \cdot \mathbf{u}_i = 0$ for all $1 \leq i
    \leq s$, and thus $\mathbf{e}_1 \cdot \mathbf{u} =0$.  We may thus
    restrict to the case that both $J_1$ and $J_2$ have dimension zero.

By Theorem~\ref{t:degissummult} we
know that
$$\deg(J_i) = \sum_{\mathbf{w} \in V(J_i)}
\mult_{V(J_i)}(\mathbf{w})$$ for $i=1,2$.  
 By Part~\ref{item:initial} of Lemma~\ref{l:basicfacts} we have $\inn_{\mathbf{w}}(J_1) =
 (\inn_{(1,\mathbf{w})}(I))|_{x_1=0}$ for all $\mathbf{w} \in \mathbb
 R^{n-1}$.  If $(1,\mathbf{w})$ lies on a ray $\RR_{\geq 0} \mathbf{u}'$, then $\inn_{(1,\mathbf{w})}(I) = \inn_{\mathbf{u}'}(I)$,
 as the initial ideals with respect to $\mathbf{u}'$ and $\lambda
 {\mathbf{u}'}$ are the same for any $\lambda >0$, since all circuits of $I$ have coefficients in $\mathbb B$.
 Thus
$$\deg(J_1)
  = \sum_{i: (\mathbf{u}_i)_1>0}
  \deg(\inn_{\mathbf{u}_i}(I)|_{x_1=0}) 
  \quad \text{ and } \quad
\deg(J_2)
  = \sum_{i: (\mathbf{u}_i)_1<0}
  \deg(\inn_{\mathbf{u}_i}(I)|_{x_1=0}).$$ 

Suppose now that $\mathbf{u}'$ is the first lattice point on a ray of
$V(I)$.  We next  show that 
\begin{equation} \label{eqtn:latticeindex}
\deg(\inn_{\mathbf{u}'}(I)|_{x_1=0}) =|u'_1| \mult_{V(I)}(\mathbf{u}').
\end{equation}
Grade $S=\mathbb B[x_1^{\pm 1},\dots,x_n^{\pm 1}]$  by $\deg(x_i)=u'_i$.  
The multiplicity $\mult_{V(I)}(\mathbf{u}')$ is equal to the degree of
the zero-dimensional ideal $\inn_{\mathbf{u}'}(I) \cap S_{\mathbf{0}} \subseteq S_{\mathbf{0}}$.
We now show that $\deg(\inn_{\mathbf{u}'}(I)|_{x_1=0}) = |u_1| \deg( \inn_{\mathbf{u}'}(I) \cap
S_{\mathbf{0}})$.
To see this, first note that $S_{\mathbf{0}} \cong \mathbb B[y_1^{\pm
    1},\dots, y_{n-1}^{\pm 1}]$.  This isomorphism is given by fixing
a basis $\mathbf{b}_1,\dots,\mathbf{b}_{n-1}$ for the kernel of the
map $\mathbb Z^n \rightarrow \mathbb Z$ given by sending $\mathbf{w}$
to $\mathbf{w} \cdot \mathbf{u}'$, and sending $y_i$ to
$\mathbf{x}^{\mathbf{b}_i}$.  While this isomorphism depends on the
choice of basis, different choices differ by an automorphism of
$\mathbb B[y_1^{\pm 1},\dots,y_{n-1}^{\pm 1}]$, so by
Part~\ref{enum:automorphismdeg} of Proposition~\ref{p:basicdeg},
$\inn_{\mathbf{u}'}(I) \cap S_{\mathbf{0}}$ can be considered as an
ideal in $\mathbb B[y_1^{\pm 1},\dots,y_{n-1}^{\pm 1}]$ for the
purpose of computing the degree.  The map $\phi^*$ from $\mathbb
B[y_1^{\pm 1},\dots,y_{n-1}^{\pm 1}]$ to $\mathbb B[x_2^{\pm
    1},\dots,x_n^{\pm 1}]$ then takes $y_i$ to
$\mathbf{x}^{\mathbf{b}'_i}$, where $\mathbf{b}'_i$ is the projection
of $\mathbf{b}_i$ onto the last $n-1$ coordinates.  Write $B$ for the
$(n-1) \times (n-1)$ matrix with columns $\mathbf{b}'_i$.  Note that
$|\det(B)| = |u'_1|$.  One way to see this is to note that since the
map $\mathbb Z^n/\mathbb Z(\mathbf{b}_1,\dots,\mathbf{b}_{n-1})
\rightarrow \mathbb Z$ given by $\mathbf{e}_i \mapsto u'_i$ is an
isomorphism, the induced map $\mathbb Z^{n-1}/\mathbb
Z(\mathbf{b}'_1,\dots,\mathbf{b}'_{n-1}) \rightarrow \mathbb Z/\mathbb
Zu'_1$ on the last $n-1$ coordinates given by $\mathbf{e}_i \mapsto
u'_i$ is an isomorphism as well, so $|\det(B)| = |u'_1|$.  Write
$\inn_{\mathbf{u}'}(I)_{\phi}$ for the ideal in $\mathbb B[x_2^{\pm
    1},\dots,x_n^{\pm}]$ generated by $\phi^*(\inn_{\mathbf{u}'}(I)
\cap S_{\mathbf{0}})$.  By Part~\ref{enum:degmap} of
Proposition~\ref{p:basicdeg}, the degree of
$\inn_{\mathbf{u}'}(I)_{\phi}$ equals $|\det(B)|
\deg(\inn_{\mathbf{u}'}(I) \cap S_{\mathbf{0}})$.  Finally, note that
$\inn_{\mathbf{u}'}(I)_{\phi}$ is generated by $\{ f|_{x_1=0} : f
\text{ is a generator of } \inn_{\mathbf{u}'}(I) \cap S_{\mathbf{0}}
\}$.  If $f$ is a homogeneous element of $\inn_{\mathbf{u}'}(I)$, then
$\mathbf{x}^{\mathbf{v}} f \in \inn_{\mathbf{u}'}(I) \cap
S_{\mathbf{0}}$, where $\mathbf{x}^{\mathbf{v}}$ is any monomial of
degree $-\deg(f)$, so $(\mathbf{x}^{\mathbf{v}} f)|_{x_1=0}$, and thus
also $f|_{x_1=0}$, are in $\inn_{\mathbf{u}'}(I)_{\phi}$.  Thus
$\inn_{\mathbf{u}'}(I)_{\phi} = \inn_{\mathbf{u}'}(I)|_{x_1=0}$, so
$\deg(\inn_{\mathbf{u}'}(I)|_{x_1=0}) = |u'_1|
\deg(\inn_{\mathbf{u}'}(I) \cap S_{\mathbf{0}})$ as required.  This
finishes the proof of \eqref{eqtn:latticeindex}.

We thus have 
\begin{equation} \label{eqtn:degJ1mult}
\deg(J_1) = \sum_{i : (\mathbf{u}_i)_1>0} |(\mathbf{u}_i)_1|
\mult_{V(I)}(\mathbf{u}_i),
\end{equation}
and analogously for $J_2$.  
We conclude that
\begin{equation*}
\textstyle
\mathbf{u} \cdot \mathbf{e}_1  = (\sum_{i=1}^s m_i \mathbf{u}_i)
\cdot \mathbf{e}_1 
  = \sum_{i=1}^s m_i (\mathbf{u}_i)_1 
 = \deg(J_1) - \deg(J_2) 
 = 0,
\end{equation*}
where the last equality follows from Lemma~\ref{l:basicfacts}, as 
$\deg(J_1) = \deg(\varphi(J_1)) = \deg(\varphi(J_2)) = \deg(J_2)$.
\end{proof}

\begin{theorem} \label{t:balanced}
Let $I \subseteq \Rbar[x_1^{ \pm 1},\dots,x_n^{\pm 1}]$ be a $d$-dimensional tropical
ideal, and let $\Sigma$ be a polyhedral complex with support equal to
$V(I)$.  
Then the collection of closed $d$-dimensional cells of $\Sigma$ 
forms an $\mathbb R$-rational
polyhedral complex $\Sigma^d$ that is balanced, 
with weights given by the multiplicities
of Definition~\ref{d:multiplicitydefn}.
\end{theorem}

\begin{proof}
Fix a $(d-1)$-dimensional cell $\sigma$ of $\Sigma^d$, and $\mathbf{w}$
in the relative interior of $\sigma$.  
 By Proposition~\ref{p:starweighted} the variety
 $V(\inn_{\mathbf{w}}(I))$ equals $\starr_{V(I)}(\mathbf{w})$ as
a weighted rational polyhedral fan.  Write $L=\spann(\sigma)$ for the
 $(d-1)$-dimensional lineality space of this fan.  After a monomial
 change of coordinates, which by
 Part~\ref{enum:automorphismdeg} of Proposition~\ref{p:basicdeg} 
 does not change degrees or multiplicities of zero-dimensional ideals, we
 may assume that $L$ is the span of
 $\mathbf{e}_1,\dots,\mathbf{e}_{d-1}$. 
Let $J= \inn_{\mathbf{w}}(I) \cap \mathbb
 B[x_d^{\pm 1},\dots,x_n^{\pm 1}]$.  By Theorem~\ref{t:projection} we
 have $V(J) = V(\inn_{\mathbf{w}}(I))/L$, so $V(J)$ is a
 one-dimensional rational polyhedral fan in $\mathbb R^{n-d+1}$.
 By Theorem~\ref{t:dimiscorrect} we have that $\dim(J)=1$ as well, so $V(J)$ is balanced by
 Lemma~\ref{l:curvesbalancing}.

Each ray in $V(J)$ is the quotient of a cone in
$\starr_{V(I)}(\mathbf{w})$ by the lineality space $L$.  To show that
$\Sigma^d$ is balanced at $\sigma$ it suffices to show that the
multiplicity on a ray in $V(J)$ equals the weight of the
corresponding cone in $\starr_{V(I)}(\mathbf{w})$.
After a monomial change of coordinates we may assume that the ray is
$\mathbb R_{\geq 0} \mathbf{e}_n$, so the corresponding cone $\tau$
is $\spann(\mathbf{e}_1,\dots,\mathbf{e}_d)+\mathbb R_{\geq 0}
\mathbf{e}_n$.  The multiplicity of $\RR_{\geq 0} \mathbf{e}_n$ in
$V(J)$ is the degree of the zero-dimensional ideal obtained by
intersecting $\inn_{\mathbf{e}_n}(J)$ with $\mathbb B[x_d^{\pm
    1},\dots,x_{n-1}^{\pm 1}]$.  The multiplicity of $\tau$ in
$\starr_{V(I)}(\sigma)= V(\inn_{\mathbf{w}}(I))$ is the degree of
$\inn_{\mathbf{e}_n}(\inn_{\mathbf{w}}(I)) \cap \mathbb B[x_d^{\pm
    1},\dots,x_{n-1}^{\pm 1}]$. 
The equality then follows from the fact that, since
$\inn_{\mathbf{w}}(I)$ is homogeneous with respect to the grading by
$\deg(x_i)=\mathbf{e}_i \in \mathbb Z^{d-1}$ for $1 \leq i \leq d-1$ and
$\deg(x_i)=\mathbf{0} \in \ZZ^{d-1}$ otherwise, we have 
$\inn_{\mathbf{e}_n}(J) =\inn_{\mathbf{e}_n}(\inn_{\mathbf{w}}(I))
\cap \mathbb B[x_d^{\pm 1}, \dots, x_n^{\pm 1}]$.
\end{proof}

\begin{remark} \label{r:HilbertChow} In \cite{TropicalIdeals} tropical
ideals were used to define subschemes of arbitrary toric varieties.
One consequence of Theorem~\ref{t:balanced} is the existence of a
{\defbold Hilbert-Chow morphism} that takes a subscheme of an
$n$-dimensional tropical toric variety $\trop(X_{\Sigma})$ given by a
locally tropical ideal (see \cite{TropicalIdeals}*{\S 4}) to a class
in $A^*(X_{\Sigma})$.  More specifically, if $I$ is a locally tropical ideal in $S:=
\Cox(X_{\Sigma})$, and $m$ is the product of the variables of
$S$, then $I'=(IS_{m})_{\mathbf{0}} \subseteq
(S_{m})_{\mathbf{0}} \cong \Rbar[y_1^{\pm 1},\dots,y_n^{\pm
1}]$ is a tropical ideal.  Set $d= \dim(I')$.  By
Theorem~\ref{t:balanced}, the union $\Delta$ of the $d$-dimensional
cones of $V(\varphi(I')) \subseteq \mathbb R^n$ forms a finite
$\mathbb R$-rational balanced polyhedral complex that is pure of
dimension $d$.  Then the techniques of \cite{FultonSturmfels} and
\cite{TropicalBook}*{\S 6.7} associate a class $[V(I)] \in
A_d(X_{\Sigma})$ to $\Delta$.
\end{remark}    

For an irreducible variety $X \subseteq (K^*)^n$ and fixed $a \in \mathbb R$, the
variety $\trop(X \cap V(x_n-\alpha))$ is the {\em stable intersection}
of $\trop(X)$ and the hyperplane $\trop(V(x_n-\alpha)) = \{ x_n=a \}$ 
for most choices of $\alpha \in K$ with
$\val(\alpha) = a$.  Here ``most'' means that given a fixed $t \in K$
with $\val(t)=a$, there is a finite set in the residue field $\K$ of
$K$ for which if the residue $\overline{\alpha/t} \in \K$ is not in
this set then the tropicalization has the given form; see
\cite{TropicalBook}*{Proposition 3.6.15}.
Remark~\ref{r:realizableslicing} suggests a connection of this fact to
the specialization construction, which we now explain.

\begin{proposition} \label{p:stableintersection}
Let $I \subseteq \Rbar[x_1^{\pm 1},\dots,x_n^{\pm 1}]$ be a tropical
ideal, and fix $a \in \mathbb R$. If $V(I)$ is the support of a
pure $d$-dimensional polyhedral complex, we have
$$V(I|_{x_n=a}) = \pi(V(I) \cap_{st} \{x_n =a \})$$
as weighted polyhedral complexes, 
where $\pi : \mathbb R^n \rightarrow \mathbb R^{n-1}$ is the projection onto
the first $n-1$ coordinates and $\cap_{st}$ denotes the stable intersection.
\end{proposition}

\begin{proof}
Fix a pure $\mathbb R$-rational polyhedral complex $\Sigma$ with
support $V(I)$.  By definition, the stable intersection of $\Sigma$ with
the hyperplane $H:=\{x_n=a \}$ is the
polyhedral complex $\Sigma \cap_{st} H$ that has a cell $\tau
\cap H$ for all cells $\tau \in \Sigma$ with $\dim(\tau+H)=n$, or
equivalently with $\tau$ not contained in $H$.

By Proposition~\ref{p:transverseintersection} we know that
$\pi(\Sigma \cap_{st} H) \subseteq V(I|_{x_n=a}) \subseteq \Sigma \cap H$.
Conversely, if $\mathbf{w} \in \Sigma \cap H$ but $\mathbf{w} \not
\in \Sigma \cap_{st} H$, then any $\tau \in \Sigma$ containing
$\mathbf{w}$ lies in $H$, so $\starr_{V(I)}(\mathbf{w})$ is contained
in the hyperplane $\{x_n=0 \}$.  By
Proposition~\ref{p:star} this means that $V(\inn_{\mathbf{w}}(I))$ is
contained in this hyperplane. Theorem~\ref{t:projection} then implies
that $\inn_{\mathbf{w}}(I) \cap \mathbb B[x_n^{\pm 1}] \neq \{ \infty
\}$.  This means that $\inn_{\mathbf{w}}(I)|_{x_n=0}$, which equals
$\inn_{\pi(\mathbf{w})}(I|_{x_n=a})$ by Part~\ref{item:initial} of
Lemma~\ref{l:basicfacts}, equals $\langle 0 \rangle$, and so
$\pi(\mathbf{w}) \not \in V(I|_{x_n=a})$ as required.  This shows that
$V(I|_{x_n=a}) = \pi(\Sigma \cap_{st} H)$ as a set.

We now prove that the multiplicities also coincide: for $\mathbf{w}
\in \mathbb R^n$ with $w_n=a$, we have $\mult_{V(I|_{x_n=a})}(\pi(\mathbf{w}))
= \mult_{\Sigma \cap_{st} H}(\mathbf{w})$.  We may assume that the
polyhedral complex $\Sigma$ has been chosen so that $\Sigma \cap H$ is
a subcomplex of $\Sigma$, so every cell $\sigma \in \Sigma \cap_{st}
H$ can also be viewed as a cell in $\Sigma$.  Maximal cells of the
stable intersection $\Sigma \cap_{st} H$ have dimension $d-1$.  By
definition, if $\mathbf{w}$ is in the relative interior of a maximal
cell $\sigma$ of $\Sigma \cap_{st} H$, the multiplicity of
$\mathbf{w}$ is $\sum_{\tau} \mult_{\Sigma}(\tau) [N : N_{\tau} +
  N_H]$, where $N_{\tau}$ and $N_H$ are the sublattices of $N =
\mathbb Z^n$ generated by the lattice points in $\spann(\tau)$ and $H$
respectively, $[N : N_{\tau}+N_H]$ is the lattice index, and the sum
is over all maximal cells $\tau$ in $\Sigma$ containing $\sigma$ with
$\tau \cap (\epsilon \mathbf{v}+H) \neq \emptyset$ for fixed generic
$\mathbf{v}$ and $0 < \epsilon \ll 1$.  Since $H$ is a coordinate hyperplane, we may take
$\mathbf{v} = \mathbf{e}_n$.  As
both sides of the equality in the proposition statement are invariant
under changes of coordinates in $x_1,\dots,x_{n-1}$ that fix $x_n$, we may assume that $\spann(\sigma) = \spann(\mathbf{e}_1,\dots,\mathbf{e}_{d-1})$.

For $\tau$ a maximal cell in $\Sigma$ containing $\sigma$, 
let $\mathbf{v}_{\tau} \in N$ be a representative for a generator
of the one-dimensional lattice $N_{\tau}/N_{\sigma}$.  The
sublattice $N_{\tau}+N_H$ is generated by 
$\{ \mathbf{e}_1,\dots,\mathbf{e}_{n-1} \}$ and $\mathbf{v}_{\tau}$, so
the index $[N : N_{\tau} + N_H]$ is up to sign equal to the last coordinate
$(\mathbf{v}_{\tau})_n$.  Thus the multiplicity of $\mathbf{w}$ in
$\Sigma \cap_{st} H$ equals
\begin{equation} \label{eqtn:stableintmult}
\textstyle
\sum_{\tau} (\mathbf{v}_{\tau})_n \mult_{\Sigma}(\tau),
\end{equation}
where the sum is over all maximal cells $\tau$ containing $\sigma$ with $\tau \cap
(\epsilon \mathbf{e}_n + H) \neq \emptyset$, which are precisely those $\tau$
with $(\mathbf{v}_{\tau})_n>0$.  The multiplicities
$\mult_{\Sigma}(\tau)$ are the same as the multiplicities of the
corresponding cones in $\starr_{\Sigma}(\sigma) =
V(\inn_{\mathbf{w}}(I))$.
Grade $S=\mathbb B[x_1^{\pm 1},\dots,x_n^{\pm 1}]$ by $\spann(\sigma)$.
We have $S_{\mathbf{0}} = \mathbb B[x_{d}^{\pm 1},\dots,x_n^{\pm 1}]$.  The
variety of the tropical ideal $\inn_{\mathbf{w}}(I) \cap
S_{\mathbf{0}}$ is a one-dimensional fan in $\mathbb R^{n-d+1}$ which
is isomorphic to $\starr_{\Sigma}(\sigma)/\spann(\sigma)$ by
Theorem~\ref{t:projection}.  

Equation~\eqref{eqtn:stableintmult} is then the right-hand side of
\eqref{eqtn:degJ1mult} applied to the ideal $\inn_{\mathbf{w}}(I)
\cap S_{\mathbf{0}}$, so the multiplicity of $\mathbf{w}$ in $\Sigma
\cap_{st} H$ equals $\deg((\inn_{\mathbf{w}}(I) \cap
S_{\mathbf{0}})|_{x_n=1})$.  
This degree is equal to $\deg((\inn_{\mathbf{w}}(I) \cap
S_{\mathbf{0}})|_{x_n=0})$ by Part~\ref{item:trivialization} of Lemma~\ref{l:basicfacts}, 
since the degree of a tropical ideal depends only on its trivialization.
Now, we have 
$$(\inn_{\mathbf{w}}(I) \cap S_{\mathbf{0}})|_{x_n=0} \,=\, \inn_{\mathbf{w}}(I)|_{x_n=0} \cap S_{\mathbf{0}}|_{x_n=0} \,=\, \inn_{\pi(\mathbf{w})}(I|_{x_n=a}) \cap \BB[x_d^{\pm 1},\dots,x_{n-1}^{\pm 1}],$$
where the first equality follows from the fact that $\deg(x_n)=\mathbf{0}$, 
and the second from Part~\ref{item:initial} of Lemma~\ref{l:basicfacts}.
Thus the multiplicity of $\mathbf{w}$ in $\Sigma \cap_{st} H$ equals
$\deg(\inn_{\pi(\mathbf{w})}(I|_{x_n=a}) \cap \BB[x_d^{\pm 1},\dots,x_{n-1}^{\pm 1}])$,
which is by definition the multiplicity of $\pi(\mathbf{w})$ in $V(I|_{x_n=a})$.
\end{proof}

Stable intersection is only defined for balanced polyhedral complexes, to ensure 
the resulting multiplicities are independent of  choices. 
The reason we require $V(I)$ to be pure in the previous theorem is because
Theorem~\ref{t:balanced} only guarantees that the maximal-dimensional
subcomplex of $V(I)$ is balanced, and thus the stable intersection 
can only be defined when that is the entire complex.

In the case that $I = \trop(J)$ for a {\em prime} ideal $J \subseteq
K[x_1^{\pm 1},\dots,x_n^{\pm 1}]$, the Structure Theorem for tropical
geometry implies that $V(I)$ is the support of a pure-dimensional
polyhedral complex.  It would be good to have a more general condition
on tropical ideals that still guarantees this.  If $J$ is not prime,
it still has a primary decomposition, so $V(I)$ decomposes as the
support of a union of balanced polyhedral complexes.
Proposition~\ref{p:stableintersection} thus extends to these cases.
An analogous construction is currently missing from tropical scheme
theory.

\end{document}